\documentclass[a4paper]{amsart}
\calclayout
\numberwithin{equation}{section}

\usepackage[latin1]{inputenc}

\usepackage{standalone}
\usepackage{amssymb}
\usepackage{array}
\usepackage{makecell}
\usepackage{amsmath}
\usepackage{mathrsfs}
\usepackage{amsthm}
\usepackage{amsfonts}
\usepackage{textcomp}
\usepackage{graphicx}
\usepackage{cellspace}
\usepackage{relsize}
\usepackage[pdftex]{color}
\usepackage{paralist}
\usepackage[shortlabels]{enumitem}
\usepackage{hyperref}
\usepackage{comment}
\usepackage[arrow, matrix, curve]{xy}

\usepackage{cleveref}
\usepackage{tikz}
\usetikzlibrary{decorations.pathreplacing,calligraphy, calc, arrows.meta, shapes.misc, decorations.pathmorphing}
\usepackage{tikz-cd}

\definecolor{jazzberryjam}{rgb}{0.65, 0.04, 0.37}
\definecolor{mediumred-violet}{rgb}{0.73, 0.2, 0.52}

\hypersetup{
    colorlinks = true,
    citecolor=cyan,
    linkcolor=cyan
}

\newtheorem*{corollary*}{Corollary}
\newtheorem{theorem}{Theorem}[section]
\newtheorem*{theorem*}{Theorem}

\newtheorem{corollary}[theorem]{Corollary}
\newtheorem{lemma}[theorem]{Lemma}

\newtheorem*{claim*}{Claim}

\theoremstyle{definition}
\newtheorem{definition}[theorem]{Definition}

\newtheorem*{theorem }{Theorem}
\newtheorem{remark}[theorem]{Remark}
\newtheorem{example}[theorem]{Example}

\theoremstyle{remark}

\numberwithin{equation}{theorem}


\renewcommand{\mod}{\operatorname{mod}}

\newcommand{\Ext}{\operatorname{Ext}}
\newcommand{\im}{\operatorname{Im}}

\newcommand{\grade}{\operatorname{grade}}
\newcommand{\cograde}{\operatorname{cograde}}

\renewcommand{\top}{\operatorname{\mathrm{top}}}
\newcommand{\rad}{\operatorname{\mathrm{rad}}}
\newcommand{\soc}{\operatorname{\mathrm{soc}}}
\newcommand{\coker}{\mathrm{Coker}}
\renewcommand{\ker}{\mathrm{Ker}}
\newcommand{\idim}{\operatorname{idim}}
\newcommand{\pdim}{\operatorname{pdim}}
\newcommand{\gldim}{\operatorname{gldim}}

\tikzset{cross/.style={cross out, draw=black, minimum size=2*(#1-\pgflinewidth), inner sep=0pt, outer sep=0pt},
cross/.default={6pt}}

\tikzset{every picture/.style={line width=0.7pt, font=\Large}}

\title{The Auslander-Gorenstein condition for monomial algebras}
\date{\today}
\author[Vikt\'oria Kl\'asz]{Vikt\'oria Kl\'asz}%
\address[Vikt\'oria~Kl\'asz]{Mathematical Institute of the University of Bonn, Endenicher Allee 60, 53115 Bonn, Germany}%
\email{klasz@math.uni-bonn.de}%


\begin{document}

\begin{abstract}

This paper investigates the Auslander-Gorenstein property for monomial algebras. First, we prove that every Auslander-Gorenstein monomial algebra is a string algebra and present a simple combinatorial classification of Auslander-Gorenstein gentle algebras.
Furthermore, we describe a procedure to transform any 2-Gorenstein monomial algebra into a Nakayama algebra, thereby reducing the classification of Auslander-Gorenstein monomial algebras to that of Auslander-Gorenstein Nakayama algebras. As an application of this reduction method, we prove that every monomial algebra satisfies a stronger version of the Auslander-Reiten Conjecture.

Our second main result establishes that a monomial algebra is Auslander-Gorenstein if and only if it has a well-defined, bijective Auslander-Reiten map, confirming a conjecture of Marczinzik for monomial algebras. This yields a new homological characterisation of the Auslander-Gorenstein property. Additionally, we provide an explicit description of the Auslander-Reiten bijection in the case of gentle algebras. Along the way, we also generalise a result of Iwanaga and Fuller: We show that every $2n$-Gorenstein monomial algebra is also $(2n+1)$-Gorenstein for every $n\ge 1.$

\end{abstract}


\maketitle
\vspace{-0.2cm}
{
  \hypersetup{linkcolor=black}
  \tableofcontents
}
\vspace{-2cm}
\newpage

\section{Introduction}
\label{chap::introduction}

Auslander-Gorenstein algebras were introduced by Auslander as non-commutative analogues of Gorenstein rings from commutative algebra \cite[Section 3]{FGR75}. 
Let $A$ be a two-sided Noetherian ring, and consider a minimal injective resolution of the regular representation $A_A$: 
$$0\longrightarrow A_A \longrightarrow I^0 \longrightarrow I^1 \longrightarrow \ldots. $$
We say that $A$ is \textcolor{blue}{$n$-Gorenstein} for some $n\ge 1$ if $\text{flat}\dim I^i\le i$ for every $0\le i< n$. If $A$ is $n$-Gorenstein for all $n\ge1$ and $\idim A_A<\infty$, then $A$ is called \textcolor{blue}{Auslander-Gorenstein}.
Prominent examples of Auslander-Gorenstein algebras include Weyl algebras, universal enveloping algebras of finite-dimensional Lie algebras \cite{VO00}, higher Auslander algebras \cite{I07}, and blocks of the BGG category $\mathcal{O}$ \cite{KMM21}.

In \cite{AR94}, Auslander and Reiten initiated the study of these properties in the context of Artin algebras, where the Auslander-Gorenstein condition ensures remarkable homological properties for an algebra.
For instance, syzygy categories of Auslander-Gorenstein algebras are functorially finite, extension-closed, and explicit descriptions of minimal approximations exist in these categories.

The Auslander-Gorenstein condition is also closely related to classical homological conjectures. If $A$ is $n$-Gorenstein for all $n$ then $\idim {}_AA<\infty$ if and only if $\idim A_A<\infty$ \cite[Corollary 5.5]{AR94}. This means that these algebras satisfy the Gorenstein Symmetry Conjecture. It remains an open question, posed in the same paper, whether the condition of being $n$-Gorenstein for all $n$ already implies that the algebra is Auslander-Gorenstein. This is known today as the Auslander-Reiten Conjecture. A positive answer to this would imply the Nakayama Conjecture and several other homological conjectures.

The main purpose of this article is to gain a better understanding of the $n$-Gorenstein and Auslander-Gorenstein properties for certain well-known classes of finite-dimensional algebras. For the remainder of this paper, we assume that $A=kQ/I$ is a finite-dimensional algebra given by a finite quiver $Q$ and an admissible ideal $I$. (For an algebraically closed field $k,$ every finite-dimensional $k$-algebra is Morita equivalent to an algebra of the form $kQ/I.$) In this context, the flat dimension of a finite-dimensional module agrees with its projective dimension.

Deciding if an algebra $kQ/I$ is Auslander-Gorenstein seems to be a rather difficult task. This becomes more accessible in case the ideal $I$ is defined in combinatorial terms. For example, in \cite{IM22}, Iyama and Marczinzik showed that the incidence algebra $kL$ of a finite lattice $L$ is Auslander-Gorenstein if and only if $L$ is a distributive, thereby translating the homological property in question into a lattice-theoretic one. 
In this paper, we restrict our attention to three well-studied classes: gentle algebras (\Cref{chap::gentle}), Nakayama algebras (\Cref{chap::Nakayama}), and general monomial algebras (\Cref{chap::monomial} \& \ref{chap::MainResultsForMonomial}).
For these, we aim to provide explicit combinatorial criteria in terms of quivers and relations that determine whether the algebra is $n$-Gorenstein or Auslander-Gorenstein. 

The first two of these classes have a strong combinatorial structure associated with them. Since gentle algebras are string algebras, all their indecomposable modules can be expressed via strings and bands \cite{BR87}. Linear Nakayama algebras, on the other hand, correspond to Dyck paths, providing a powerful combinatorial framework for their representation theory, see e.g. \cite{MRS21}, \cite{RS17}, or \cite{KMMRS25}. The class of monomial algebras (also known as zero-relation algebras) generalises both gentle and Nakayama algebras. They are defined by the condition that the admissible ideal $I\subset kQ$ is generated by a set of paths, rather than linear combinations of them. Thus, monomial algebras form a much broader class than gentle or Nakayama algebras, but the restriction on $I$ still makes them relatively accessible to study.

\subsection{Main results}
In \cite{AR94}, Auslander and Reiten observed that every Auslander-Gorenstein algebra comes with a distinguished bijection between the isomorphism classes of its indecomposable injective and indecomposable projective modules. 
We say that a finite-dimensional $k$-algebra $A$ has a \textcolor{blue}{well-defined} Auslander-Reiten map if every indecomposable injective $A$-module $I$ has a finite minimal projective resolution where the last non-zero term is indecomposable. In this case, the \textcolor{blue}{Auslander-Reiten map, $\psi_A$} of $A$ is defined as 

$$\psi_A: \{I\text{ indec. injective }\}/_{\cong} \longrightarrow \{P\text{ indec. projective }\}/_{\cong}$$ 
$$I\mapsto \Omega^{\pdim I}(I).$$ 

\begin{theorem*}\cite[Proposition 5.4]{AR94} Let $A$ be an Auslander-Gorenstein algebra. Then $\psi_A$
is well-defined and bijective. Moreover, $\pdim I=\idim \psi_A(I)$. 
\\
In this case, we call $\psi_A$ the \textcolor{blue}{Auslander-Reiten bijection.} 
\end{theorem*}

This bijection can be naturally interpreted as a permutation of the simple $A$-modules, or equivalently, a permutation on the vertices of the underlying quiver.

Note that the well-definedness of $\psi_A$  is already a remarkable property of Auslander-Gorenstein algebras. Demanding that the last term of a projective resolution is indecomposable is a strong requirement, since usually, syzygy modules tend to get more complicated as we consider higher terms in the resolution. On the other hand, even when the well-definedness of $\psi_A$ holds, bijectivity remains a non-trivial condition. For instance, while the Auslander-Reiten map is well-defined for every linear Nakayama algebra, it is seldom bijective. This motivates the following conjecture.
\\
\\
\hypertarget{conjectureM}{\textbf{Conjecture.}} (Marczinzik, \cite{M23}): Let $A$ be a finite-dimensional algebra over an algebraically closed field. Then $A$ is Auslander-Gorenstein if and only if it has a well-defined Auslander-Reiten map that is a bijection.
\\

The Auslander-Reiten bijection appears to play a central role in the theory of Auslander-Gorenstein algebras. In a recent paper \cite{KMT25}, it was shown that this bijection coincides with Iyama's grade bijection, a homological permutation on the simple modules introduced in \cite{I03}. As a consequence, it has the property that $\grade(\soc I)=\cograde(\top \psi_A( I))$. In \cite{IM22}, it was discovered that for incidence algebras of distributive lattices, the Auslander-Reiten permutation coincides with the rowmotion map on the lattice, a central object in dynamical algebraic combinatorics, see e.g. \cite{S17},\cite{TW19}, and \cite{H24}. In the context of higher Auslander algebras, the Auslander-Reiten map corresponds to the higher Auslander-Reiten translate, as shown in \cite{MTY24}.
A new, linear algebraic reinterpretation of this bijection has also been proposed in \cite{KMT25}, where it is described in terms of Bruhat decompositions of Coxeter matrices. This approach opens the door to extending the definition of the Auslander-Reiten bijection to non-Auslander-Gorenstein algebras with finite global dimension. 

As a main result of this article, we prove Marczinzik's conjecture for the class of monomial algebras.

\begin{theorem*}[\Cref{thm::AGiffARbijwelldef_Monomial}]
     For a monomial algebra $A$ the following statements are equivalent:
    \begin{enumerate}[\rm(i)]
        \item $A$ is Auslander-Gorenstein.
        \item $A$ has a well-defined Auslander-Reiten map that is a bijection. 
    \end{enumerate}
\end{theorem*} 

The other main achievement of this paper is that it provides a deep insight about the structure of Auslander-Gorenstein monomial algebras.

\begin{theorem*}[\Cref{cor::2GormonomialImpliesString}]
    Every Auslander-Gorenstein monomial algebra is a string algebra.
\end{theorem*} 

For the important subclass of gentle algebras, we obtain a full classification of the Auslander-Gorenstein property. This gives a simple combinatorial characterisation of the homological condition in question.

\begin{theorem*}[\Cref{gentleCharacterisation}]
Let $A=kQ/I$ be a gentle algebra. Then the following are equivalent:
\begin{enumerate}[\rm (i)]
    \item $A$ is Auslander-Gorenstein.
    \item For every vertex $v\in Q_0$ we have $\deg^{\text{in}}(v)=2$ if and only if $\deg^{\text{out}}(v)=2$.
\end{enumerate}
\end{theorem*}

An explicit characterisation in terms of quivers and relations is also obtained for the 2-Gorenstein property of monomial algebras, see \Cref{thm::classification2GorMonomial}. As a direct consequence of this, we obtain the main result of \cite{M19}, which states that every monomial algebra with dominant dimension at least two is a Nakayama algebra (\Cref{cor::2GorMonAlgDomdimAtLeast2}).

Our classification result shows that 2-Gorenstein monomial algebras are closely related to gentle and Nakayama algebras. More precisely, such an algebra can be viewed as a collection of Nakayama algebras glued together at degree 4 vertices in a specific way. 
Thus, 
any 2-Gorenstein monomial algebra can be reduced to a Nakayama algebra by "cutting" the algebra at degree 4 vertices. Remarkably, several important homological properties of the algebra are preserved under this cutting process (see \Cref{thm::cutting}). As a key consequence, the classification problem of Auslander-Gorenstein monomial algebras reduces to that of Auslander-Gorenstein Nakayama algebras. 

\begin{theorem*}[\Cref{cor::redofQ1}]
      We can reduce the problem of classifying all Auslander-Gorenstein monomial algebras to classifying all Auslander-Gorenstein Nakayama algebras.
\end{theorem*}

This reduction is significant because Nakayama algebras form a much smaller and more tractable subclass of monomial algebras, often allowing for explicit computations. Classifying all Auslander-Gorenstein Nakayama algebras, however, turns out to be a surprisingly difficult task.

The most important application of the reduction process described above is Marczinzik's conjecture. We only prove this explicitly for Nakayama algebras and the reduction method allows us to extend it to monomial algebras. 
Similarly, we use this reduction process to  generalise a theorem of Fuller and Iwanaga from \cite{FI93}, which states that if a Nakayama algebra is $2$-Gorenstein, then it is also $3$-Gorenstein.

\begin{theorem*}[\Cref{cor::2kimplies2k+1_monomial}]
      If a monomial algebra is $2n$-Gorenstein then it is also $(2n+1)$-Gorenstein. This holds for every $n\ge 1.$
\end{theorem*}

As a final application of the reduction technique, we prove that monomial algebras satisfy a stronger version of the Auslander-Reiten Conjecture. For this, note that an algebra $A$ is called Iwanaga-Gorenstein if $\idim A_A= \idim {}_AA < \infty.$ We say that $A$ is \textcolor{blue}{$d$-Iwanaga-Gorenstein} if $\idim A_A= \idim {}_AA\le d.$

\begin{theorem*}[\Cref{cor::strongerVersionofARConj}]
    Let $A$ be a finite-dimensional monomial algebra with $n$ simple modules. If $A$ is $(4n-2)$-Gorenstein, then $A$ is $(4n-2)$-Iwanaga-Gorenstein.
\end{theorem*}

Based on this evidence, we may pose the following question:
\\
\\
\textbf{Question.} Let $A$ be a finite-dimensional $k$-algebra with $n$ simple modules. Assume that $A$ is $(4n-2)$-Gorenstein. Does this imply that $A$ is also Iwanaga-Gorenstein?

\subsection{Organisation of the paper}{
In Section \ref{chap::preliminaries}, we start by recalling the necessary preliminaries on Auslander-Gorenstein algebras. Section \ref{chap::gentle} discusses the case of gentle algebras: It contains the combinatorial characterisation of the Auslander-Gorenstein property (\Cref{gentleCharacterisation}) and the proof of \hyperlink{conjectureM}{Marczinzik's conjecture} (\Cref{thm::AGiffARbijwelldefGentle}) for this class. In \Cref{chap::monomial}, we turn our attention to general monomial algebras, and provide a classification of 2-Gorenstein monomial algebras in terms of quivers and relations (\Cref{thm::classification2GorMonomial}). Building on this result, \Cref{chap::reduction} introduces the procedure to reduce 2-Gorenstein monomial algebras to Nakayama algebras, while preserving their relevant homological properties. \Cref{chap::Nakayama} focuses on the class of Nakayama algebras, it proves \hyperlink{conjectureM}{Marczinzik's conjecture} in this setting (\Cref{thm::AGiffARbijwelldef_Nak}) and generalises a result of Fuller and Iwanaga (\Cref{2kImplies2k+1}). Finally, in Section \ref{chap::MainResultsForMonomial}, we apply the reduction procedure described in \Cref{chap::reduction} to extend our results from Nakayama algebras to monomial algebras. In particular, \Cref{thm::AGiffARbijwelldef_Monomial} extends \hyperlink{conjectureM}{Marczinzik's conjecture}, and \Cref{cor::2kimplies2k+1_monomial} extends Fuller and Iwanaga's result. As a final application of this reduction method, Section \ref{subsec::strongAuslanderReitenConj} proves a stronger version of the Auslander-Reiten Conjecture for monomial algebras.

}

\subsection{Notation}{In this paper, we are interested in finite-dimensional algebras over a field $k.$ These will always be given as the path algebra of a finite quiver $Q$ modulo an admissible ideal $I$. We use the notation $A=kQ/I$ for this algebra, and call such algebras \textcolor{blue}{quiver algebras}. Whenever we talk about modules, we always mean finite-dimensional right $A$-modules, unless otherwise stated, and $\mod(A)$ denotes the corresponding module category. Recall that indecomposable projective, injective and simple modules over $A$ are enumerated by the vertices of $Q.$ The modules corresponding to a vertex $v$ are denoted by $P(v)=e_vA, I(v)=D(Ae_v),$ and $S_v$, respectively. 

}

\newpage

\section{Preliminaries}
\label{chap::preliminaries}
Let $A$ be a finite-dimensional $k$-algebra. Throughout this paper, we assume that the reader is familiar with the basics of homological algebra and the representation theory of finite-dimensional algebras, see for example \cite{ASS06} for a reference.

\subsection{Notation related to quivers}
\label{subsec::quiverNotation}

For a quiver $Q$ we use the notation $Q=(Q_0,Q_1,s,t)$, where $Q_0$ denotes the set of vertices and $Q_1$ the set of arrows of $Q$, $s$ and $t$ are two functions from $Q_1$ to $Q_0$ which assign to an arrow its start (resp. end) vertex. For a vertex $v\in Q_0$ we denote its in-degree by $\textcolor{blue}{\deg^{\text{in}}(v)}$ and its out-degree by $\textcolor{blue}{\deg^{\text{out}}(v)}$. The degree of this vertex is then $\textcolor{blue}{\deg(v)}=\deg^{\text{in}}(v)+\deg^{\text{out}}(v).$ 

In this paper, we compose arrows from left to right, i.e. we denote the composition of two arrows $x\xrightarrow[]{\alpha}y\xrightarrow[]{\beta}z$ by $\alpha\beta.$

Next, we introduce some less standard language and notation related to paths that are special to this paper. Let $p$ and $q$ be paths in $Q$. We say that $q$ is a \textcolor{blue}{left} (resp. \textcolor{blue}{right}) \textcolor{blue}{subpath} of $p$ if there is a path $r$ such that $p=qr$ (resp. $p=rq.$) A subpath is called  \textcolor{blue}{strict} if length$(r)\neq 0.$ Our convention is that the \textcolor{blue}{length} of a path is equal to the number of arrows it contains. So the path $e_x$ corresponding to a vertex $x$ has length $0.$ 

Let $I$ be an admissible ideal of the path algebra $kQ$. Then a path $p\notin I$ is said to be \textcolor{blue}{left-maximal} (resp. \textcolor{blue}{right-maximal}) if there is no arrow $a\in Q_1$ such that $ap\notin I$ (resp. $pa\notin I$). If a path is both left- and right-maximal, then we call it \textcolor{blue}{maximal}. 
On the other hand, we say that a path $p$ is a \textcolor{blue}{minimal relation} or a minimal path in $I$, if $p\in I$ but no strict right or left subpath of $p$ is contained in $I.$ For a path $p$ of non-zero length, $\textcolor{blue}{p^+}$ (resp. $\textcolor{blue}{p^-}$) is going to denote the path we obtain from $p$ by omitting its last (resp. first) arrow. 

Finally, assume that $I$ is generated by a set of paths, i.e. $A=kQ/I$ is a monomial algebra. Then, given a path $p\notin I$, there is a canonical uniserial $A$-module associated to this path. We denote this by $\textcolor{blue}{M(p)}.$

\subsection{Auslander-Gorenstein algebras}
\label{subsec::prliminariesAG}

\begin{definition}
\label{def::Auslander-Gorenstein}
Consider a minimal injective resolution of the regular representation $A_A$:
    $$0\longrightarrow A_A \longrightarrow I^0 \longrightarrow I^1 \longrightarrow \ldots. $$
Then $A$ is called 
\begin{enumerate}[\rm (1)]
    \item \textcolor{blue}{n-Gorenstein} for an $n\ge 1$ if $\pdim I^i\le i$ for every $0\le i<n.$
    \item \textcolor{blue}{Auslander-Gorenstein} if it is $n$-Gorenstein for all $n$ and $\idim A_A<\infty.$
    \item \textcolor{blue}{Auslander regular} if it is $n$-Gorenstein for all $n$ and $\gldim A<\infty.$
    
\end{enumerate}
\end{definition}

We list the most important properties of these algebras which will be needed in our proofs.

\begin{theorem}\cite[Corollary 5.5]{AR94}
\label{thm::preliminaries::IG}
Let $A$ be Auslander-Gorenstein. Then $A$ is Iwanaga-Gorenstein, i.e. $\idim A_A=\idim {}_AA< \infty.$
\end{theorem}

Note that sometimes in the literature, $n$-Iwanaga-Gorenstein algebras (i.e. algebras with $\idim A_A=\idim {}_AA\le n$) are also called $n$-Gorenstein. However, these two notions do not coincide, and we will always be referring to the one in \Cref{def::Auslander-Gorenstein}.

\begin{theorem} \cite[Auslander's theorem 3.7]{FGR75}
\label{thm::preliminaries:left-right-symmetry}
    An algebra $A$ is n-Gorenstein if and only if $A^{op}$ is.
    Thanks to \Cref{thm::preliminaries::IG}, $A$ is Auslander-Gorenstein if and only if $A^{op}$ is.
\end{theorem}

Thus, the $n$-Gorenstein property can also be checked using minimal projective resolutions of indecomposable injective modules. Explicitly, $A$ is $n$-Gorenstein if and only if  for every indecomposable injective $A$-module $I$ we have $\idim P_i\le i$ for every $0\le i <n$, where 
$$\ldots \longrightarrow P_1\longrightarrow P_0 \longrightarrow I \longrightarrow 0$$
is a minimal projective resolution of $I$.

\newpage

\begin{theorem}\cite[Proposition 5.4 and Corollary 5.5]{AR94} 
\label{thm::preliminaries::ARbij}
\begin{enumerate}[\rm 1.]
    \item Let $A$ be an $n$-Gorenstein algebra. Then for any $d\le n$
$$\psi_d: \{I \text{ indec. injective}\ | \ \pdim I=d\}/_{\cong} \longrightarrow \{P\text{ indec. projective} \ | \ \idim P=d\}/_{\cong}$$ 
$$I\mapsto \psi_d(I)=\Omega^{d}({I})$$
defines a bijection. The inverse map is given by $P \mapsto \Sigma^{d}({P}).$

\item Let $A$ be an  Auslander-Gorenstein algebra. Then we can combine the above maps to obtain a bijection 
$$\psi_A: \{I\text{ indec. injective}\}/_{\cong} \longrightarrow \{P\text{ indec. projective}\}/_{\cong}$$ 
$$I\mapsto \psi_A(I)=\Omega^{\pdim I}({I}).$$
We have $\idim \psi_A(I)=\pdim I$, and the inverse map is given by $P \mapsto \Sigma^{\idim P}({P}).$
This is called the \textcolor{blue}{Auslander-Reiten bijection.}

\end{enumerate}

\end{theorem}

Motivated by this, we say that a finite-dimensional $k$-algebra $A$ \textcolor{blue}{has a well-defined Auslander-Reiten map}, if every indecomposable injective $A$-module $I$ has a finite minimal projective resolution in which the last non-zero term is indecomposable. In this case, the \textcolor{blue}{Auslander-Reiten map, $\psi_A$} of $A$ is defined as 

$$\psi_A: \{I\text{ indec. injective }\}/_{\cong} \longrightarrow \{P\text{ indec. projective }\}/_{\cong}$$ 
$$I\mapsto \Omega^{\pdim I}({I}).$$ 

This naturally induces a map on the simple $A$-modules, or in case $A=kQ/I$, on the vertices of the underlying quiver. We abuse notation and write $\psi_A$ to denote this map as well. Explicitly,  $\psi_A: Q_0 \mapsto Q_0$ is defined as $\textcolor{blue}{\psi_A(x)}=y$ if $\psi_A(I(x))=P(y)$ for $x,y\in Q_0.$

\newpage
\section{Gentle algebras and the Auslander-Gorenstein property}
\label{chap::gentle}

In this chapter, we examine the Auslander-Gorenstein property in the context of gentle algebras.  Introduced by Assem and Skowro\'nski in their 1987 paper \cite{AS87}, gentle algebras emerged from tilting theory as a means to classify iterated tilted algebras of path algebras of type $\Tilde{A}$. Since then, they have gained a lot of interest and as of today, we have a very good understanding of their module category and their derived category. Gentle algebras belong to the class of string algebras, therefore, their indecomposable modules are all string or band modules, we know how irreducible morphisms look in the module category, and can describe an explicit basis of the homomorphism space 
between any two indecomposable string modules \cite{BR87}, \cite{CB89}. This explicit and combinatorial description makes them an ideal class for investigating homological properties and for testing related conjectures.
Gentle algebras were found to have many remarkable properties. For instance, the class of gentle algebras is closed under derived equivalences \cite{SZ03}, their derived category is tame \cite{BM03}, and they are Iwanaga-Gorenstein \cite{GR05}. This last result makes them natural candidates for investigating the Auslander-Gorenstein property, as it suffices to focus on the $n$-Gorenstein conditions.

In the following, we successfully answer all questions raised in \Cref{chap::introduction} for the class of gentle algebras. 
First, in \Cref{gentleCharacterisation} we express the Auslander-Gorenstein property as a completely combinatorial condition on the underlying quiver. Next, we show in \Cref{thm::AGiffARbijwelldefGentle} that Marczinzik's conjecture is true for every gentle algebra, so the Auslander-Reiten bijection can be used to characterise the Auslander-Gorenstein property in this case. Finally,  we give a combinatorial description of this bijection in terms of $Q$ and $I$ in \Cref{cor::descriptionARbijGentle}.

\subsection{Definitions for gentle algebras}
\label{subsec::definitionsGentle}
Let us begin by recalling the definition of gentle algebras and fixing some notation specific to these algebras. 

\begin{definition}
\label{def::gentleAlgebras}
Let $A=kQ/I$ be a finite-dimensional quiver algebra. Then $A$ is called \textcolor{blue}{gentle} if it satisfies the following conditions:

\begin{enumerate}[\rm (i)]
\item $Q$ is biserial, i.e. there are at most 2 incoming and at most 2 outgoing arrows at every vertex

\item for any $a_1, a_2, b \in Q_1$ with $a_1\neq a_2$ and $t(a_1) = t(a_2) = s(b),$ we have
$|\{a_1b, a_2b \} \cap I| = 1$

\item for any $a, b_1, b_2 \in Q_1$ with $b_1\neq b_2$ and $t(a) = s(b_1) = s(b_2)$, we have
$|\{ab_1, ab_2 \} \cap I| = 1.$

\item $I$ is generated by a set of paths of length 2.
\end{enumerate}

\end{definition}

We will use the notation from \cite{GR05}. Let $A=KQ/I$ be a gentle algebra. We say that a path $p=a_1a_2\ldots a_m$ in $Q$, where the $a_i$ are arrows, is \textcolor{blue}{critical} if $a_ia_{i+1}\in I$ for every $1\le i\le m-1.$ Let $a$ be an arrow. Then there is a unique path $\textcolor{blue}{p(a)}$ in $Q$ such that $ap(a)\notin I$ is right-maximal. Note that this path is non-zero, but $p(a)=e_{t(a)}$ is possible. Dually, there is a unique path $\textcolor{blue}{i(a)}$ such that $i(a)a\notin I$ is left-maximal. These can be used to describe the indecomposable projective and injective modules, which are string modules of a particularly simple form. 

Let $x\in Q_0.$ Then by \Cref{def::gentleAlgebras}, there are either one or two right-maximal paths starting at $x$. In the first case, $P(x)=e_xA\cong M(p)$ is uniserial, and corresponds to the unique right-maximal path $p$ starting at $x.$ If $p$ has length at least one, then $p=b p(b)$ for the unique arrow $b$ with $s(b)=x.$
In the second case, the two right-maximal paths $p_1$ and $p_2$ with $s(p_1)=s(p_2)=x$ must start with two different arrows, $b_1$ and $b_2$. Then $P(x)= e_xA$ has a simple top isomorphic to $S_x$ and its radical is a direct sum of two uniserial modules: $\rad P(x) \cong M(p(b_1))\oplus M(p(b_2)).$ Dually, $I(x)=D(Ae_x)$ is either uniserial, or there are two different arrows $a_1$ and $a_2$ ending at $x$, in which case $I(x)$ has a simple socle and $I(x)/\soc I(x) \cong M(i(a_1))\oplus M(i(a_2))$. Every indecomposable projective-injective module is uniserial.

\begin{example} 
\label{ex::section3.1}
Consider the algebra $A=kQ/\langle ab,c^2 \rangle$ where $Q$ is the quiver

{\centering
    \begin{tikzpicture}[font=\normalsize]
    
    \node (A) at (-0.3, 0) {$1$};
    \node (B) at (1, 0) {$2$};
    \node (C) at (2.3, 0) {$3$};
    
    \node (AB) at (0.3, -0.2) {$a$};
    \node (BC) at (1.7, -0.2) {$b$};

    \draw[->] (A) to (B){};
    \draw[->] (B) to (C){};
  
    \draw[->, loop] (B) to (B) node[xshift=0.5 cm,yshift=1cm] {$c$};


\end{tikzpicture}

}
This defines a gentle algebra. The indecomposable projective and injective modules are:

 \vspace{10 pt}{\centering

\resizebox{\linewidth}{!}{
    \input{pictures/P2_exampleSection3}
    }
}
 \vspace{5 pt}

So the uniserial ones are: 
$$P(1)\cong M(ap(a))=M(acb)=M(i(b)b)\cong I(3),\ P(3)\cong S_3,\ I(1)\cong S_1.$$

For the non-uniserial ones we have: $$\rad P(2)\cong M(p(c))\oplus M(p(b))=M(b)\oplus M(e_3)\text{ and  }I(2)/\soc I(2)\cong M(i(a))\oplus M(i(c))=M(e_1)\oplus M(a).$$

\end{example}

\vspace{10pt}

\subsection{Characterisation of gentle Auslander-Gorenstein algebras} 

Our first main result establishes that for gentle algebras the Auslander-Gorenstein property can be translated to a simple condition on the underlying quiver.

\begin{theorem}
\label{gentleCharacterisation}
Let $A=kQ/I$ be a gentle algebra. Then the following are equivalent:
\begin{enumerate}[\rm (i)]
    \item $A$ is Auslander-Gorenstein.
    \item For every vertex $v\in Q_0$ $\deg^{\text{in}}(v)=2$ if and only if $\deg^{\text{out}}(v)=2$.
\end{enumerate}
\end{theorem}

Observe that condition (ii) can be equivalently phrased as the quiver $Q$ not containing vertices of the following four types: 

 
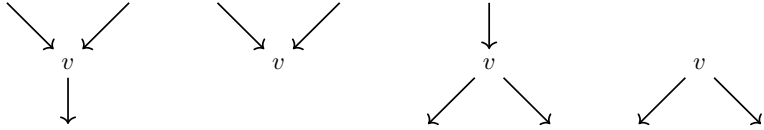
\vspace{10 pt}
    {\centering

    \resizebox{0.7\linewidth}{!}{
    \input{pictures/forbiddenVertices}
    }

}

\vspace{10 pt}

The following lemma is central to our proof:

\begin{lemma}
\label{lem::I(x)P(x)proj-inj}
    Let $A=kQ/I$ be a gentle algebra so that it satisfies (ii) from \Cref{gentleCharacterisation}. Then if $p$ is a right-maximal path in $Q$, $I(t(p))$ is projective-injective. 
    Dually, if $p$ is left-maximal, $P(s(p))$ is projective-injective. 
\end{lemma}

\begin{proof}
    If $p$ is right-maximal, then by the definition of gentle algebras, $\deg^{\text{out}}(t(p))< 2.$ Then assumption (ii) implies that  $\deg^{\text{in}}(t(p)) < 2,$ hence, $I(t(p))$ is uniserial, it corresponds to the unique left-maximal path ending at $t(p),$ which we can write as $ip$ for some path $i.$ Now we can use the left-maximality of $ip$ to obtain  $\deg^{\text{in}}(s(i)) < 2$, which then results in $\deg^{\text{out}}(s(i)) < 2$ by (ii). So $P(s(i))\cong M(q)$ where $q$ is the unique right-maximal path starting at $s(i)$. Note that $ip$ is a left subpath of $q$, and since $ip$ is right-maximal, $ip=q.$ This implies,  $I(t(p))\cong P(s(i))$ is projective-injective, and this is what we wanted. A dual argument proves the dual statement.
    
\end{proof}

\begin{proof}[Proof of \Cref{gentleCharacterisation}]
First, let $A$ be Auslander-Gorenstein, and for the sake of contradiction, assume that there is a vertex $v$ such that $\deg^{\text{in}}(v)=2$ but $\deg^{\text{out}}(v)<2.$ Let us denote the incoming arrows by $\alpha$ and $\beta$. Since $A$ is gentle, one of them, say $\alpha$, is a right-maximal path, meaning that there is no arrow $\gamma$ such that $\alpha\gamma\notin I.$  
To arrive at a contradiction consider the minimal injective resolution of $P(s(\alpha))=e_{s(\alpha)}A$. As $\alpha$ is a right-maximal path starting at $s(\alpha),$  $S_{t(\alpha)}=S_v$ is a direct summand of $\soc P(s(\alpha))$. So the injective envelope of this projective module contains $I(v)$ as a direct summand. However, since $\deg^{\text{in}}(v)=2,$ $I(v)$ is not uniserial, hence not projective-injective. Thus, $A$ is not $1$-Gorenstein, if such a $v$ exists. This means, there is no such $v.$

Since $A$ is gentle and $1$-Gorenstein, $A^{op}$ is also gentle and $1$-Gorenstein. Thus, by the previous argument, there is no vertex $v$ such that $(\deg_{Q^{op}}^{\text{in}}(v),\text{\textcolor{white}{ }}\deg_{Q^{op}}^{\text{out}}(v))=(2,1)$ or $(2,0)$ in $Q^{op}.$ 
The in-degree of a vertex is the same as the out-degree of this vertex in the opposite quiver and vice versa, so we can conclude that there is no vertex $v$ such that $(\deg_{Q}^{\text{in}}(v),\text{\textcolor{white}{ }}\deg_{Q}^{\text{out}}(v))=(1,2)$ or $(0,2)$ in $Q.$

For the other direction assume (ii). By \cite[Theorem 3.4]{GR05}, every gentle algebra is Iwanaga-Gorenstein, so it remains to show that if we consider a minimal injective resolution of an indecomposable projective module $P$
$$0\longrightarrow P \longrightarrow I^0 \longrightarrow I^1 \longrightarrow \ldots,$$
then $\pdim I^i\le i$ for every $i\ge 0.$  We do this by obtaining an explicit description of these resolutions. Let $v\in Q_0$. We consider two cases: when $P(v)$ is not uniserial, i.e. when $\deg^{\text{out}}(v)=2$, and when $P(v)$ is uniserial, i.e. when $\deg^{\text{out}}(v)<2.$ 
In the first case, $P(v)$ has a minimal injective resolution 
\begin{equation}\label{VInjRes}
0\longrightarrow P(v) \longrightarrow I(t_1) \oplus I(t_2) \longrightarrow I(v) \longrightarrow 0,
\end{equation}
where $t_1$ and $t_2$ are the endpoints of the two right-maximal paths $p_1$ and $p_2$ starting at $v.$ We have $\soc P(v)\cong S_{t_1}\oplus S_{t_2},$ so the injective envelope of $P(v)$ is $I(t_1)\oplus I(t_2).$ Since $p_j$ is right-maximal,  
$I(t_j)$ is projective-injective by \Cref{lem::I(x)P(x)proj-inj} for $j=1,2$, and corresponds to the left-and right-maximal path $i_jp_j.$

 
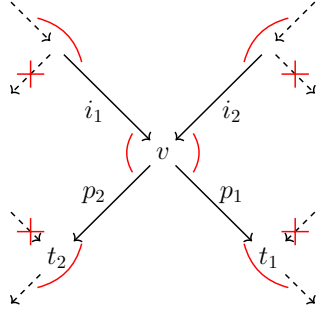
\vspace{10 pt}
    {\centering

    \resizebox{0.3\linewidth}{!}{
    \input{pictures/GentleCharacterisation1}
    }
    

}

\vspace{10 pt}

Since $P(s(i_j))\cong I(t_j)$ is uniserial, $s(i_j)\neq v$ and $i_j$ has length at least one. This implies that $i_j$ is already left-maximal and
$i_1$ and $i_2$ are the two left-maximal paths ending at $v.$ Note that $i_1$ and $i_2$ end with different arrows.  We can then easily see that $\Sigma^1(P(v))\cong I(v).$
So $\idim P(v)=1.$ Moreover, (\ref{VInjRes}) is a minimal projective resolution of $I(v);$ therefore, $\pdim (I(v))=1.$ 
In conclusion, $\pdim I^i \le i$ is fulfilled for every term in a minimal injective resolution of $P(v).$

Now consider the case when $\deg^{\text{out}}(v)<2.$ Then $P(v)$ corresponds to the unique right-maximal path $p$ starting at $v$, i.e. $M(p)\cong P(v)$. Let $t=t(p).$ The same argument as in the previous case tells us that the injective envelope of $P(v)$ is the uniserial projective-injective module $I(t)\cong P(s(i)) \cong M(ip)$ corresponding to the unique left-maximal path $ip$ ending at $t.$ Then either $i=e_v$ and $P(v)$ is injective or $\Sigma^1(P(v))\cong M(i(\beta_1))$ where $\beta_1$ is the unique arrow ending at $v.$ Note that $i=i(\beta_1)\beta_1.$ 
The uniqueness of $\beta_1$ follows from (ii). In the first case, the minimal injective resolution of $P(v)$ clearly satisfies our conditions, so let us consider the second case. We have that $M(i(\beta_1))$ is either injective or $\Sigma^2(P(v))\cong M(i(\beta_2))$ where $\beta_2$ is the unique arrow such that $t(\beta_2)=s(\beta_1)$ and $\beta_2\beta_1\in I.$ 
We can continue this argument and get a maximal critical path $\beta_m\ldots\beta_2\beta_1$ ending with $\beta_1$ such that $\Sigma^j(P(v))\cong M(i(\beta_j))$ for every $j=1,\ldots,m$ and $M(i(\beta_j))$ is injective if and only if $j=m.$ Note that there is no infinite critical path ending with $\beta_1.$ If there were, $\beta_1$ would be contained in a critical cycle. However, $\deg^{\text{out}}(v)<2$ and we are in the case when $P(v)$ is not injective, so there is no arrow $\alpha$ starting at $v$ with $ \beta_1\alpha\in I.$ This ensures that $m<\infty$. Thus, a minimal injective resolution looks as follows:

\begin{equation}\label{IInjRes}
0\longrightarrow P(v) \longrightarrow I(t)  \longrightarrow I(s(\beta_1))  \longrightarrow I(s(\beta_2)) \longrightarrow \ldots \longrightarrow I(s(\beta_m)) \longrightarrow 0.
\end{equation}

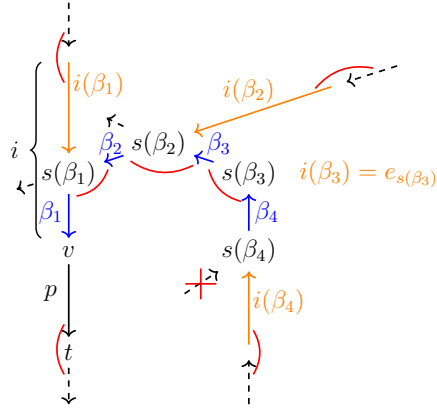
\vspace{10 pt}

{\centering

    \resizebox{0.4\linewidth}{!}{
    \input{pictures/GentleCharacterisation2}
    }
    

}

We already discussed that $\pdim I(t)=0.$ By (\ref{VInjRes}), if $\deg^{\text{in}}(s(\beta_j))=2$ then $\pdim I(s(\beta_j))=1\le j$ for every $j\ge 1.$
Otherwise, $\deg^{\text{in}}(s(\beta_j))\le 1$ and $I(s(\beta_j))$ is uniserial, it corresponds to the unique left-maximal path $p'$ ending at $s(\beta_j).$ For $j<m$,
$p'=i(\beta_{j+1})\beta_{j+1}.$ By (ii), the only arrow starting at $s(\beta_j)=t(\beta_{j+1})$ is $\beta_j$ itself, and $\beta_{j+1}\beta_j\in I$. Thus, $p'$ is right-maximal, and then we can conclude with \Cref{lem::I(x)P(x)proj-inj} that 
$I(s(\beta_j))\cong M(p')$ is projective-injective.  

It remains to see what happens for $j=m.$ 
Note that since $\beta_m\beta_{m-1}\ldots\beta_1$ was a maximal critical path ending with $\beta_1$, there is no arrow $\alpha$ with $t(\alpha)=s(\beta_m)$ such that $\alpha\beta_m\in I.$ In particular, by (ii), at most one arrow ends at $s(\beta_m)$. Then a dual argument to the one above (or \cite[3.3 Proposition]{GR05}) tells us that the projective dimension of $I(s(\beta_m))$ is equal to the length of the maximal critical path starting with $\beta_m.$ We already discussed that there is no arrow $\alpha$ with $t(\beta_1)=s(\alpha)$ such that $\beta_1\alpha\in I;$  therefore, $\beta_m\beta_{m-1}\ldots\beta_1$ is this maximal critical path. Hence, $\pdim I(s(\beta_m))=\text{length}(\beta_m\ldots\beta_1)=m\le m.$ In conclusion, (ii) implies (i).

\end{proof}

\begin{remark}
Note that in (\ref{IInjRes}) all the middle terms have projective dimension $0$ or $1$, while $\pdim I(s(\beta_m))=m$ since $\deg^{\text{in}}(s(\beta_m))\neq 2.$
\end{remark}

The proof of this theorem implies the following corollary. 

\begin{corollary}
\label{cor::1GoriffAG_gentle}
For a gentle algebra $A$ the following are equivalent:
\begin{enumerate}[\rm (1)]
    \item $A$ is Auslander-Gorenstein.
    \item $A$ is $1$-Gorenstein. 
\end{enumerate}
\end{corollary}

It is an easy-to-prove fact that if $A=kQ/I$ is a gentle algebra then it has finite global dimension if and only if $Q$ does not contain a critical cycle. So we also obtain a completely combinatorial characterisation of the Auslander regular gentle algebras.

\begin{corollary}
Let $A=kQ/I$ be a gentle algebra. Then the following are equivalent:
\begin{enumerate}[\rm (i)]
    \item $A$ is Auslander regular.
    \item $Q$ does not contain a critical cycle and for every $v\in Q_0$ $\deg^{\text{out}}(v)=2$ if and only if $\deg^{\text{in}}(v)=2$.
    \end{enumerate}
\end{corollary}

\begin{example}
 Applying \Cref{gentleCharacterisation}, we can see that the following quivers give us Auslander-Gorenstein algebras for any choice of relations which make them gentle: 

 
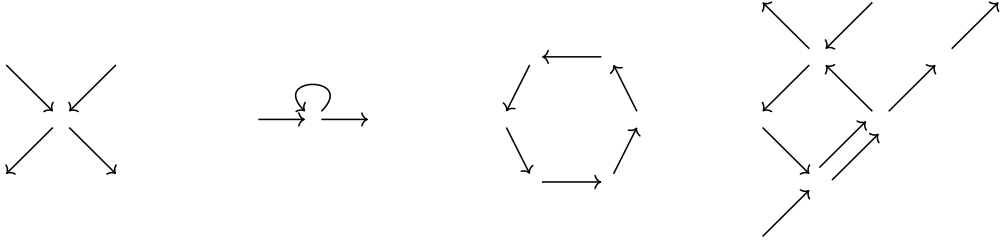
\vspace{10 pt}

{\centering

    \resizebox{0.9\linewidth}{!}{
    \input{pictures/GentleExamples}
    }
    

}
\vspace{5 pt}

The following quivers cannot belong to gentle Auslander-Gorenstein algebras:

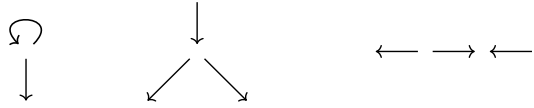
\vspace{10 pt}
{\centering

    \resizebox{0.5\linewidth}{!}{
    \input{pictures/GentleNonExamples}
    }
    

}

\end{example}
\begin{example}
    Let us spell out the computations for the algebra from \Cref{ex::section3.1}. 
    We have the following minimal injective resolutions for the indecomposable projective modules: 
    
\vspace{5pt}

    $$0\longrightarrow P(1)\longrightarrow I(3)\longrightarrow 0$$
    $$0\longrightarrow P(2)\longrightarrow I(3)\oplus I(3)\longrightarrow I(2)\longrightarrow 0$$
    $$0\longrightarrow P(3)\longrightarrow I(3)\longrightarrow I(2)\longrightarrow I(1) \longrightarrow 0$$

\vspace{5pt}
We have $\pdim I(3)=0$. The second resolution corresponds to (\ref{VInjRes}). It is also a minimal projective resolution of $I(2) $, thus, $\pdim I(2)=1$. The third resolution corresponds to the one described in (\ref{IInjRes}), since $ab$ is a maximal critical path ending with $b.$  Finally, 
$$0\longrightarrow P(3)\longrightarrow P(2)\longrightarrow P(1)\longrightarrow I(1) \longrightarrow 0$$
is a minimal projective resolution for $I(1).$ This corresponds to the dual of \ref{IInjRes}, since $ab$ is the maximal critical path starting with $a.$ So, $\pdim I(1)=2,$ which implies that $A$ from \Cref{ex::section3.1} is Auslander-Gorenstein, as expected from \Cref{gentleCharacterisation}.  
\end{example}

\subsection{The Auslander-Reiten bijection}

\vspace{10 pt}

The following theorem proves Marczinzik's conjecture for every gentle algebra. This yields an alternative homological characterisation of the Auslander-Gorenstein property.

\begin{theorem}

\label{thm::AGiffARbijwelldefGentle}

Let $A$ be a gentle algebra. Then the following are equivalent:

\begin{enumerate}[\rm (i)]
    \item $A$ is Auslander-Gorenstein.
    \item $A$ has a well-defined Auslander-Reiten map which is a bijection.
\end{enumerate}

\end{theorem}

Note that this theorem is also a consequence of \Cref{thm::AGiffARbijwelldef_Monomial}, which proves this result for every monomial algebra. However, in the special case of gentle algebras, there exists a much shorter proof, so we decided to include it here.  

\begin{proof}
By \Cref{thm::preliminaries::ARbij}, we only need to prove that (ii) implies (i). So assume (ii). We first show that if there is a $v$ with $\deg^{\text{out}}(v)=2\neq \deg^{\text{in}}(v)$ then $P(v)$ cannot be the last term in a minimal projective resolution of an injective module. For the sake of contradiction, assume that $\Omega^d(I(x))\cong P(v)$ for some vertex $x.$ Let
$$0\longrightarrow P(v)\cong P_d \xrightarrow[]{p_d}P_{d-1}\xrightarrow[]{p_{d-1}}P_{d-1}\ldots \longrightarrow P_0\longrightarrow I(x)\longrightarrow 0$$
denote a minimal projective resolution of this $I(x).$
The module $P(v)$ is not uniserial, so not injective, hence $d>0.$ If $d\ge 2,$ then $\Omega^{d-1}(I(x))$ has exactly one indecomposable non-projective direct summand, which according to \cite[3.3 Proposition]{GR05} is of the form $M(p(\alpha))$ for some arrow $\alpha.$ Thus, 

$$0\longrightarrow P(v)\cong P_d \xrightarrow[]{p_d\cong(0,\iota)} P_{d-1}\cong P'\oplus P(t(\alpha)) \xrightarrow{p_{d-1}\cong (id,\pi)}P'\oplus M(p(\alpha))\cong \Omega^{d-1}(I(x))\longrightarrow 0,$$
where $P'$ is projective. Since $\pi$ is a projective cover, $p_d(P_d)\subseteq \rad(P_{d-1}).$ In particular, $P(v)$ is (isomoprhic to) a proper submodule of $P(t(\alpha))$. However, in a gentle algebra indecomposable, proper submodules of indecomposable projectives are uniserial, as follows from our discussion in Section \ref{subsec::definitionsGentle}.
    So since $P(v)$ is not uniserial, $d=1$ must hold. We can use the following formula from the proof of \cite[3.3 Proposition]{GR05}:
    
    \begin{equation}
        \label{eq::firstSyzygyGentle}
        \Omega (I(x)) \cong P \oplus M_1 \oplus M_2
    \end{equation}
    where $P=P(x)$ if $I(x)$ is not uniserial and $P=0$ otherwise, and the $M_i$ are either 0 or of the form $M(i(\alpha))$ for some arrow $\alpha.$ So the non-uniserial module $P(v)$ can only be a direct summand of the first syzygy of $I(x)$ if $x=v$ and $I(x)$ is not uniserial. However, we chose $v$ such that $\deg^{\text{in}}(v)<2,$ so $I(v)$ is uniserial and thus, $d=1$ is not possible either. In conclusion, $P(v)$ is not in the image of $\psi_A,$ which contradicts (ii). 
    
    Now assume that there is a vertex $v$ with $\deg^{\text{in}}(v)=2\neq \deg^{\text{out}}(v).$ Let $i_1$ and $i_2$ be the two left-maximal paths ending at $v$. Note that they have length at least one. Since $A$ is gentle and there is at most one arrow starting at $v,$ at least one of the $i_j$-s is right-maximal. We may assume, $i_1$ is right-maximal. We will show that in this case $P(s(i_1))$ is not in the image of $\psi_A.$ 
    For this, note that  $P(s(i_1))$ is not a proper submodule of any indecomposable projective. This follows from the left-maximality of $i_1.$ So just as in the previous case, there is no vertex $x$ so that $\Omega^d(I(x))\cong P(s(i_1))$ for some $d\ge 2.$ 
    
    Every module of the form  $M(i(\alpha))$ where $\alpha$ is an arrow is a proper submodule of the indecomposable projective $P(s(\alpha))$. Therefore $P(s(i_1))$ is not of this form, so by (\ref{eq::firstSyzygyGentle}), $P(s(i_1))$ can only be a direct summand of $\Omega(I(x))$ if $x=s(i_1)$ and $I(x)$ is not uniserial. However, the left-maximality of $i_1$ implies that $\deg^{\text{in}}(s(i_1))<2$, in particular, $I(s(i_1))$ is uniserial. Thus, there is no $x$ with $\Omega(I(x))\cong P(s(i_1)).$
    
    Finally, $P(s(i_1))$ is not projective-injective as $S_v=S_{t(i_1)}$ is a direct summand of $\soc P(s(i_1))$. Here we use the right-maximality of $i_1.$ By choice, $I(v)$ is not uniserial, so $P(s(i_1))\ncong I(v)$. In particular, $P(s(i_1))$ is not projective-injective, and $P(v)$ is not in the image of $\psi_A$ in this case either. 
    
    In conclusion, (ii) implies that for every vertex $v$ $\deg^{\text{in}}(v)=2$ if and only if $\deg^{\text{out}}(v)=2$. By \Cref{gentleCharacterisation}, this is equivalent to (i).

\end{proof}

We can use the proof of \Cref{gentleCharacterisation} to describe the Auslander-Reiten bijection for Auslander-Gorenstein gentle algebras explicitly. In this description, we regard the underlying quiver and relations only as combinatorial objects, not considering their associated algebraic meaning. As previously discussed, it makes sense to think of this bijection as a permutation on the vertices of the quiver. 

\begin{corollary}

\label{cor::descriptionARbijGentle}
    Let $A=kQ/I$ be an Auslander-Gorenstein gentle algebra, and $\psi_A$ its Auslander-Reiten bijection. Let $v\in Q_0.$ We can describe $\psi_A$ explicitly as follows: 
    \begin{enumerate}[\rm (i)]
        \item If $\deg(v)=4$ or $0$ then $\psi_A(v)=v.$

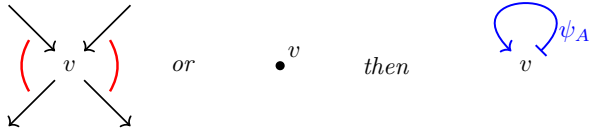
\vspace{5 pt}
{

    \resizebox{0.6\linewidth}{!}{
    \input{pictures/ARbijGentle1}
    }
    

}

\item If $\deg^{\text{in}}(v)=1$ and there is no arrow $\beta$ such that $\alpha\beta \notin I$ where $\alpha$ is the unique arrow ending at $v$, then $\psi_A(v)=s(p)$ with $p$ being the unique left-maximal path ending at $v$.

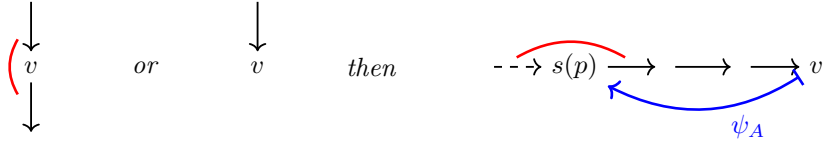
\vspace{5 pt}
{

    \resizebox{0.8\linewidth}{!}{
    \input{pictures/ARbijGentle2}
    }
    

}

\item If $\deg^{\text{in}}(v)=0$ and $\deg^{\text{out}}(v)=1$, or if $\deg^{\text{in}}(v)=1=\deg^{\text{out}}(v)$ and $\alpha\beta\notin I$ for the unique arrow $\alpha$  ending at $v$ and the unique arrow $\beta$ starting at $v$, then $\psi_A(v)=t(w)$ where $w$ is the unique maximal critical path starting at $v.$

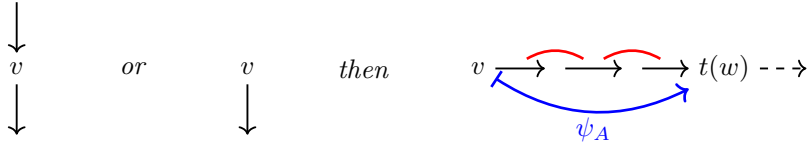
{

    \resizebox{0.8\linewidth}{!}{
    \input{pictures/ARbijGentle3}
    }
    

}

\end{enumerate}

\end{corollary}

\begin{proof}
If there are no arrows attached to $v,$ then $S(v)$ is projective-injective. If $v$ has degree $4$ then (\ref{VInjRes}) from the proof of \Cref{gentleCharacterisation} describes a minimal projective resolution of $I(v)$. In both cases, $\psi_A(v)=v.$

In case (ii), $I(v)\cong M(p),$ where $p$ is the unique left-maximal path ending at $v.$ According to \Cref{gentleCharacterisation} and  \Cref{lem::I(x)P(x)proj-inj}, this module is projective, so $I(v)\cong P(s(p))$. 

Finally, assume that $v$ satisfies the assumptions of (iii). Let $w=\beta_1 \beta_2 \ldots \beta_m.$ Then dual to ($\ref{IInjRes}$), the last term in a minimal projective resolution of $I(v)$ is isomorphic to $P(t(w))$.

\end{proof}

\begin{example}
    In this example, we draw the arrows of our quiver in two colours, either blue or orange. 
    A path of length two is contained in $I$ if and only if it contains two arrows of different colours. 
    Let $Q$ be as follows:
    
    
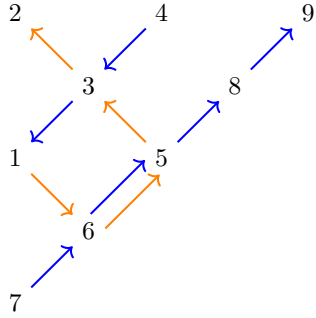
{
    \centering
    \vspace{10pt}
    

    \resizebox{0.3\linewidth}{!}{
    \input{pictures/ARbijExampleGentle}
    }
    \par
    }
    \vspace{10 pt}
    
    This defines a gentle algebra $A$. 
    In this case, we can think of searching for a left-maximal path ending at a vertex (resp. a maximal critical path starting at a vertex), as looking for a maximal path containing arrows of only one colour ending at the given vertex (resp. a maximal path in which the arrows alternate between the two colours starting at the given vertex.)
    
Vertices 3, 5, and 6 belong to Case (i), 1, 2, and 9 to Case (ii), and 4, 7, and 8 to Case (iii). Applying \Cref{cor::descriptionARbijGentle} and using the above description, it is easy to see that $\psi_A=(142)(3)(5)(6)(789).$

\end{example}

\newpage

\section{Characterisation of 2-Gorenstein monomial algebras}

\label{chap::monomial}

In the previous chapter, we saw that gentle algebras accept a nice combinatorial characterisation of the Auslander-Gorenstein property. It is natural to ask whether we can extend this result to a broader class of algebras, namely, to monomial algebras. We call a quiver algebra $A=kQ/I$ \textcolor{blue}{{monomial}} if $I$ is generated by a set of paths, rather than by linear combinations of paths. By definition, every gentle algebra is monomial, but the class of monomial algebras is significantly more general. Notably, there are no restrictions on the quiver $Q$ in the monomial case.

While a full classification of the Auslander-Gorenstein property for monomial algebras will not be achieved in this chapter, we will classify $2$-Gorenstein monomial algebras. 
From this, we will see that being $2$-Gorenstein is already a very restrictive condition for a monomial algebra, in particular, any such algebra is a string algebra. This result is significant because indecomposable modules and irreducible morphisms over string algebras are well-understood, and their combinatorial structure makes them more accessible to study, as already mentioned in \Cref{chap::gentle}.
Another key observation is that $2$-Gorenstein monomial algebras lie, in a sense, very close to gentle or to Nakayama algebras. The main goal of this chapter is to prove the following result.

\begin{theorem}
\label{thm::classification2GorMonomial}
Let $A=kQ/I$ be a monomial algebra. Then $A$ is $2$-Gorenstein if and only if it satisfies the following conditions:
\begin{enumerate}[\rm (1)]
    \item $Q$ is biserial, i.e. for every vertex $x\in Q_0$ we have $\deg^{\text{out}}(x)\le 2$ and $\deg^{\text{in}}(x)\le 2.$
    \item For every vertex $x\in Q_0$ we have $\deg^{\text{out}}(x)=2$ if and only if $\deg^{\text{in}}(x)=2.$ 
    \item If $\deg(x)=4$ then we can label the incoming arrows at $x$ by $a_1$ and $a_2$ and the outgoing arrows from $x$ by $b_1$ and $b_2$, such that $a_1b_2,a_2b_1\in I,$ and $a_ib_i$ are not contained in any minimal relations for $i=1,2$.
  
{
\centering
\resizebox{0.15\linewidth}{!}{
    \input{pictures/degree4vertex}}
    
}

\item If an arrow is contained in a minimal relation, then this arrow must be the start or the end of a minimal relation.

{
\centering
\resizebox{0.3\linewidth}{!}{
    \input{pictures/startOrEndOfMinimalRel_picture1}}
    
}

\end{enumerate}
\end{theorem}

\begin{corollary}
\label{cor::2GormonomialImpliesString}
    Every $2$-Gorenstein (and hence, every Auslander-Gorenstein) monomial algebra is a string algebra.
\end{corollary}

\begin{proof}
By definition \cite{BR87}, a string algebra is a biserial, monomial algebra $A=kQ/I$ such that 
\begin{enumerate}
    \item for any $a_1, a_2, b \in Q_1$ with $a_1\neq a_2$ and $t(a_1) = t(a_2) = s(b)$ we have
$|\{a_1b, a_2b \} \cap I| \ge 1$

\item for any $a, b_1, b_2 \in Q_1$ with $b_1\neq b_2$ and $t(a) = s(b_1) = s(b_2)$, we have
$|\{ab_1, ab_2 \} \cap I| \ge 1.$
\end{enumerate}
    
\end{proof}

\begin{example}
\label{ex::Section42GorMon}

Consider the following quiver $Q$
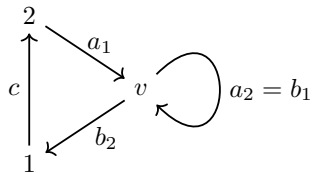
\vspace{10pt}

{\centering

    \resizebox{0.3\linewidth}{!}{
    \input{pictures/Section4Example2Gor}
    }
    
}
\vspace{10pt}
Let $A_1=kQ/\langle a_1b_2, a_2^2, b_2ca_1 \rangle$ and $A_2=kQ/\langle a_1b_2, a_2^2, ca_1 \rangle.$ Both of these define monomial algebras satisfying conditions (1), (2), and (3) from \Cref{thm::classification2GorMonomial}. For this, note that a loop at $v$ contributes $+1$ to $\deg^{\text{in}}(v)$, $+1$ to $\deg^{\text{out}}(v)$, and thus, $+2$ to $\deg(v)=\deg^{\text{in}}(v)+\deg^{\text{out}}(v).$ Condition (4) is also true for $A_2$; however, it does not hold for $A_1$ as $b_2ca_1$ is a minimal relation, but no minimal relation starts or ends with arrow $c$ in $A_1.$

The indecomposable projective and injective modules over $A_1$ can be depicted as follows:

\vspace{10pt}
{\centering

    \resizebox{\linewidth}{!}{
    \input{pictures/ExampleSection4A}
    }
}
\\
\vspace{10pt}
Then a minimal injective presentation of $P(2)$ is given by
$$0\longrightarrow P(2) \longrightarrow I(2) \longrightarrow I(1).$$ 
A projective cover $p$ of $I(1)$ fits into the short exact sequence
$$ 0\longrightarrow S_2\cong \Omega(I(1))\longrightarrow P(1) \xrightarrow{p} I(1)\longrightarrow 0,$$ and since $S_2$ is not projective, $\pdim I(1)\ge 2,$ which implies that $A_1$ is not 2-Gorenstein.

Consider now the modules over $A_2$.

\vspace{10pt}
{\centering

    \resizebox{\linewidth}{!}{
    \input{pictures/ExampleSection4B}
    }
}
\\
\vspace{10pt}
We should first observe that $P(2)\cong I(2)$ is projective-injective. Thus, the following minimal injective resolution of $P(v)$ is also a minimal projective resolution of $I(v)$:
$$ 0\longrightarrow P(v)\longrightarrow I(2)\oplus I(2) \longrightarrow I(v)\longrightarrow 0.$$
Finally, a minimal injective presentation of $P(1)$ is given by 
$$ 0\longrightarrow P(1)\longrightarrow I(2) \longrightarrow I(v).$$
By our previous observations, $\pdim I(2)=0$ and $\pdim I(v)=1.$ Thus, $A_2$ is 2-Goresntein, as expected from \Cref{thm::classification2GorMonomial}.

\end{example}

\begin{remark}
Note that conditions (1) and (3) also show up in the definition of gentle algebras, and (2) is the same condition as seen in the classification of Auslander-Gorenstein gentle algebras in \Cref{gentleCharacterisation}. So we can think of 2-Gorenstein monomial algebras as a collection of Nakayama algebras glued together at degree 4 vertices in a specific way, making the algebra look like a gentle algebra at any degree 4 vertex. This observation will be crucial for the next chapter, where we prove that one can reduce the classification problem of the Auslander-Gorenstein property for monomial algebras to that of Nakayama algebras.
\end{remark}

\begin{remark}
We would like to draw the reader's attention to the paper \cite{LZ21}. In that work, the authors classify when a monomial algebra is 1-Iwanaga-Gorenstein, in which case, all minimal relations must lie on oriented cycles. This phenomenon is somewhat similar to what we discover in the present paper; namely, that 2-Gorenstein monomial algebras can be constructed by gluing together Nakayama algebras.
\end{remark}

\begin{remark}
\Cref{thm::classification2GorMonomial} implies that a gentle algebra $A=kQ/I$ is $2$-Gorenstein if and only if they satisfy condition (2). This is a somewhat weaker result than Theorem \ref{gentleCharacterisation}, which proves that these properties are further equivalent to $A$ being Auslander-Gorenstein, and to $A$ being 1-Gorenstein. 
\end{remark}

\subsection{Background on monomial algebras}
\label{subsec::notationMonomial}
Before getting to work with monomial algebras, let us discuss some important facts and fix some notation. One important resource regarding monomial algebras is \cite{CB89}. In this paper, Crawley-Boevey explicitly describes a basis of the homomorphism spaces between any two tree modules. In the special case when our monomial algebra is a string algebra, tree modules are exactly the string modules. 
 
Another significant result about the homological properties of monomial algebras is due to Huisgen-Zimmermann, who proves that indecomposable direct summands of higher syzygies (not including first syzygies) are all of a simple form, which will play a central role in the proof of \Cref{thm::AGiffARbijwelldef_Monomial}.

\begin{lemma}\cite[Chapter 3. Theorem I]{ZH91} \label{lem::syzygiesInMonomialAlg}
    Let $A$ be a monomial algebra and $M$ an $A$-module. Then every indecomposable direct summand of the syzygy $\Omega^{r}(M)$ with $r\ge 2$ is isomorphic to a module of the form $pA$ where $p$ is a path of length at least one. 
\end{lemma}

Since we are interested in $n$-Gorenstein properties of monomial algebras, we first need to understand first syzygies of indecomposable injectives (or dually, first cosyzygies of indecomposable projectives), which are not covered by the previous lemma. We start with an explicit description of the indecomposable projectives and injectives.

Let $x$ be a vertex in the monomial algebra $A=kQ.$
Let $\textcolor{blue}{B}$ be the set of paths $p$ in $Q$ such that $t(p)=x$ and $p\notin I$. Then a $k$-basis of $I(x)$ can be given by $\{p'\text{ }|\text{ }p\in B\}$ where the right $A$-module action is described as follows: For every $i\in Q_0, a \in Q_1$ and $p,q\in B$ we have $p'\cdot e_i=p'$ if $s(p)=i$ and $p'\cdot e_i=0$ otherwise, and $p'\cdot a=q'$ if $p=aq$ and $p'\cdot a=0$ otherwise. This determines $I(x)\cong D(e_xA).$ So $I(x)$ can be depicted as a rooted tree, as shown in \Cref{fig:injectiveMonomial_picture}. Let $\textcolor{blue}{B_{\text{max}}}\subset B$ be the subset of left-maximal paths in $B.$ Then $\top I(x)\cong \bigoplus_{p\in B_{\text{max}}} S_{s(p)}.$

Dually, since $P(x)=e_xA,$ a $k$-basis of $P(x)$ is given by $\textcolor{blue}{C}$, the set of paths starting at $x$ not contained in $I.$ The right $A$-module action is given by composition, which is also illustrated in \Cref{fig:injectiveMonomial_picture}. Let $\textcolor{blue}{C_{\text{max}}}\subset C$ be the subset of right-maximal paths in $C.$ Thus, $\soc P(x) \cong \bigoplus_{p\in C_{\text{max}}} S_{s(p)}.$

    \vspace{10 pt}
    {\centering
\begin{figure}[h]
    \centering

    \resizebox{0.7\linewidth}{!}{
    \input{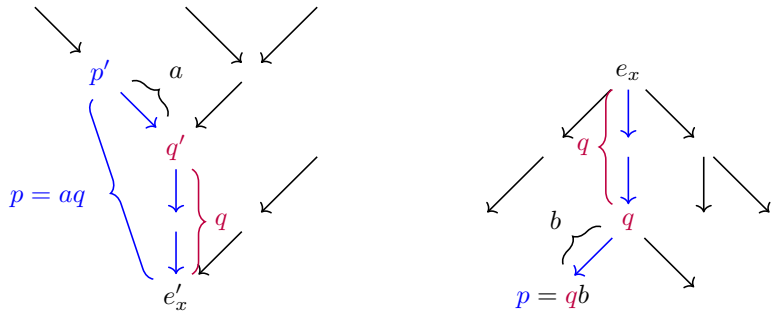}
    }
    
    \caption{The module structure of $I(x)$ (left) and $P(x)$ (right)}
    \label{fig:injectiveMonomial_picture}
\end{figure}

}
Hence, every projective-injective module is uniserial. 
For the remainder of this section, let us assume that $A$ is 1-Gorenstein. Consider a minimal projective presentation of $I(x):$ $$P_1\xrightarrow[]{d_1} P_0 \xrightarrow{d_0} I(x)\to 0.$$
We have $P_0\cong \bigoplus_{p\in B_{\text{max}}} P({s(p)}).$ If $A$ is $1$-Gorenstein, every $P({s(p)})$ for $p\in B_{\text{max}}$ is projective-injective, hence uniserial, and corresponds to the unique right-maximal path starting at $s(p).$ We can then describe the map $d_0$ on the direct summands of $P_0$ as $$d_0: P(s(p))\to I(x),\text{ } q\mapsto p'\cdot q.$$ This is illustrated in \Cref{fig:projCoverI(x)}.

   \vspace{10 pt}
    {\centering
\begin{figure}[h]
    \centering

    \resizebox{0.7\linewidth}{!}{
    \input{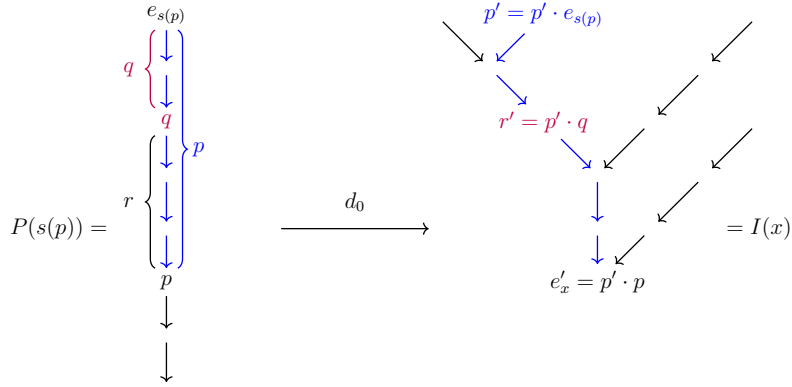}
    }
    
    \caption{The projective cover of $I(x)$}
    \label{fig:projCoverI(x)}
\end{figure}

}

For any two different paths $p,t\in B_{\max}$ let \textcolor{blue}{$r_{p,t}$} be the shortest left subpath of $p$ such that there is a path $s$ with $p=r_{p,t}s$ and $t=r_{t,p}s.$ Let us view $r_{p,t}$ as an element of the direct summand $P(s(p))$ of $P_0$. Then we get $r_{t,p}-r_{p,t}\in \ker(d_0),$ since $d_0(r_{p,t})=d_0(r_{t,p})=s'.$ Using that $A$ is a monomial algebra and that the $P(s(p))$ are uniserial, one can show the following statement, which is needed for the description of $\Omega(I(x))$:

\begin{lemma}\label{lem::descriptionFirstSyzygyMonomial}
    Let $A$ be a 1-Gorenstein monomial algebra. Then using the above notation, 
    
    $$\mathcal{S}=\{r_{t,p}-r_{p,t}\text{ }|\text{ }p,t\in B_{\max}\}\cup \{pa\in P(s(p))\text{ }|\text{ }p\in B_{\text{max}}, a\in Q_1, pa\notin I \}$$
    generates $\Omega(I(x))=\ker(d_0)$ as an $A$-module.
\end{lemma}

\begin{proof}
We already argued that $r_{t,p}-r_{p,t}\in \ker(d_0)$, and the elements of the second set are clearly sent to $0$ by $d_0$.

Let us denote the maximal paths by $B_{\max}=\{p_i \ |\ i\in J \}$ for some (finite) index set $J.$ Then $P_0=\bigoplus_{i\in J} P(s(p_i)).$ Take any element $w\in P_0$ such that $d_0(w)=0.$ This $w$ is a linear combination of non-zero paths starting at vertices of the form $s(p_i)$ where $i\in J.$ As $d_0$ is a linear map, without loss of generality, we may omit the summands corresponding to a path $q$ which contains a $p_i$ as a proper left subpath, in other words, when there is an arrow $a$ such that $p_ia$ is a left subpath of $q.$ In this case $q$ is obviously in the $A$-module generated by $\mathcal{S}$ and $d_0(q)=0.$ 

Hence, we can write $w$ as a finite sum $$w=\Sigma_{i\in J,r\in B}c_i^r q_i^r=:\Sigma_{r\in B}w_r$$ where for a tuple $(i,r)\in J\times B, c_{i}^r\in k$ and $q_i^r$ is either zero or a path starting at $s(p_i)$ such that $p_i=q_i^r r.$ We define $w_r=\Sigma_{i\in J}c_i^rq_i^r.$ If $q_i^r\neq 0,$ $q_i^r\in P(s(p_i))$ and $d_0(q_i^r)=p_i'\cdot q_i^r=r'.$ 
So we have $$d_0(w)=\Sigma_{i,r}c_i^r (p_i'\cdot q_i^r)=\Sigma_{r\in B}(\Sigma_{i\text{: r is a right subpath of }p_i}c_i^r) r'=0.$$
    Consider a fixed $r\in B$. Then by the previous formula, 
    \begin{equation}
        \label{eq::coeffSumToZero}
        \Sigma_{i\text{: r is a right subpath of }p_i}c_i^r=0.
    \end{equation}
    Let us renumber the elements of $B_{\max}$ such that $p_1,p_2,\ldots,p_n$ are exactly the elements in $B_{\max}$ which contain $r$ as a right subpath.
    Note that for each $i\in\{1,\ldots,n\}$ there is a path $h_i$ such that $q_i^r=r_{p_i,p_{i+1}}h_i$ by definition of $r_{p_i,p_{i+1}}$ and since $q_i^rr=p_i$ and $q_{i+1}^rr=p_{i+1}.$ Note that this implies $q_{i+1}^r=r_{p_{i+1},p_{i}}h_i$ as well. Here we set $n+1:=1.$  Then we can then write  
    $$w_r=\Sigma_{i=1}^n c_i^rq_i^r = c_1^r(q_1^r-q_2^r)+(c_1^r+c_2^r)(q_2^r-q_3^r)+\ldots + (c_1^r+\ldots +c_{n-1}^r)(q_{n-1}^r-q_n^r) =$$
    $$=c_1^r(r_{p_1,p_2}-r_{p_2,p_1})h_1+(c_1^r+c_2^r)(r_{p_2,p_3}-r_{p_3,p_2})h_2+\ldots+(c_1^r+\ldots +c_{n-1}^r)(r_{p_{n-1},p_n}-r_{p_n,p_{n-1}})h_{n-1}.$$
    The second equality follows from (\ref{eq::coeffSumToZero}), i.e. $c_1^r+\ldots+c_n^r=0$. Thus, we get that $w_r$ is contained in the $A$-module generated by $\mathcal{S}.$ This concludes the proof. 
    
\end{proof}

\subsection{Proof of Theorem \ref{thm::classification2GorMonomial}}

We devote the remainder of this chapter to proving our classification theorem. The arguments that follow will be quite technical and will rely heavily on the combinatorial structure of monomial algebras, in particular, the tree module structure of projective and injective modules, and the left-right symmetry of the $n$-Gorenstein property. 
We break up the proof into smaller steps. We will be using the notation introduced in \Cref{subsec::notationMonomial} throughout this chapter. The following statement will be useful in our later proofs.

\begin{lemma}
\label{lem::1GorMonomial: dimsocP(x)=dimtopI(x)}
    Let $A$ be a 1-Gorenstein monomial algebra. Then, for every vertex $x$ we have $$\dim(\soc P(x))= \dim(\top I(x)).$$ 
    
    Moreover, let $C_{\max}=\{z_1,\ldots,z_n\}$ be the right-maximal paths starting at $x.$ For each $i$ there is a path $p_i$ ending at $x$ such that  $w_i=p_iz_i$ is a right-and left-maximal path corresponding to the projective-injective module $M(w_i)\cong I(t(z_i))\cong P(s(p_i)).$ We have $p_iz_j\in I$ for every $i\neq j.$  
\end{lemma}

\begin{proof}

Recall that $\soc P(x)\cong \bigoplus_{z\in C_{\max}} S_{t(z)}.$ So the 1-Gorenstein property of $A$ ensures that $I(t(z_i))$ is projective-injective for every $z_i\in C_{\max},$ in particular, there exists a path $p_i$ ending at $x$ so that $w_i=p_iz_i$ is left- and right-maximal and $I(t(z_j))\cong M(w_j) \cong P(s(w_j))$. Note that we do not claim that $p_i$ is left-maximal itself, but in any case, it is a right subpath of a left-maximal path $p_i'\in B_{\max}$ ending at $x$.

Observe that $p_i$ is not a right subpath of $p_j$ if $j\neq i.$ If it were, $p_jz_j\notin I$ would imply $p_iz_j\notin I$ either. So $w_i=p_iz_i$ and $p_iz_j$ were two non-zero paths starting at $s(w_i)$, which in addition, are both different and right-maximal as $z_i\neq z_j\in C_{\max}.$ This meant that the socle of $P(s(w_i))$ is at least two-dimensional, which however, contradicts the projective-injectivity of $P(s(w_i))$. Thus, $p_i$ is not a right subpath of $p_j$ if $j\neq i.$

This immediately implies that $p_iz_j\in I$ as $p_jz_j$ is the unique left-maximal path ending at $t(z_j).$ Furthermore, it is also clear from this that $p_i'\neq p_j'$ as they end with different paths. This implies that $\dim(\top I(x))=|B_{\max}| \ge |C_{\max}|=\dim (\soc P(x)).$ A dual argument works for the direction $|B_{\max}| \le |C_{\max}|.$

\end{proof}

We know that if $A$ is a monomial algebra, then the projective-injective modules are uniserial. To understand the $2$-Gorenstein property, we would like to have a description of how injective modules with projective dimension one look.

\begin{lemma}
\label{pdim1Lemma}
    Let $A=kQ/I$ be a $1$-Gorenstein monomial algebra. Let $I(x)$ be the indecomposable injective module at vertex $x$ and assume that $I(x)$ has projective dimension $1.$ Then the following hold:
    \begin{enumerate}[\rm (i)]
        \item $\dim(\top I(x))\le 2.$ In other words, there are at most two different left-maximal paths ending at $x.$
        \item Assume that $p_1\neq p_2$ are two left-maximal paths ending at $x.$ Then there is no arrow $a$ such that $p_1a\notin I$ and $p_2a\notin I$ hold at the same time.
        
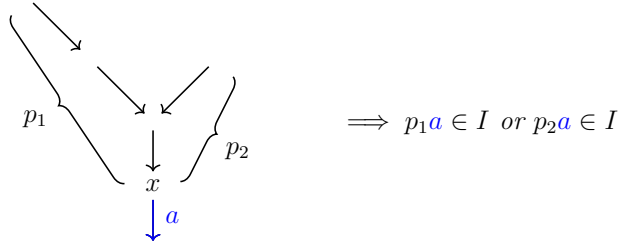
\vspace{10 pt}
{\centering
    \centering

    \resizebox{0.6\linewidth}{!}{
    \input{pictures/pdim1InjMonomialAlg_part2}
    }
    

}
        
\end{enumerate}

Since being 1-Gorenstein is left-right symmetric, the dual statements for describing indecomposable projective modules with injective dimension one also hold.

\end{lemma}
  
\begin{proof}
We will show part (i) by explicitly describing a minimal projective presentation and the first syzygy module of $I(x).$ For this, recall that we use the notation 
$$P_1\xrightarrow[]{d_1} P_0 \xrightarrow{d_0} I(x)\to 0$$ for a minimal projective presentation of $I(x)$ and that we already described $d_0$ in Section \ref{subsec::notationMonomial}. 
Assume that $|B_{\text{max}}|\ge 3,$ which is equivalent to $\dim(\top I(x))\ge3.$ Then there exists a path in $Q$, call it $q$, such that $p_1=q_1q$, $p_2=q_2q$ and $p_3=q_3q$ for three different paths $q_1,q_2$ and $q_3$ with $p_1,p_2,p_3\in B_{\text{max}}$. Choose $q$ such that it is maximal with this property, i.e. such that there is no $a\in Q_1$ such that this statement also holds for $aq.$
Note that $q=e_x$ is possible. 
There are two different cases to consider. First, assume that $q_1-q_2$ is not in $\mathcal{S}$.
Here $\mathcal{S}$ is the set introduced in \Cref{lem::descriptionFirstSyzygyMonomial}. 
This means that there exists a path $\Tilde{q}$ of length at least one such that $q_1=r_{p_1,p_2}\Tilde{q}$ and $q_2=r_{p_2,p_1}\Tilde{q}.$
We should have the following picture in mind:
   
   
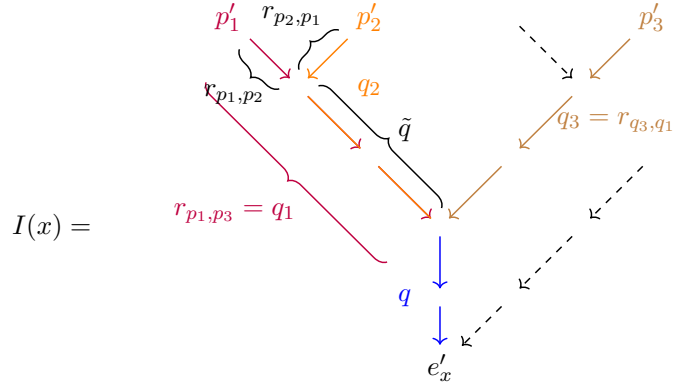
\vspace{10pt}

 {


\centering
\resizebox{0.6\linewidth}{!}{
    \input{pictures/3dimtopIx_case1}}
    

}

\vspace{10pt}
   
    By the maximality of $q$, we have $q_1=r_{p_1,p_3}, q_2=r_{p_2,p_3}$ and $q_3=r_{p_3,p_1}=r_{p_3,p_2}.$ So $q_1-q_3, q_2-q_3\in \mathcal{S}.$ A further consequence of this maximality is that $r_{p_1,p_2}-r_{p_2,p_1}, q_1-q_3\notin \rad(\ker(d_0)).$ This can be seen with the help of \Cref{lem::descriptionFirstSyzygyMonomial}. As a consequence,  $S_{t(r_{p_1,p_2})}\oplus S_{t(q_1)}$ is (isomorphic to a) direct summand of $\top(\ker(d_0))$, $P(t(r_{p_1,p_2}))\oplus P(t(q_1))$ is a direct summand of $P_1$, and $d_1$ can be defined to act on it as 
   
   $$d_1: P(t(r_{p_1,p_2}))\oplus P(t(q_1)) \to P(s(p_1))\oplus P(s(p_2))\oplus P(s(p_3)),\text{ }(u,v)\mapsto (r_{p_1,p_2}u+q_1v,-r_{p_2,p_1}u,-q_3v).$$
   We can visualise $d_1$ (restricted to the aforementioned direct summands of $P_1$) as follows:

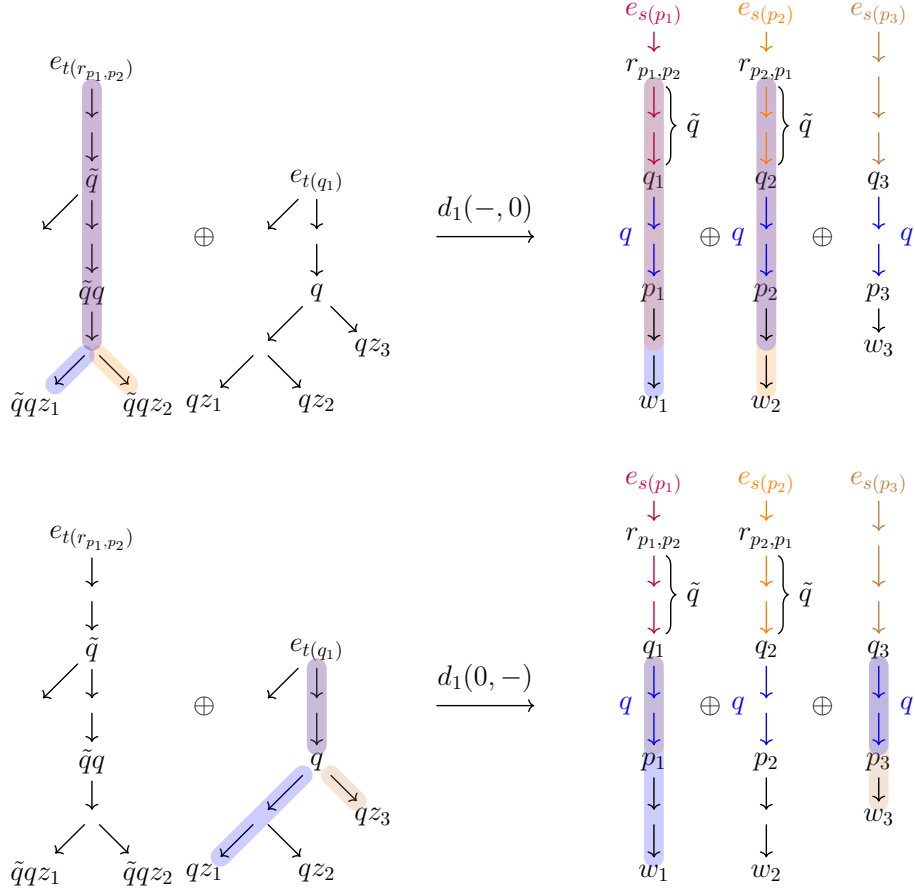
\vspace{10pt}
 {


\centering
\resizebox{0.8\linewidth}{!}{
    \input{pictures/d1_-,0_}}
    

}

\vspace{15 pt}
   {


\centering
\resizebox{0.8\linewidth}{!}{
    \input{pictures/d1_0,-_}}
    

}

\vspace{10pt}

   For $i=1,2,3$ let $z_i$ be the path completing $p_i$ to the unique right-maximal path starting at $s(p_i),$ $w_i=p_iz_i.$ The dual of the statement of \Cref{lem::1GorMonomial: dimsocP(x)=dimtopI(x)} implies $p_iz_j=0$ if $i\neq j.$ Therefore, $d_1(0,qz_1)=(w_1,0,0)=d_1(\Tilde{q}qz_1,0).$ Hence, $0\neq (\Tilde{q}qz_1,-qz_1)\in \ker(d_1).$ This, however, contradicts $\pdim I(x)=1.$

   It remains to consider the case when $q_1-q_2,q_2-q_3,q_1-q_3\in \mathcal{S}.$ In this case, $I(x)$ can be depicted as follows:

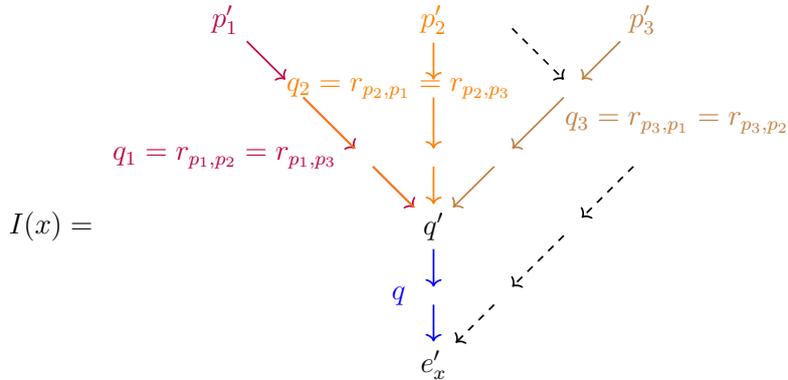
\vspace{10pt}
   {


\centering
\resizebox{0.7\linewidth}{!}{
    \input{pictures/3dimtopIx_case2}}
    

}

\vspace{10pt}

The same argument as before works here, if we replace $r_{p_1,p_2}$ by $q_1,$ $r_{p_2,p_1}$ by $q_2,$ and set $\Tilde{q}=e_{s(q)}.$ Hence, we must arrive at a contradiction in this case as well. This concludes the proof of part (i).

Consider part (ii). Let $\pdim I(x)=1$ and $B_{\max}=\{p_1,p_2\}$ with $p_1\neq p_2.$ 
Let $s_i=s(p_i)$ for $i=1,2$. Then $P_0\cong P(s_1)\oplus P(s_2)$ and because of its projective-injectivity, $P(s_i)$ corresponds to the unique right-maximal path $w_i$ starting at $x_i.$ Note that there are paths $z_i$ such that $w_i=p_iz_i$. Moreover, $p_1-p_2\in \ker(d_0).$ 
Let $p$ be the path such that $p_1=r_{p_1,p_2}p$ and $p_2=r_{p_2,p_1}p$. 
For a visualisation of the projective cover $P_0 \xrightarrow[]{d_0} I(x)$, we refer to \Cref{fig:d_0}.

{

\begin{figure}[h]

\centering
\resizebox{0.5\linewidth}{!}{
    \input{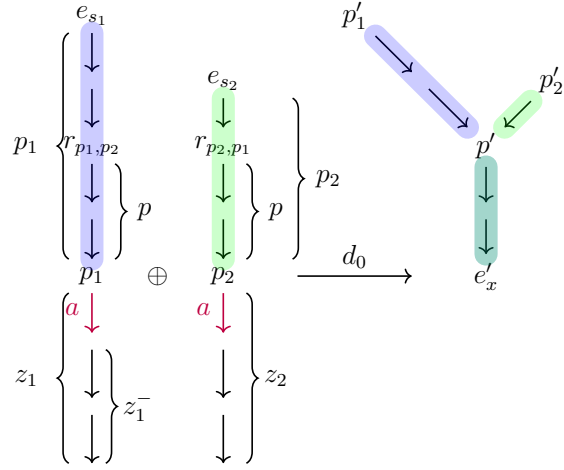}}
    
    \caption{Projective cover of $I(x)$}
    \label{fig:d_0}
\end{figure}
} 
If $w_i=p_iaz_i^-$ for both $i=1$ and $i=2$ for some $a\in Q_1$ then $\ker(d_0)$ is generated as a right $A$-module by the $A$-linearly independent paths $r_{p_1,p_2}-r_{p_2,p_1}$ and $p_1a$. This can be deduced from \Cref{lem::descriptionFirstSyzygyMonomial}. 
Thus, $\top(\ker(d_0))\cong S_{t(r_{p_1,p_2})}\oplus S_{t(a)}$, in particular, the top is not simple.  However, $\ker(d_0)$ is indecomposable, as $0\neq (r_{p_1,p_2}-r_{p_2,p_1})\cdot pz_1=w_1=(p_1a)\cdot z_1^-\in \ker(d_0).$ Here, the first equality is a consequence of $p_iz_j\in I$ for $i\neq j$, which is the dual statement of \Cref{lem::1GorMonomial: dimsocP(x)=dimtopI(x)}. This implies that $\ker(d_0)$ is not projective, contradicting $\pdim I(x)=1.$ 
So there cannot exist such an arrow $a$.

\end{proof}

We use the previous lemma to show that $Q$ must be biserial if $A$ is $2$-Gorenstein. 

\begin{lemma}
\label{biserialLemma}
    Let $A=kQ/I$ be monomial and $2$-Gorenstein. Then $\dim(\soc(P(x)))\le 2$ and $\dim(\top(I(x)))\le 2$ for every vertex $x.$  These immediately imply that $Q$ is biserial.
\end{lemma}
\begin{proof}
It is enough to prove the statement for $P(x),$ the dual statement for $I(x)$ follows then immediately from the fact that $A$ is $2$-Gorenstein if and only if $A^{op}$ is. 
We will show this by examining the minimal injective resolution of a $P(x)$.

Let 
$$0\longrightarrow P(x)\longrightarrow I^0 \longrightarrow I^1 \longrightarrow \ldots $$ be a minimal injective resolution of $P(x).$ 
As described in \Cref{subsec::notationMonomial}, $P(x)= e_xA$ is generated as a $k$-vector space by $C$, the set of paths starting at $x$ not contained in $I,$ and $C_{\max}$ denotes the set of right-maximal paths in $C$.

Assume that there is a path $p$ such that there are three different paths $p_1,p_2,p_3\in C_{\max}$ with $p_i=pq_i$ for some paths $q_i,$ and that $p$ is maximal with this property. Then dual to our argument in Lemma \ref{pdim1Lemma}, by the maximality of $p,$ $S(t(p))$ is going to be a direct summand of $\soc(\coker(d^0)).$ Thus, $I(t(p))$ is a direct summand of $I^1.$ However, $\dim (\soc P(t(p))\ge 3$ as $q_1,q_2$ and $q_3$ are three different paths starting at $t(p)$ and none of them is a subpath of the others. We can conclude from \Cref{lem::1GorMonomial: dimsocP(x)=dimtopI(x)} that $\dim (\top I(t(p)))=\dim (\soc P(t(p)))\ge 3$, which ensures that $\pdim I(t(p))\ge 2$ as established in \Cref{pdim1Lemma} (i). This contradicts the 2-Gorenstein assumption on $A$.

\end{proof}

With the next lemma we prove that a $2$-Gorenstein monomial algebra must satisfy the same condition as Auslander-Gorenstein gentle algebras, see \Cref{gentleCharacterisation}.

\begin{lemma}
\label{deg4vertexLemma}
    Let $A=kQ/I$ be monomial and $2$-Gorenstein. Then for any vertex $x$ we have $\deg^{\text{out}}(x)=2$ if and only if $\deg^{\text{in}}(x)=2$. 
    
\end{lemma}
\begin{proof}
We will use the notation from \Cref{subsec::notationMonomial} and some similar arguments as in Lemma \ref{biserialLemma} (ii).
 For the sake of contradiction, assume that there is an $x\in Q_0$ with $\deg^{\text{out}}(x)=2$ but $\deg^{\text{in}}(x)\le 1.$ Then by \Cref{biserialLemma}, $\soc(P(x))$ is $2$-dimensional and $C_{max}=\{z_1, z_2\},$ $z_1\neq z_2.$ Set $y_i=t(z_i).$ Then the injective envelope of $P(x)$ is $I^0\cong I(y_1)\oplus I(y_2).$ 
By the $1$-Gorenstein property of $A,$ $I(y_i)\cong M(w_i)\cong P(s(w_i))$ where $w_i$ is the unique left-maximal path ending at $y_i$. We write $w_i=p_iz_i.$  Since $P(s(w_i))$ is uniserial $p_i\neq e_x$, so $p_i$ consists of at least one arrow. By the $\deg^{\text{in}}(x)\le 1$ assumption, $p_1$ and $p_2$ must end with the same arrow $a,$ so $p_i=p_i^+a$ for $i=1,2.$

 
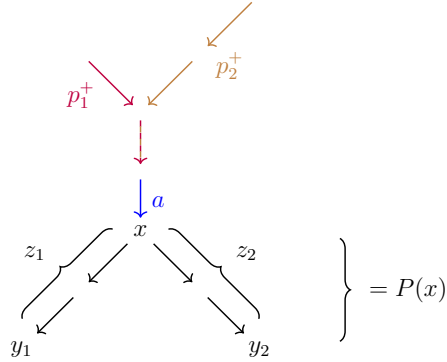
{


\centering
\resizebox{0.4\linewidth}{!}{
    \input{pictures/deg_in=2iffdeg_out=2}}
    

}

In particular, $az_i\notin I$ for both $i=1$ and $i=2$, so we can apply the dual of \Cref{pdim1Lemma} (ii) to obtain that $\idim P(x)\ge 2.$ This motivates the following plan to reach a contradiction: Show that  $P(x)$ appears as a direct summand in the first term of a minimal projective resolution of an indecomposable injective $A$-module. This then violates the 2-Gorenstein assumption on $A.$

We prove that a correct choice for such an injective $A$-module is $I(s(a)).$ Denote a minimal projective resolution of this as
$$\ldots \rightarrow P_1 \xrightarrow{d_1} P_0 \xrightarrow{d_0} I(s(a)) \rightarrow 0. $$
Note that $p_1$ cannot be a subpath of $p_2$ and vice versa, as already discussed in \Cref{lem::1GorMonomial: dimsocP(x)=dimtopI(x)}. So the top of $I(s(a))$ is at least $2$-dimensional. By Lemma \ref{biserialLemma}, it is exactly $2$-dimensional. Let $q_1$ and $q_2$ be the left-maximal paths ending at $s(a).$ Then
$P_0\cong P(s(q_1))\oplus P(s(q_2))$, and because $A$ is $1$-Gorenstein, these modules are uniserial.
Because ${p_1}^+$ is not contained in ${p_2}^+$ and vice versa, 
$q_i$ ends with ${p_i}^+$ for $i=1,2.$ Let $p$ be a maximal path such that $q_i=r_ip$ for some $r_i.$  

 
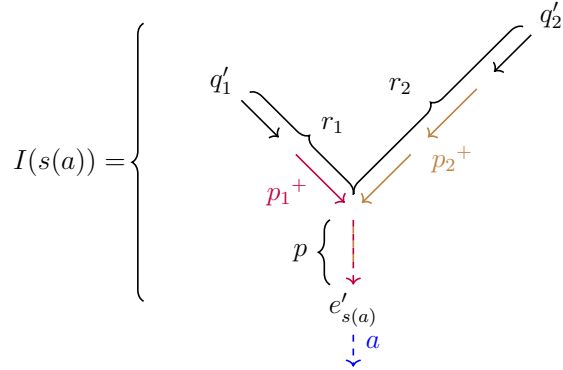
{


\centering
\resizebox{0.5\linewidth}{!}{
    \input{pictures/deg_in=2iffdeg_out=2_picture2}}
    

}

Then according to \Cref{lem::descriptionFirstSyzygyMonomial}, $r_1-r_2$ and $q_1a=r_1pa$ generate $\Omega^1(I(s(a)))=\ker(d_0)$ as a right $A$-module, and none of them are in the radical of this module. Thus, $\top \Omega^1(I(s(a)))\cong S(t(r_1))\oplus S(t(a))= S(t(r_1))\oplus S(x)$ and so $P^1\cong P(t(r_1))\oplus P(x).$ This is what we wanted to show.

A dual argument shows that $\deg^{\text{in}}(x)=2$ and $\deg^{\text{out}}(x)\le 1$ is also impossible. This concludes the proof.

\end{proof}

The following lemma tells us that at a degree $4$ vertex, the relations are fixed and the algebra looks locally like a gentle algebra.

\begin{lemma}
\label{lem::nolongrelLemma}
    Let $A=kQ/I$ be a $2$-Gorenstein monomial algebra. Let $x$ be a vertex with two incoming and two outgoing arrows, $a_1,a_2$ and $b_1,b_2,$ respectively. Then we can choose the labeling so that $a_1b_2, a_2b_1\in I$, and the other paths of length two, $a_1b_1$ and $a_2b_2$, are not contained in any minimal relations.


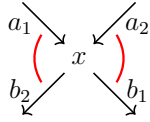
\centering
\resizebox{0.15\linewidth}{!}{
    \input{pictures/degree4vertex}}
    

\end{lemma}
\begin{proof}

    First, we show that $|\{a_ib_1, a_ib_2\}\cap I|\le 1.$ If this was not true, $a_i$ would be a right-maximal path, so the socle of $P(s(a_i))$ contained $S(x)$ as a direct summand. Therefore, the injective envelope of $P(s(a_i))$ contained $I(x)$ as a direct summand. However, $I(x)$ is not uniserial, and hence, not projective-injective. This contradicts the $1$-Gorenstein assumption on $A.$ 
    Dually, $|\{a_1b_i, a_2b_i\}\cap I|\le 1.$ So we can choose the labeling such that $a_1b_1,a_2b_2\notin I.$

    Next, we show that $ a_1b_2\notin I$ is not possible. For the sake of contradiction, assume it is true. Let $a_1b_2q\notin I$ be a right-maximal path. Then the injective envelope of $P(s(a_1))$ contains $I(t(q))$ as a direct summand. It follows from the 1-Gorenstein assumption on $A$ that this must be uniserial, so $a_2b_2q\in I.$
    Since $a_2b_2\neq 0$, there exists a right-maximal path $a_2b_2r\notin I$ as well. Then $I(t(r))$ is a direct summand of the injective envelope of $P(s(a_2))$, so again by the 1-Gorenstein assumption, we obtain that $I(t(r))$ is uniserial, so $a_1b_2r\in I.$ 
    Hence, $r$ cannot be a left subpath of $q$ and vice versa. This implies, that there are at least three different right-maximal paths starting at $x$, starting with $b_1, b_2q$ and $b_2r$, respectively. This, however, contradicts \Cref{biserialLemma}. Therefore, $a_1b_2, a_2b_1 \in I$.
\vspace{10pt}

{
\centering
\resizebox{0.25\linewidth}{!}{
    \input{pictures/degree4vtxLemma_1}}
    
}

It only remains to show that $a_1b_1$ (and then by symmetry also $a_2b_2$) is not contained in any minimal path in $I.$
Since $A$ is $2$-Gorenstein, by Lemma \ref{biserialLemma}, there are exactly two left-maximal paths ending at $x,$ $p_1$ and $p_2$ with $p_i$ ending with the arrow $a_i$; and two right-maximal paths starting at $x,$ $z_1$ and $z_2$ with $z_i$ starting with $b_i.$ Then the projective cover of $I(x),$ $P_0\cong P(s(p_1))\oplus P(s(p_2)),$ and the injective envelope of $P(x)$, $I^0\cong I(t(z_1))\oplus I(t(z_2))$ are both projective-injective, so, $P(s(p_i))$ corresponds to the unique right-maximal path starting at $s(p_i)$, call it $u_i=p_i\Tilde{z_i},$ and $I(t(z_i))$ corresponds to the unique left-maximal path ending at $t(z_i),$ call it $w_i=\Tilde{p_i}z_i.$ We have $p_ib_i\notin I$, otherwise $I(x)$ would be a direct summand of the injective envelope of $P(s(p_i)),$ which contradicts $A$ being $1$-Gorenstein. Similarly, $a_iz_i\notin I$.

\vspace{10 pt}

{
\centering
\resizebox{0.3\linewidth}{!}{
    \input{pictures/degree4vtxLemma_2}}
    
}

\vspace{10pt}

So we have that $$P(x)\xrightarrow{d_1} P(s(p_1))\oplus P(s(p_2))\xrightarrow{d_0} I(x)\longrightarrow 0$$ is a minimal projective presentation of $I(x)$ and
$$0\longrightarrow P(x)\xrightarrow{d^0} I(t(z_1))\oplus I(t(z_2))\xrightarrow{d^1} I(x)$$
is a minimal injective presentation of $P(x).$ Because $A$ is assumed to be $2$-Gorenstein, $I(x)$ must have projective dimension $1$ and $P(x)$ must have injective dimension $1,$ equivalently, $d_1$ must be a monomorphism and $d^1$ must be an epimorphism. Using the notation of \Cref{subsec::notationMonomial} where we provided an explicit basis and the right $A$-module structure of indecomposable injective and projective $A$-modules, we can spell out the maps in the above sequences: 
$$d_0: (s_1,s_2)\mapsto p_1'\cdot s_1+p_2'\cdot s_2, \textcolor{white}{space} d_1: s\mapsto (p_1s,-p_2s),$$
where $s_i$ is a path starting at $s(p_i)$ and $s$ is a path starting at $x,$ and 
$$d^0: s\mapsto z_1'\cdot s+z_2'\cdot s, \textcolor{white}{space} d^1: (s_1',s_2')\mapsto \Tilde{p_1}'\cdot \hat{s_1}-\tilde{p_2}'\cdot \hat{s_2},$$
where $s$ is again a path starting at $x$, $s_i$ is a path ending at $t(z_i),$ and $\hat{s_i}$ is the unique path such that $w_i=\hat{s_i}s_i.$
Using these, we can establish that $d_1$ is injective and $d^1$ is surjective if and only if $p_iz_i\notin I$ for both $i=1$ and $i=2.$ This is equivalent to $w_i=u_i=p_iz_i$, which is the same as $p_i=\Tilde{p_i} $ and $z_i=\Tilde{z_i}.$  

Now, if there would be a non-zero path $p\in I$ such that $p$ is minimal in $I$ and  $p=q_1a_1b_1q_2$ for some paths $q_1$ and $q_2,$ then by minimality of $p$ we had $q_1a_1\notin I$ and $b_1q_2\notin I,$ so there existed some paths $r_1$ and $r_2$ such that $p_1=r_1q_1a_1$ and $z_1=b_1q_2r_2.$ This, however, contradicts $w_1=p_1z_1,$ since $p_1z_1=r_1pr_2\in I$ but $w_i\notin I$ by definition. So such a $p$ cannot exist, and that was to prove. 

\end{proof}

\begin{remark}
\label{rmk::2GorMon_pdimIv=1}
    A consequence of the proof of the last lemma is that if we have a $2$-Gorenstein monomial algebra and $x$ is a degree 4 vertex then $\pdim I(x)=1$ and $\idim P(x)=1$.
\end{remark}

With this, we obtained some necessary conditions for a monomial algebra to be $2$-Gorenstein. These, however, are not sufficient. Indeed, every Nakayama algebra satisfies the conditions (1), (2) and (3) from \Cref{thm::classification2GorMonomial} but not every Nakayama algebra is $2$-Gorenstein. The missing piece for our classification result is described in the next lemma. 

\begin{lemma}

\label{endorstartLemma}
    Let $A=kQ/I$ be a $2$-Gorenstein monomial algebra. If an arrow $\alpha\in Q_1$ is contained in a minimal path in $I$ (i.e. a minimal relation) then $\alpha$ has to be the start or the end of a minimal path in $I.$
    
\vspace{5 pt}

{
\centering
\resizebox{0.3\linewidth}{!}{
    \input{pictures/startOrEndOfMinimalRel_picture1}}
    
}

\vspace{10pt}
\end{lemma}
\begin{proof}
For a visualisation of this proof, we refer to \Cref{fig:startOrEndOfMinimalRel_picture2}.
    Lemma \ref{lem::nolongrelLemma} tells us that every arrow starting or ending at a degree $4$ vertex is either the start or the end of a minimal path of length two contained in $I.$ So we only need to show the statement for arrows not connected to a degree $4$ vertex. 
    
    Assume for the sake of contradiction that the statement is not true. Then there exists an arrow $\alpha\in Q_1$ such that it is neither the start nor the end of a minimal path in $I$ but there exists a path $p\in I$ which is minimal in $I$ and is of the form $p=p_1\alpha p_2$, where $p_1$ and $p_2$ are paths with a non-zero length.   
    Let us denote $s(\alpha)=x$ and $t(\alpha)=y.$  \Cref{lem::nolongrelLemma} implies that the inner vertices of $p$ cannot have degree $4$, so $\deg(x)\neq 4 \neq \deg(y).$ 
    To arrive at a contradiction, we need to look at a minimal injective presentation of $P(y).$ 
    We use the usual notation for this:
    $$0\rightarrow P(y) \rightarrow I^0 \rightarrow I^1.$$
    Since $A$ is a string algebra and $\deg^{\text{out}}(y)=1,$ $P(y)$ is a uniserial module corresponding to the unique right-maximal path starting at $y,$ called $p_y$. 
    Denote the end vertex of $p_y$ by $j.$ So $I^0\cong I(j)$ is the injective envelope of $P(y),$ and since $A$ is $1$-Gorenstein, $I(j)$ is also uniserial, and thus, corresponds to the unique left-maximal path $i_j$ ending at vertex $j$. 
    Note that $i_j$ must end with the path $\alpha p_y$. This is because $p_y\notin I$ and we assumed that no minimal path in $I$ starts with $\alpha$, so $\alpha p_y\notin I.$
    This implies that $\soc(\Sigma(P(y)))\cong S_x,$ so $I^1\cong I(x).$ 
    Because $x$ is not a degree $4$ vertex but there must be an arrow ending at $x$ as length$(p_1)\ge 1,$ there is exactly one arrow ending at $x.$ As $A$ is a string algebra, this means that $I(x)$ is uniserial, and corresponds to the unique left-maximal path $i_x$ ending at $x.$ 
    \\
    
    \textbf{Claim:} $\pdim I^1 \ge 2.$\\
    
    It is clearly enough to show this to obtain a contradiction to the assumption that $A$ is $2$-Gorenstein. This would imply that no such arrow $\alpha$ can exist, which concludes our proof. 
    
    Consider a minimal projective presentation of $I(x).$ 
    Assume that $i_x$ starts at vertex $k.$ Since $A$ is $1$-Gorenstein, $P_0\cong P(k)$, the projective cover of $I(x),$ must be uniserial. In other words, there must be a unique right-maximal path $p_k$ starting at $k.$ Moreover, we can show that $p_k=i_x\alpha r$ where $r$ is a strict left-subpath of $p_2.$ First, $i_x\notin I$ and $\alpha$ is not the end of a minimal relation, so $i_x\alpha\notin I.$ 
    Second, $i_x\alpha p_2$ contains $p$ as a right-subpath. Indeed, since $p$ was chosen to be a minimal relation, $p_1$ is not contained in $I.$ Thus, $p_1$ is a right-subpath of $i_x,$ and so $i_x\alpha p_2 $ contains $p$ as a subpath. This forces $i_x\alpha p_2\in I.$ 
    Since no inner vertex of $p$ has degree 4, $r$ must be a left subpath of $p_2.$ As $i_x\alpha p_2\in I$, $r$ is a strict subpath of $p_2.$ Thus, we can conclude that $p_k=i_x\alpha r,$ which implies that $\Omega(I(x))\cong M(r).$ But since $p_2\notin I$, $r$ is not right-maximal and $M(r)$ is not projective. Hence, $\pdim I^1=\pdim I(x)>1.$ This concludes the proof. 
\vspace{10 pt}
{

\begin{figure}[h]

\centering
\resizebox{0.8\linewidth}{!}{
    \input{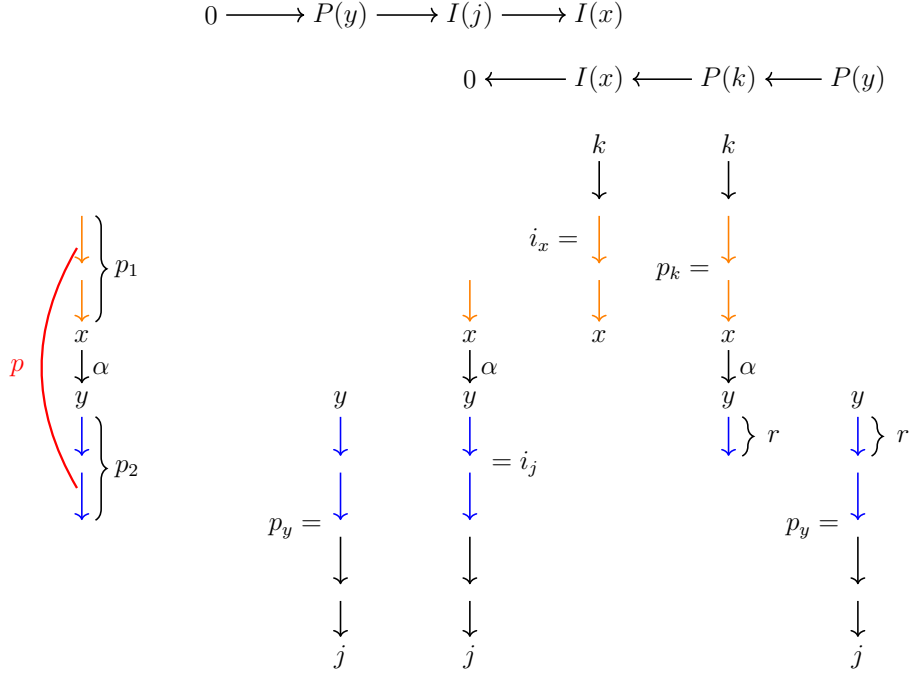}}
    
    \caption{Resolutions used in the proof of \Cref{endorstartLemma}}
    \label{fig:startOrEndOfMinimalRel_picture2}
\end{figure}

}

\end{proof}

As a final step, we prove that the conditions found so far are in fact sufficient. This way we achieve an explicit combinatorial characterisation of the $2$-Gorenstein property for monomial algebras. 

\begin{proof}[Proof of \Cref{thm::classification2GorMonomial}]
    By Lemma \ref{biserialLemma}, \ref{deg4vertexLemma}, \ref{lem::nolongrelLemma}, and \ref{endorstartLemma}, it only remains to prove the 'if' direction. Assume that (1)-(4) hold. 
    To prove the $2$-Gorenstein property, we will choose an arbitrary vertex $x$ and understand the minimal injective presentation of the $A$-module $P(x)$, which we denote by 
    $$0\longrightarrow P(x)\longrightarrow I^0 \longrightarrow I^1.$$
We distinguish two cases based on how the algebra $kQ/I$ looks around $x$. \\

    \textbf{Case I.} Assume that $\deg(x)=4.$ 
    
    Label its adjacent arrows according to (3). Conditions (1)-(3) imply that $A$ is a string algebra and so there are exactly two right-maximal paths starting at $x$, $z_1$ and $z_2$, which start with the different arrows $b_1$ and $b_2,$ respectively.
    Condition (3) implies that a right-maximal path not contained in $I$ cannot end at a vertex of degree $4.$ Indeed, let $y$ be a vertex of degree $4$ and label its incoming and outgoing arrows according to condition (3) by $a_1',a_2'$ and $b_1',b_2'$. Then if there is a path $p\notin I$ ending with the arrow $a_i',$ then $pb_i'\notin I$ either. Thus, $t(z_i)$ for $i=1,2$ cannot have degree $4.$ Then (1)-(3) imply that $I(t(z_i))$ is uniserial and  corresponds to the unique left-maximal path $w_i\notin I$ ending at $t(z_i).$ As $z_i\notin I,$ $z_i$ is a right-subpath of $w_i$. Similarly as before, condition (3) also implies that no left-maximal path starts at a degree 4 vertex, and thus, $w_i\neq z_i$ and $w_i$ ends with $a_iz_i.$ So we can write $w_i=r_ia_iz_i$ for some path $r_i.$
    With these observations, we can spell out the injective presentation of $P(x):$

    $$0\longrightarrow P(x) \xrightarrow{d^0} I^0\cong I(t(z_1))\oplus I(t(z_2)) \xrightarrow{d^1} I^1\cong I(x). 
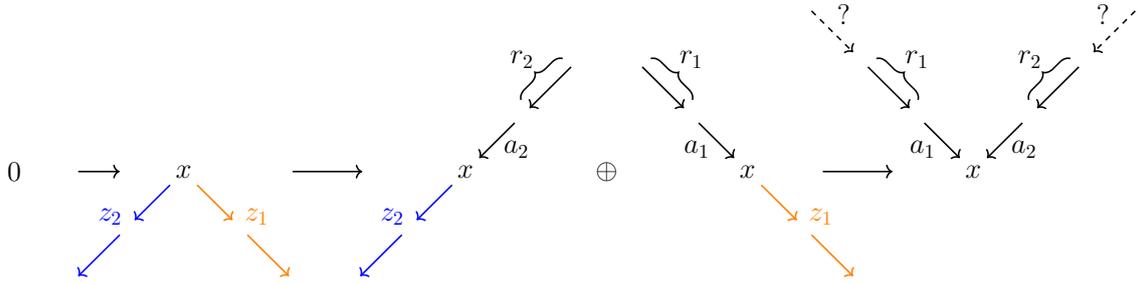
$$
    
\vspace{10 pt}
{


\centering
\resizebox{\linewidth}{!}{
    \input{pictures/proofOf2GorMon1}}
    

}

\vspace{5 pt}

    Next, we show that $I(t(z_i)),$ and hence $I^0$, is projective-injective. 
    As we established, $s(w_i)$ cannot have degree $4.$ Using (1) and (2), we obtain that $\deg^{\text{out}}(s(w_i))=1$ and $P(s(w_i))$ is uniserial as $A$ is a string algebra. So there is a unique right-maximal path starting at $s(w_i).$ This must start with $w_i.$ Because $z_i$ is right-maximal, so is $w_i$, which implies that $I(t(z_i))\cong P(s(w_i))\cong M(w_i)$ is projective-injective. 

    Finally, we would like to show that $\pdim I(x)\le 1.$ For this, we need to understand $I(x)$ better.   Assumptions (1) and (3) ensure that there are exactly two left-maximal paths ending at $x$. We will show that these are $r_1a_1$ and $r_2a_2.$ 
    By the left-maximality of $w_i,$ $aw_i\in I$ for every arrow $a.$ Condition (3) states that $a_ib_i$ cannot be contained in any minimal path in $I,$ so already $ar_ia_i\in I$ must hold for every arrow $a.$ As $r_ia_i\notin I,$ $r_ia_i$ is a left-maximal path ending at $x.$ 
    This implies that $d^1$ is surjective, and therefore, $\idim P(x)=1.$ Because $I^0$ is projective, this sequence must be a minimal projective resolution for $I(x),$ which implies $\pdim I(x)=1.$ 

    Note that we have not used assumption (4) so far, but we will need it for our arguments in the following case.\\
    
    \textbf{Case II.} Let $\deg(x)\neq 4$. 
    
    Then (1) and (2) imply that $\deg^{\text{in}}(x),\deg^{\text{out}}(x)\le 1.$ 
    $A$ is a string algebra, therefore $P(x)$ is uniserial, corresponding to the unique right-maximal path starting at $x$. Call this $p_x$ and denote the last vertex of this path by $y.$ Then $\soc(P(x))\cong S(y),$ so $I^0\cong I(y)$. By condition (3), as $y$ is the end vertex of a right-maximal path, $\deg(y)\neq 4$. Then (1) and (2) imply that there is at most one arrow ending at $y,$ so $I(y)$ is uniserial and corresponds to the unique left-maximal path ending at $y,$ $i_y.$ For a visualisation of these paths we refer to the illustration at the end of this proof. An injective presentation of $P(x)$ looks as follows:
    $$0\longrightarrow P(x) \xrightarrow{d^0} I^0\cong I(y) \xrightarrow{d^1} I^1. $$
    
    We can show that $I(y)$ is always projective. To see this, observe that since $i_y$ ends with the path $p_x,$ which is a right-maximal path, therefore, $i_y$ is a right-maximal path as well. By definition, $i_y$ is left-maximal, so condition (3) prevents its starting vertex $s(i_y)$ to have degree $4.$ Then by conditions (1) - (3), there is only one right-maximal path starting at  $s(i_y),$ namely $i_y$.  Thus,   
    $P(s(i_y))\cong I(y)\cong M(i_y).$

It remains to show that $\pdim I^1\le 1.$ If $I^1\cong 0$, this is obviously true. If $I^1\ncong 0$, then $i_y\neq p_x,$ and so there is an arrow $\beta\in Q_1$ ending at $x$ with $\beta p_x\notin I.$ Let us denote the start of this arrow by $u.$ We have  $I^1\cong I(u).$ 
A dual argument to Case I tells us that $\pdim I(u)=1$ if $\deg(u)=4.$ 

So it remains to consider the case when $\deg(u)\neq 4.$ By assumptions (1)-(3) this is equivalent to $I(u)$ being uniserial. 
Observe that in this case $\beta$ cannot be the start of a minimal relation. Indeed, assume that there is a path $q\notin I$ with $\beta q\in I$. Then as $p_x$ is the unique right-maximal path starting at $x$, $q$ must be a left subpath of $p_x.$ This would, however, imply $\beta p_x\in I,$ which is a contradiction to $I^1\ncong 0.$ So $\beta$ is indeed not the start of any minimal path in $I.$

Then by assumption (4), there must be a minimal path $p\in I$ which ends with $\beta.$ Using this, we can show that $\pdim I(u)=0.$ For this, we need to show that the unique left-maximal path ending at $u,$ $i_u,$ is also the unique right-maximal path starting at $s(i_u).$ If this would not be right-maximal, then there would be an arrow $a\in Q_1$ such that $i_u a\notin I.$ The only arrow starting at the same vertex where $i_u$ ends is $\beta$, so we have to show that $i_u\beta\in I.$ For this, consider $p^+,$ which we get from $p$ by discarding its last arrow, $\beta.$ Note that $t(p^+)=u$ and since $p$ is minimal in $I,$ $p^+\notin I.$ Thus, $p^+$ must be a right subpath of the unique left-maximal path ending at $u,$ $i_u$. This implies that $i_u\beta$ contains the path $p\in I,$ so is itself contained in $I$. Thus, $i_u$ is right-maximal. Since it is also left-maximal, $s(i_u)$ cannot have degree $4$ by condition (3). So $P(s(i_u))\cong M(i_u)\cong I(u)$ is projective-injective.

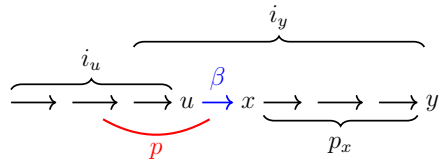
\vspace{10 pt}
{


\centering
\resizebox{0.4\linewidth}{!}{
    \input{pictures/proofOf2GorMon2}}
    

}

In conclusion, for every choice of $x$ $\pdim I^0=0$ and $\pdim I^1\le1$. This is exactly the criterion for $A$ to be $2$-Gorenstein.

\end{proof}

As a direct consequence of \Cref{thm::classification2GorMonomial} and \Cref{rmk::2GorMon_pdimIv=1}, we obtain the main result of \cite{M19}.

\begin{corollary}
\label{cor::2GorMonAlgDomdimAtLeast2}
Let $A$ be a 2-Gorenstein monomial algebra with dominant dimension at least 2. Then $A$ is a Nakayama algebra.
\end{corollary}

\newpage
\section{Reducing monomial algebras to Nakayama algebras }
\label{chap::reduction}

Assume that $A=kQ/I$ is a 2-Gorenstein monomial algebra with a degree 4 vertex $v$. Then it  satisfies condition (3) from \Cref{thm::classification2GorMonomial}.
More precisely, there are two incoming arrows at $v$: $a_1$ and $a_2,$ and two outgoing arrows: $b_1$ and $b_2,$ such that $a_1b_2,a_2b_1\in I$ but for $i=1,2$ $a_ib_i$ is not contained in any minimal relations. 
Then we can create a new (not necessarily connected) monomial algebra $B$ from $A$ by cutting the quiver at vertex $v$. We do this by replacing $v$ with two new vertices $v_1$ and $v_2$ such that both of them have degree two, each of them is connected to two of the arrows adjacent to $v$ so that we keep the relations, i.e. $t(a_1)=v_1=s(b_2)$ and   $t(a_2)=v_2=s(b_1).$ All the other vertices, arrows, and relations stay the same. The point is that this way we do not lose any relations since $a_ib_i$ were assumed not to be part of any minimal relation.

Conversely, we obtain $A$ from $B$ by gluing the vertices $v_1$ and $v_2.$ 
Note that in this gluing procedure it is important that the resulting algebra $A$ is again finite-dimensional; this is not guaranteed for an arbitrary monomial algebra $B$ with vertices $v_1$ and $v_2.$

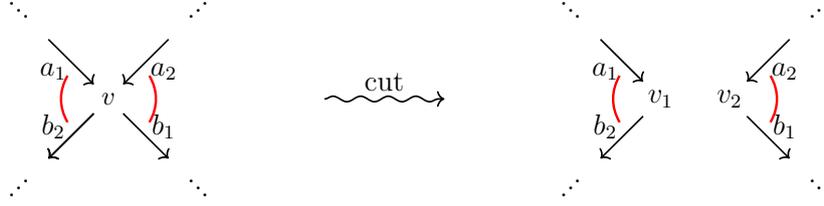
{

\begin{figure}[h]

\centering
\resizebox{0.8\linewidth}{!}{
    \input{pictures/cutting_picture}}
    
\end{figure}

}

Surprisingly, these cutting and gluing procedures preserve many properties of the monomial algebra we are interested in. 

\begin{theorem}
\label{thm::cutting}
Let $A$ and $B$ be as described before. In particular, we assume that $A$ is 2-Gorenstein. Then the following hold:

\begin{enumerate}[\rm (1)]
    \item for any $n\ge 2$, $A$ is $n$-Gorenstein if and only if $B$ is. 
    \item for any $n\ge 1$, $A$ is $n$-Iwanaga-Gorenstein if and only if $B$ is.
    \item $A$ is Auslander-Gorenstein if and only if $B$ is.
    \item $\psi_A$ is well-defined and bijective if and only if $\psi_B$ is.
\end{enumerate}    
\end{theorem}

This is demonstrated in \Cref{ex::CuttingExample}. \Cref{thm::cutting} allows us to reduce the problem of classifying all Auslander-Gorenstein monomial algebras, to the problem of classifying all Auslander-Gorenstein Nakayama algebras. 
Specifically, take a $2$-Gorenstein monomial algebra $A$ - we have already classified such algebras. Then according to \Cref{thm::classification2GorMonomial}, we can apply the cutting procedure at every degree 4 vertex to obtain a (not necessarily connected) Nakayama algebra $B$ which is Auslander-Gorenstein if and only if $A$ is. 

\begin{corollary}
\label{cor::redofQ1}
    We can reduce the problem of classifying all Auslander-Gorenstein monomial algebras to classifying all Auslander-Gorenstein Nakayama algebras.
\end{corollary}

The importance of this reduction lies in the fact that the class of Nakayama algebras is much more restrictive than the one of monomial algebras. For instance, for Nakayama algebras the underlying quiver $Q$ must be a linearly oriented line or a linearly oriented cycle, whereas for monomial algebras, $Q$ can be any finite quiver. We refer to \Cref{chap::Nakayama}, where we recall the definition of Nakayama algebras and investigate their relation to the $n$-Gorenstein properties. Thanks to the previous theorem, many of the results proven for Nakayama algebras will be immediately inherited to the  monomial case. These are spelled out in Chapter \ref{chap::MainResultsForMonomial}.

\subsection{Proof of Theorem \ref{thm::cutting}}
\label{subsec::proofOfCuttingThm}
The proof of this theorem relies heavily on the next lemma which states roughly that the minimal injective resolutions of certain uniserial modules (most importantly, of uniserial projectives) stay the same after applying the cutting procedure. 

Consider $A,$ $B$ and $v$ as described in the beginning of \Cref{chap::reduction}. In this section, we use the notation $P_A(x)$ to indicate the indecomposable projective $A$-module at vertex $x,$ and similarly, we denote the indecomposable projective $B$-module at vertex $x$ by $P_B(x).$ We define the notation $I_A(x)$ and $I_B(x)$ analogously. Let $r_1$ and $r_2$ be the two left-maximal paths ending at $v$ such that the last arrow of $r_i$ is $a_i.$ Similarly, let $z_1$ and $z_2$ be the two right-maximal paths starting at $v$ such that the first arrow of $z_i$ is $b_i.$ Moreover, let $s_i$ be the starting vertex of $r_i$ and $t_i$ the end vertex of $z_i$. We should have the following picture in mind:
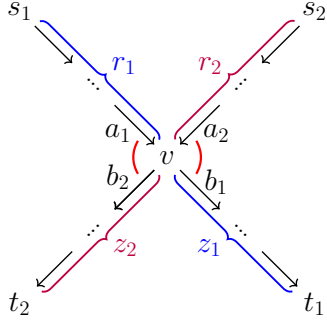
\vspace{10pt}

{\centering

    \resizebox{0.3\linewidth}{!}{
    \input{pictures/section5.1_picture}
    }
    

}

Note that $r_iz_i\notin I$, so $M(r_iz_i)$ is a projective-injective module for $i=1,2.$ This is a consequence of \Cref{thm::classification2GorMonomial} (3) and the 2-Gorenstein property of $A.$ 

\begin{lemma}
\label{lem::same_resol}
    Let $A=kQ/I$, $B$ and $v$ be as described before. In particular, we assume that $A$ is 2-Gorenstein. Let $p\notin I$ be a path in $Q$ such that one of the following three cases occurs:
    
    \begin{enumerate}[\rm (1)]
        \item $v\notin p$
        \item $p=r_iq_i$ for some path $q_i$ such that $q_i$ contains $v$ only as its starting vertex.
        \item $p=q_iz_i$ where $q_i\neq e_v$, $q_i\neq r_i$ and $q_i$ contains $v$ only as its end vertex.
    \end{enumerate}
    Let 
    $$0 \longrightarrow M(p) \longrightarrow \bigoplus_{\ell=1}^{m_0} I(v_{0,\ell}) \longrightarrow  \bigoplus_{\ell=1}^{m_1} I(v_{1,\ell}) \longrightarrow \ldots $$
    be a minimal injective resolution of $M(p)$ in mod(A). Here $m_j\in \mathbb{N}_0$ and the $v_{j,\ell}$ stand for vertices of $Q$. Then we obtain a minimal injective resolution of the $B$-module $M(p')$ via
    
    $$0 \longrightarrow M(p') \longrightarrow \bigoplus_{\ell=1}^{m_0} I(v'_{0,\ell}) \longrightarrow  \bigoplus_{\ell=1}^{m_1} I(v'_{1,\ell}) \longrightarrow \ldots $$
    where $p'$ is defined to be the following path in $B:$

    $$p'=
    \begin{cases}
    p, \text{ in Case {\rm (1) }}  \\
    e_{v_{1}},  \text{ in Case {\rm (2)} if }p=r_2\\
    e_{v_{2}}, \text{ in Case {\rm (2)} if }p=r_1\\
    q_i,  \text{ in Case {\rm (2)} if }q_i\neq e_v \\
    q_i,  \text{ in Case {\rm (3)}}
    \end{cases}
    $$
    and 

    \begin{equation}
    \label{socles_in_inj_resol}
    v'_{j,\ell}=\begin{cases}
    v_{j,\ell},\ \text{ if }v_{j,\ell}\notin\{v,t_1,t_2\}\\
    v_2 \text{ or } t_2, \ \text{ if }v_{j,\ell}=t_2\\
    v_1\text{ or } t_1, \ \text{ if }v_{j,\ell}=t_1\\
    v_1\text{ or } v_2,\ \text{ if }v_{j,\ell}=v.
    \end{cases}
    \end{equation}
    
\end{lemma}

\begin{remark}
\label{rmk:: r_iContainsv}
Every path $p=r_iq_i$ in $A$ can be written as $p=r_jq_j$ where $q_j$ contains $v$ only as its starting vertex. Here $i,j\in \{1,2\}.$ This is a consequence of Lemma \ref{lem::z_iContainsV2X}. Thus, this extra assumption for $p$ in Case (2) is not an additional restriction, it is only necessary to choose the representation $r_jq_j$ for $p$ in order to get the correct definition for $p'.$ The same is true in Case (3).
\end{remark}

\begin{lemma}
\label{lem::z_iContainsV2X}
    If $z_i$ contains $v$ at least twice for an $i\in \{1,2\}$, then $z_i=qz_j$ and $r_j=r_iq.$ Here $i\neq j\in\{1,2\}$ and $q$ is a path which contains $v$ exactly twice (as its start and end vertex). Moreover, $z_j$ and $r_i$ contain $v$ only once.
    
    An example of this situation can be seen in Example \ref{ex::Section42GorMon}.
\end{lemma}
\begin{proof}
  Since $\deg(v)=4,$ by condition (3) in \Cref{thm::classification2GorMonomial}, no right-maximal path can end at $v$. So if $z_i$ contains $v$ more than once, $z_i=b_i h a_jb_j r$ for some non-zero paths $h$ and $r$ and $j\in\{1,2\}.$ Observe that $i\neq j.$ Otherwise, $b_i h a_ib_i\notin I,$ so since $a_ib_i$ is not contained in any minimal relation, $(b_iha_i)^n\notin I$ for any power $n\ge 1$ of the cycle $b_iha_i.$ This contradicts that $A$ is a finite-dimensional algebra, so $i\neq j.$ 
By the same argument, $h$ and $r$ cannot contain $v,$ i.e. $z_i$ can contain $v$ at most twice.
Since $a_jb_j$ is not contained in any minimal relation, $b_jr$ is already right-maximal. The only such path starting with $b_j$ is $z_j.$
Thus, $q=b_iha_j$ and $z_i=qz_j$ indeed. As $q\notin I$ and $a_ib_i$ is not contained in any minimal relation, $r_iq\notin I.$ Because this is left-maximal and ends with the arrow $b_j,$ $r_iq=r_j.$

\end{proof}

\begin{example} 
\label{ex::CuttingExample}
Let $A=kQ/I$ be the monomial algebra where $Q$ is the quiver on the left-hand side in the following picture and $I=\langle ca_1,a_2^2, a_1b_2 \rangle.$ Note that this is the same algebra as $A_2$ from \Cref{ex::Section42GorMon}.
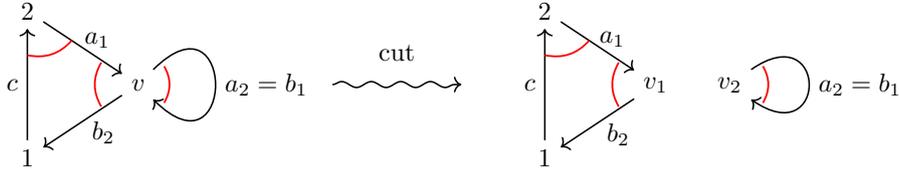
\vspace{10pt}

{\centering

    \resizebox{0.8\linewidth}{!}{
    \input{pictures/Section5_example}
    }
    

}

\vspace{10pt}
According to \Cref{thm::classification2GorMonomial} (or the explicit calculations in \Cref{ex::Section42GorMon}), $A$ is 2-Gorenstein. Applying the cutting procedure to vertex $v,$ we obtain $B,$ the non-connected quiver algebra on the right. Recall that in \Cref{ex::Section42GorMon}, we already described indecomposable projective and injective modules over $A.$ The indecomposable projective and injective modules over $B$ look as follows:

\vspace{10pt}
{

\centering

    \resizebox{\linewidth}{!}{
    \input{pictures/Section5_exampleProjInj}
    }
}
\\
First, consider a minimal injective resolution of $P_A(1)=M(c).$

\vspace{10pt}
{
\centering

    \resizebox{\linewidth}{!}{
    \input{pictures/new4}
    }
}
\\
In pictures, 

\vspace{10pt}
{

\centering

    \resizebox{\linewidth}{!}{
    \input{pictures/new}
    }
}

Then, consider the minimal injective resolution of $P_B(1)=M(c).$ Note that the path $c$ satisfies (1) from \Cref{lem::same_resol}, so $c'=c.$
\vspace{10pt}
{

\centering

    \resizebox{\linewidth}{!}{
    \input{pictures/new3}
    }
}
\\
In pictures,

\vspace{10pt}
{

\centering

    \resizebox{\linewidth}{!}{
    \input{pictures/new2}
    }
}

 Note that all the cosyzygy modules correspond to paths which satisfy either (1) or (2) from \Cref{lem::same_resol}, moreover, if $\Sigma^i(P_A(1))\cong M(p)$ then $\Sigma^i(P_B(1))\cong M(p')$ for $i=1,2,3.$ More concretely, $e_{v_1}=(a_1a_2)'$ since $a_1a_2$ belongs to Case (2), $e_2=(e_2)'$ as $e_2$ belongs to Case (1), and $b_2=(a_1a_2b_2)'$ since $a_1a_2b_2$ belongs to Case (2). This is not a coincidence but a general phenomenon, as we will see in the proof of \Cref{lem::same_resol}. It is easy to check that the above resolutions of $M(c)$ and $M(c')$ satisfy the statement of \Cref{lem::same_resol}. The same is clearly true for the minimal injective resolutions of the $\Sigma^i(P_A(1))$ and $\Sigma^i(P_B(1))$. 

 As another example, consider a minimal injective resolution of the $A$-module $M(a_1)\cong \begin{smallmatrix}
     2\\v
 \end{smallmatrix}.$ Note that $\Sigma(M(a_1))\cong M(a_1)$, hence the resolution is infinite:  
 
 $$0\longrightarrow M(a_1)\longrightarrow I_A(v) \longrightarrow I_A(v) \longrightarrow \ldots $$
 In pictures,
  $$0\longrightarrow \ \begin{matrix}
     2\\v
 \end{matrix} \ \longrightarrow {\begin{matrix}
     \textcolor{white}{22222222}\textcolor{blue}{2}\\2\textcolor{white}{222}\textcolor{blue}{v}\\v
 \end{matrix}} \longrightarrow \begin{matrix}
     \textcolor{white}{22222222}\textcolor{purple}{2}\\\textcolor{blue}{2}\textcolor{white}{222}\textcolor{purple}{v}\\\textcolor{blue}{v}
 \end{matrix}  \longrightarrow \ldots $$

We have that $a_1$ is a path satisfying case (2) from \Cref{lem::same_resol}. Thus, $a_1'=e_{v_2}$. Consider a minimal injective resolution of the $B$-module $M(e_{v_2}).$ Note that $\Sigma(M(e_{v_2}))\cong M(e_{v_2})$, hence the resolution is infinite:  

 $$0\longrightarrow M(e_{v_2})\longrightarrow I_B(v_2) \longrightarrow I_B(v_2) \longrightarrow \ldots $$
\\
 In pictures,
  $$0\longrightarrow \ \begin{matrix}
     v_2
 \end{matrix} \ \longrightarrow \begin{matrix}\textcolor{blue}{v_2}\\v_2 \end{matrix} \longrightarrow \begin{matrix}\textcolor{purple}{v_2}\\\textcolor{blue}{v_2} \end{matrix} \longrightarrow \ldots $$

 \vspace{8pt}

In conclusion, the above resolutions of $M(a_1)$ and $M(a_1')$ also have the form described in \Cref{lem::same_resol}. 

Recall that we already described a minimal injective resolution of $P_A(v)$ and established $\idim I_A(v)=1$ in \Cref{ex::Section42GorMon}. Calculation gives that $\idim I_A(1)=3=\pdim I_B(1)$ (which also corresponds to a dual statement of \Cref{lem::same_resol}.) Thus, both $A$ and $B$ are Auslander-Gorenstein. Moreover, computation gives $\psi_A=(1\ 2)(v)$ and $\psi_B=(1)(2\ v_1)(v_2)$. In conclusion, both are bijective, and \Cref{thm::cutting} holds in this case.




\end{example}

\begin{proof}[Proof of Lemma \ref{lem::same_resol}]
We proceed by induction on $j.$
First, we show that \Cref{socles_in_inj_resol} holds for all $v'_{0,\ell}.$  The injective envelope of the uniserial $A$-module $M(p)$ is $I_A(t(p))$, i.e. $m_0=1$ and $v_{0,1}=t(p).$ The injective envelope of the $B$-module $M(p')$ is $I_B(t(p'))$. Thus, $v'_{0,1}=t(p').$ Going through all three possible cases for $p$, it is easy to check that $v'_{0,1}$ satisfies \Cref{socles_in_inj_resol}.

Let $w_1$ be the left-maximal path in $A$ ending at $t(p)$ which ends with path $p$. This is unique as $A$ is $2$-Gorenstein and therefore, a string algebra. If $\deg^{\text{in}}(t(p))=2$ then there is one further left-maximal path ending at $t(p),$ call it $w_2.$ Then,
$$0 \rightarrow M(p) \rightarrow I_A(t(p))\rightarrow \Sigma(M(p))\cong M(p_1)\oplus M(p_2)\rightarrow  0$$
where $p_1=0$ if $w_1=p$. Otherwise, $p_1$ is the path such that $w_1=p_1c_1p$ where $c_1$ is an arrow. Similarly,  $p_2=0$ if $w_2$ does not exist and otherwise, $p_2$ is the path such that $w_2=p_2c_2$ where $c_2$ is an arrow. This situation is depicted in \Cref{fig:same_resol_picture1}. We can analogously define $g_1$ and $g_2$ as (possibly zero) paths in $B$ such that 
$$0 \rightarrow M(p') \rightarrow I_B(t(p'))\rightarrow \Sigma(M(p'))\cong  M(g_1)\oplus M(g_2)\rightarrow  0.$$ This is possible as $B$ is also a string algebra if $A$ is. Hence, we have $m_1=0, 1$ or $2$ and $v_{1,\ell}=t(p_{\ell})$ and $v'_{1,\ell}=t(g_{\ell})$ for $1\le \ell \le m_1$. Note that $g_{\ell}$ is always a right subpath of $p_{\ell}.$ This will be clearer when we go through the case distinction in the remainder of the proof. 

\vspace{10 pt}

\textbf{Claim.} For $\ell=1,2:$ $p_{\ell}$ satisfies (1), (2), or (3), and we have, $g_{\ell}=p_{\ell}'.$

\vspace{10 pt}
This claim is the basis of our induction argument. Assume the statement of this lemma holds until a $j=d\ge 0.$ By the claim, this induction assumption applies, in particular, for the resolutions of $M(p_{\ell})$ and $M(p_{\ell}')=M(g_{\ell}).$ We can use this to obtain the desired statement for $j=d+1,$ by noting that the $(d+1)$-st term in the minimal injective resolution of $M(p)$ (resp. $M(p')$) is equal to the $d$-th term in the minimal injective resolution of $M(p_1)\oplus M(p_2)$ (resp. $M(g_1)\oplus M(g_2)$).

So it remains to prove this claim. First, assume that $p$ belongs to Case (1), i.e. $v$ is not contained in $p$ and $p=p'$. 
If $v\notin p_{\ell},$ then $w_{\ell}$ is also a non-zero, left-maximal path in $B$ and $g_{\ell}=p_{\ell}.$ So $p_{\ell}$ belongs to Case (1) and $p_{\ell}'=g_{\ell}$ indeed. 

Next, assume that $v\in p_{\ell}.$ 
Due to condition (3) in \Cref{thm::classification2GorMonomial}, no left-maximal path starts at a degree 4 vertex, so 
$p_{\ell}=he_vq$ for some paths $h,q\notin I$ where $h$ has length at least one. We may choose $q$ such that it contains $v$ only as its starting vertex. Since $p$ belongs to Case (1), $q$ has length at least one as well.  Recall that $w_{\ell}$ is a left-maximal path by choice. Condition (3) in \Cref{thm::classification2GorMonomial} also implies that no minimal relation of length at least $3$ contains a degree $4$ vertex. Therefore, $h$ must be a left-maximal path already. The only two left-maximal paths ending at $v$ are $r_1$ and $r_2,$ so $h=r_i$ for $i=1$ or $i=2.$ This implies that $p_{\ell}$ belongs to Case (2). If ${\ell}=1$ then  $qc_1p$ only contains $v$ as its starting vertex, and so by construction of $B$, $qc_1p$ is a non-zero path in $B.$ Analogously, if ${\ell}=2,$ $qc_2$ is a non-zero path in $B$.
In both cases, $q$ must start with the arrow $b_i$ as $r_iz_i$ is the only right-maximal path starting at $s_i.$
Hence, $s(q)=s(b_i)=v_{i'}$ where $i'\in\{1,2\}$ such that $i\neq i'.$ We keep this notation for $i'$ throughout the remainder of the proof. Moreover, $qc_1p$ (resp. $qc_2)$ is a left-maximal path in $B$ ending at $t(p)$ as the only arrow ending at $v_{i'}$ is $a_{i'},$ and $a_{i'}b_i=0,$ which implies $a_{i'}qc_j=0.$ Thus, $g_{\ell}=q.$ Comparing this to the definition of $p_{\ell}',$ we get $p_{\ell}'=q=g_{\ell}.$

\Cref{fig:same_resol_picture1} illustrates the situation when $p$ belongs to Case (1), $p_1$ to Case (2) and $p_2$ to Case (1).

\begin{figure}[ht]
    \centering
    \resizebox{0.7\linewidth}{!}{
    \input{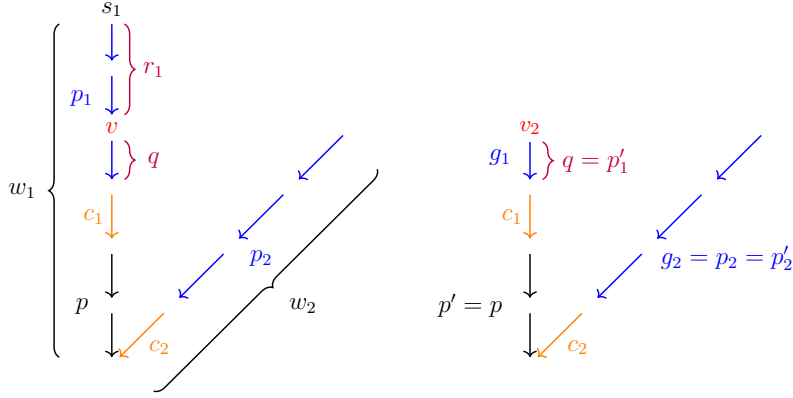}
    }
    \caption{On the left: $I_A(t(p))$, on the right: $I_B(t(p'))$, when $p$ and $p_2$ belong to Case (1) and $p_1$ to Case (2)}
    \label{fig:same_resol_picture1}
\end{figure}

Next, consider the case when $p$ belongs to Case (2), i.e. $p=r_iq_i$ and $q_i$ contains $v$ only once.  
Since $r_i$ is left-maximal, so is $p,$ and hence $p_1=0$ in this case.  
By definition, $s(p')=v_{i'}.$ In $B$, the only arrow ending at $v_{i'}$ is $a_{i'}$. If $q_i=e_v,$ then $p'=e_{v_{i'}}$ and we may assume that $g_1=0.$ If
$q_i\neq e_v$ then it must start with $b_i$ as $r_iz_i$ is the unique right-maximal path starting at $s_i.$ But $a_{i'}b_i=0$ and so $a_{i'}q_i=0$ which implies that $g_1=0$ in this case as well. So $m_1\le 1.$ 
Repeating the arguments from the previous case, we obtain that $p_2$ belongs to either Case (1) or Case (2), and in both cases, $p_2'=g_2.$ 

Figure \ref{fig:same_resol_picture2} illustrates the case when $p=r_1$ and $p_2$ belongs to Case (1).

{
\begin{figure}[h!]
\centering
    \resizebox{0.8\linewidth}{!}{
    \input{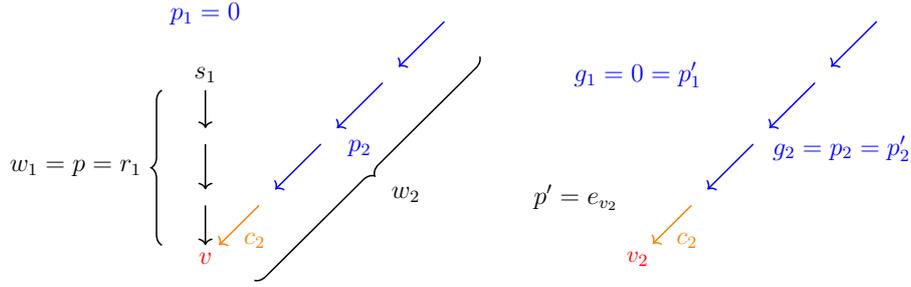}}
    
    \caption{On the left: $I_A(t(p))$, on the right: $I_B(t(p'))$, when $p=r_1$ and $p_2$ belongs to Case (1)}
    \label{fig:same_resol_picture2}
\end{figure}
}

Finally, assume that $p$ belongs to Case (3). By condition (2) and (3) in \Cref{thm::classification2GorMonomial}, if a right-maximal path in $A$ ends at a vertex then this vertex has in-degree one. Hence,  deg$^{\text{in}}(t_i)=1$ and thus, $p_2=0$. By definition, $p'=q_i$ so $t(p')=v_i.$ As $v_i$ has in-degree one, $g_2=0=p_2'.$ So $m_1\le 1.$
On the other hand, $p_1$ is the path such that $r_i=p_1cq_i$ for an arrow $c.$ We assumed that $q_i\neq r_i$
in Case (3), so such a $c$ exists and $p_1\neq 0.$ Here we used the fact that $I(t_i)=M(r_iz_i)$, so $w_1=r_iz_i$. 
Just as in the previous two cases, $p_1$ belongs either to Case (1) or Case (2) and $p_1'=g_1.$ This concludes the proof. 

We also provide a picture for Case (3) when $p_1$ belongs to Case (2), see \Cref{fig:same_resol_picture3}.
\\

{

\begin{figure}[h!]

\centering
\resizebox{0.8\linewidth}{!}{
\input{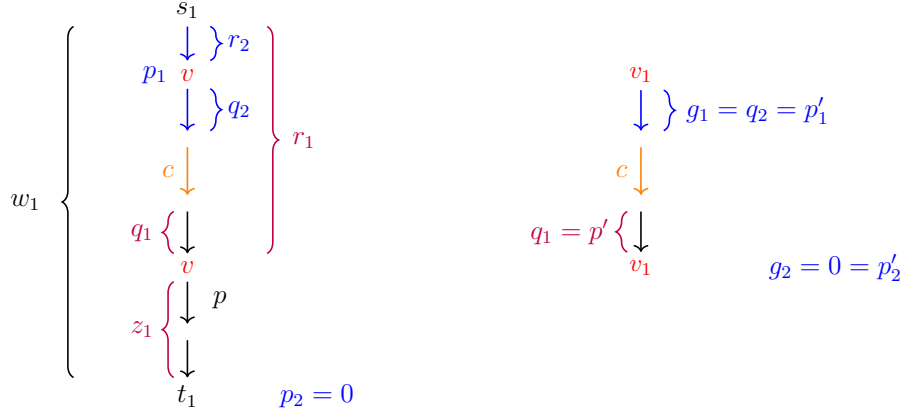}

}
    \caption{On the left: $I_A(t(p))$, on the right: $I_B(t(p'))$, when $p$ belongs to Case (3) and $p_1$ to Case (2)}
    \label{fig:same_resol_picture3}
\end{figure}
}

\end{proof}

\begin{proof}[Proof of \Cref{thm::cutting}] 
By the characterisation of $2$-Gorenstein monomial algebras, see \Cref{thm::classification2GorMonomial}, $A$ is $2$-Gorenstein if and only if $B$ is. 
Let us first deal with the minimal injective resolutions of non-uniserial indecomposable projective modules. Let $x$ be a degree 4 vertex in $A$. We saw in the proof of \Cref{lem::nolongrelLemma} that a minimal injective resolution of the $A$-module $P_A(x)$ looks as follows:
$$0\longrightarrow P_A(x) \longrightarrow I^0 \longrightarrow I_A(x) \longrightarrow 0,$$ 
where $I^0$ is projective-injective. Thus, this is also a minimal projective resolution for $I_A(x)$ and $\pdim_A(I_A(x))=1.$ If $x\neq v$ then by definition of the cutting procedure, $x$ is also a degree 4 vertex in $B.$ Since $B$ is 2-Gorenstein and monomial, we must have that a minimal injective resolution of the B-module $P_B(x)$ also has the following form:
$$0\longrightarrow P_B(x) \longrightarrow J^0 \longrightarrow I_B(x) \longrightarrow 0,$$
where $J^0$ is projective-injective and $\pdim_B(I_B(x))=1.$

Next, observe that $P_A(s_i)\cong M(r_iz_i)\cong I_A(t_i)$ for $i=1,2.$ Here $s_i$ and $t_i$ are defined as in the beginning of Section \ref{subsec::proofOfCuttingThm}. If $s_1\neq s_2$ then each of the $r_i$ and $z_i$ contain $v$ only once. This follows from  Lemma \ref{lem::z_iContainsV2X}. In this case, $P_B(s_i)\cong M(r_i) \cong I_B(v_i)$ and $P_B(v_i)\cong M(z_{i'}) \cong I_B(t_{i'})$ are projective-injective modules where $i'\in\{1,2\}$ with $i\neq i'$.
We will use this definition for $i'$ throughout the entire proof. 

If $s_1=s_2$ then $r_2=r_1q$ and $z_1=qz_2$ (the roles of the indices $1$ and $2$ may be exchanged), where $q$ is a path of length at least one and contains $v$ exactly twice (as its start and end vertex), and $r_1$ and $z_2$ contain $v$ exactly once. In this case, $P_B(s_1)\cong M(r_1) \cong I_B(v_1), P_B(v_1) \cong M(z_2)\cong I_B(t_1),$ and  $P_B(v_2)\cong M(q) \cong I_B(v_2)$ are projective-injective.  
 
Finally, we deal with the uniserial indecomposable projective $A$-modules different from $P_A(s_i)$. These are always of the form $M(p)\cong P_A(x)$ where $p$ is a path belonging either to Case (1) or Case (3) in Lemma \ref{lem::same_resol} and $x=s(p)$. It is easy to see that in both cases the indecomposable projective $B$-module $P_B(x)\cong M(p').$ Note that $x=s(p)=s(p').$ Then by the statement of Lemma \ref{lem::same_resol}, 
$$\idim_A(P_A(x))=\idim_B(P_B(x)).$$
In conclusion, this holds for every vertex $x$ of $A$ different from $v.$

Analogously, we can conclude using the dual of Lemma \ref{lem::same_resol} that $$\pdim_A(I_A(x))=\pdim_B(I_B(x))$$ for every vertex $x$ of $A$ with $x\neq v.$ Here we need to observe that in the characterisation of $2$-Gorenstein monomial algebras all the conditions are left-right symmetric, which allows for a dual argument. 

This, and Lemma \ref{lem::same_resol} imply that for any vertex $x\neq s_i$ such that  $P_A(x)$ is uniserial we have  \begin{equation}
\label{eq::pdim_injective_terms}
    \pdim_A(I^d)=\pdim_B(J^d)
\end{equation} where $I^d$ (resp. $J^d$) is the $d$-th term in a minimal injective resolution of the $A$-module $P_A(x)\cong M(p)$ (resp. of the $B$-module $P_B(x)\cong M(p')$) if $I^d$ does not contain any direct summand isomorphic to $I(v)$. 
So it remains to check the case when $I(v)$ is a direct summand of $I^d$. In this case, the corresponding direct summand in $J^d$ is $I(v_i)$ by Lemma \ref{lem::same_resol}.  Recall that $\pdim_A(I(v))=1$ and $\pdim_B(I(v_i))=0$. Note that $I(v)$ cannot appear in a term $I^d$ with $d=0.$ This is enough to obtain that $$\pdim_A(I^d)\le d \iff \pdim_A(J^d)\le d$$ for every $d.$
This concludes the proof for part (1), and (3) of the theorem. Note that we also obtained that $\idim A_A< \infty$ if and only if $\idim B_B< \infty$, so to prove (2), it is enough to show $\idim {}_AA< \infty$ if and only if $\idim {}_BB< \infty.$

In order to show this, we will once again need the dual of Lemma \ref{lem::same_resol}. This compares the terms of the minimal projective resolutions of $I_A(x)$ and $I_B(x)$ if $I_A(x)$ is uniserial and $x\neq t_i$.  
From this, and the above established results about the minimal projective resolution of $I(x)$ if $x=t_1$ or $t_2$, or $x$ is a degree 4 vertex we obtain that 
$\pdim_A(I_A(x))<\infty$ if and only if $\pdim_B(I_B(x))<\infty$ and $\Omega_A^{\pdim I_A(x)}(I_A(x))$ is indecomposable if and only if the same is true for $\Omega_B^{\pdim I_B(x)}(I_B(x))$. In other words, $\psi_A$ is well-defined for $A$ if and only if $\psi_B$ is well-defined for $B.$

In this case, $$\psi_A(x)=\psi_B(x)$$ if $\psi_A(x)\neq v,s_1$ or $s_2.$ Here we used the dual statement of Lemma \ref{lem::same_resol} once again. Note that $P_A(s_i)$ is injective, so it can only be the last non-zero term in a minimal projective resolution of $I(x)$ if it appears in degree zero, which happens only in case $x=t_i$. 
We have $\psi_A(t_i)=s_i$, $\psi_A(v)=v$. As discussed before, if $s_1\neq s_2$ $\psi_B(v_i)=s_i$ and $\psi_B(t_i)=v_{i'}$ for both $i=1$ and $i=2$. If $s_1 = s_2$ we have $\psi_B(v_1)=s_1$, $\psi_B(t_1)=v_1$ and $\psi_B(v_2)=v_{2}$ (the roles of $v_1$ and $v_2$ might be exchanged).
In both cases, we can conclude that $\psi_A$ is bijective if and only if $\psi_B$ is. This finishes the proof for parts (2) and (4). 
    
\end{proof}

\newpage

\section{Nakayama algebras and the n-Gorenstein property}

\label{chap::Nakayama}
In this chapter, we investigate the $n$-Gorenstein property of Nakayama algebras. The importance of this lies in the main result of the last chapter: If we understand the Auslander-Gorenstein condition for this class of algebras, we understand it for the much larger class of monomial algebras. 
First of all, it is easy to see that every Nakayama algebra is $1$-Gorenstein. Second, \Cref{thm::classification2GorMonomial} provides a characterisation of the  $2$-Gorenstein condition for Nakayama algebras. We recall this here. 

\begin{corollary}
\label{cor::2-Gor-class-Nak}
Let $A=kQ/I$ be a Nakayama algebra. Then $A$ is $2$-Gorenstein if and only if every arrow $\alpha\in Q_1$ which is contained in a minimal relation is the start or the end of a minimal relation.

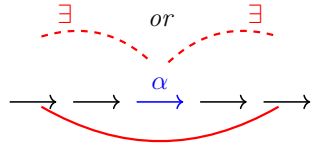
\vspace{5pt}

{
\begin{figure}[h]

\centering
\resizebox{0.3\linewidth}{!}{
    \input{pictures/startOrEndOfMinimalRel_picture1}}
    
\end{figure}

}

\end{corollary}

This exact result has also been discovered by Jacob Fjeld Grevstad \cite{G23}. A different classification of $2$-Gorenstein linear Nakayama algebras can be found in \cite{RS17}, where the authors use the connection of Nakayama algebras to Dyck paths to provide a purely combinatorial characterisation. Explicitly, they show that a linear Nakayama algebra is $2$-Gorenstein if and only if its associated Dyck path contains no double deficiencies. This criterion is equivalent to Corollary \ref{cor::2-Gor-class-Nak}. Furthermore, the paper also proves that these algebras are enumerated by Motzkin paths, confirming another conjecture of Marczinzik.
We would also like to draw the reader's attention to the article \cite{STZ24}, in which the authors classify all cyclic Nakayama algebras that are minimal Auslander-Gorenstein, as well as all dominant Auslander regular Nakayama algebras of global dimension three. Both families represent notable subclasses of Auslander-Gorenstein algebras.

Although Nakayama algebras are often considered among the simplest classes of algebras from a representation-theoretic perspective, there is no known classification of the Auslander-Gorenstein ones, nor do we provide one here. This may be surprising, given that such a classification exists for gentle algebras. The challenge, however, lies in achieving uniform, global control over the injective resolutions of all projective modules. While gentle algebras have a consistent local structure, with all relations being of length two and looking the same at degree 4 vertices, Nakayama algebras vary significantly from vertex to vertex, making a global approach much more difficult.

In the following, we prove two results about the n-Gorenstein property of Nakayama algebras: First, we generalise a result of Iwanaga and Fuller (\Cref{2kImplies2k+1}), and second, we show that the Auslander-Gorenstein property can, in fact, be characterised via the well-definedness and bijectivity of the Auslander-Reiten map (\Cref{thm::AGiffARbijwelldef_Nak}).
\subsection{Definitions and notations}
\label{subsec::notationNakayama} 
We take this opportunity to introduce the necessary notions and notations used in this chapter.
A finite dimensional quiver algebra $A=kQ/I$ is a \textcolor{blue}{Nakayama algebra} if $Q$ is either the quiver 
\[\begin{tikzcd}[font=\normalsize, line width=0.6pt]
	 1 & {2} & \ldots & {N-1} & N
	\arrow[from=1-1, to=1-2]
	\arrow[from=1-2, to=1-3]
	\arrow[from=1-4, to=1-5]
	\arrow[from=1-3, to=1-4]
\end{tikzcd}\] or the quiver 
\[\begin{tikzcd}[font=\normalsize, line width=0.6pt]
	1 & {2} & \ldots & {N-1} & N.
	\arrow[from=1-1, to=1-2]
	\arrow[from=1-2, to=1-3]
	\arrow[from=1-3, to=1-4]
	\arrow[from=1-4, to=1-5]
	\arrow[bend left=25, from=1-5, to=1-1]
\end{tikzcd}\]

In order to deal with the cyclic case, it will be convenient for us to use the entire $\mathbb{Z}$ as vertex set of $Q.$ For this, we set 
$S_x:=S_{\overline{x}}$, $P(x):=P(\overline{x})$, $I(x):=I(\overline{x})$ for every $x\in \mathbb{Z}$ where $\overline{x}\in Q_0=\{1,\ldots,N\}$ such that 
$\overline{x}\equiv x \;(\bmod\; N).$ We will say that a vertex $x\in \mathbb{Z}$ is the socle or top of a module $M$ if $\soc(M)\cong S_x$ or $\top(M)\cong S_x,$ respectively. In particular, if we write $\soc(M)<\soc(M'),$ this stands for $x<y$ where $\soc(M)\cong S_x$ and $\soc(M')\cong S_y.$ 

Nakayama algebras are exactly those algebras over which every indecomposable module is uniserial. In particular, every uniserial module over the Nakayama algebra $A=kQ/I$ corresponds to a path in $Q$ which is not contained in $I.$ We will denote a path starting at vertex $x$ and ending at vertex $y$ by $\textcolor{blue}{p^x_y}.$ Here we work in $\mathbb{Z}$ and not modulo $N,$ the uniserial module corresponding to this path will have $k$-dimension $y-x+1.$ So, every indecomposable module is isomorphic to a module of the form $M(p^x_y).$ Explicitly, $M(p^x_y)\cong e_xA/e_xJ^{y-x+1}$ where $J$ stands for the Jacobson radical of $A.$  Alternatively, we also denote this module as the interval $\textcolor{blue}{[x,y]}\subset \mathbb{Z}$. This is particularly visual, as $\top[x,y]=x$ and $\soc[x,y]=y.$ Naturally, we have $[x,y]\cong [x+cN,y+cN]$ for any $c\in \mathbb{Z}.$

It is not too hard to see that a Nakayama algebra with $N$ simple modules is uniquely determined by its \textcolor{blue}{Kupisch series}, which is defined to be $[c_1,\ldots,c_N]$ where $\textcolor{blue}{c_j}=\dim_k(P(j)).$ Dually, it is also determined by its Co-Kupisch series $[d_1,\ldots,d_N],$ where $\textcolor{blue}{d_j}=\dim_k(I(j)).$ 

Next, we define the following maps:

$$\textcolor{blue}{f}: \mathbb{Z}\rightarrow \mathbb{Z}, \ j \mapsto j+c_j$$ and 
$$\textcolor{blue}{g}: \mathbb{Z}\rightarrow \mathbb{Z}, \ j \mapsto j-d_j.$$ Here $d_x:=d_{\overline{x}}$ and $c_x:=c_{\overline{x}}$. These maps $f$ and $g$ were already introduced in \cite{G85}, and their properties were studied, e.g., in \cite{Sh17} and \cite{R21}. It is important to note that $f$ and $g$ are both monotone increasing, and they are especially useful for describing homological resolutions of indecomposable modules. By definition of $c_j$, $\soc P(j)=S_{j+c_j-1}$ and $\top I(j)=S_{j-d_j+1}.$ In other words, $P(j)=[j,j+c_j-1]=[j,f(j)-1]$ and $I(j)=[j-d_j+1,j]=[g(j)+1,j].$  
Then a minimal injective resolution of $P(i)$ is given by:
\begin{equation}
\label{eq::injResolUsingg}
    0\rightarrow {P(i)} \rightarrow \textcolor{purple}{I(a)} \rightarrow \textcolor{orange}{I(i-1)} \rightarrow \textcolor{purple}{I(g(a))} \rightarrow \textcolor{orange}{I(g(i-1))} \rightarrow \textcolor{purple}{I(g^2(a))} \rightarrow \textcolor{orange}{I(g^2(i-1))}\rightarrow \ldots
\end{equation}
where $a=i+c_i-1.$ Thus, the resolution can be expressed as two sequences of neighbouring intervals:  $$P(i)=[i,a],$$ \ $$\textcolor{purple}{I^{2j}}=\textcolor{purple}{I(g^j(a))}=\textcolor{purple}{[g^{j+1}(a)+1, g^j(a)]},$$ and $$\textcolor{orange}{I^{2j+1}}=\textcolor{orange}{I(g^j(i-1))}= \textcolor{orange}{[g^{j+1}(i-1)+1, g^j(i-1)]}$$ for $j\ge 0$ if $I^{ 2j},$ resp. $I^{2j+1}$ is a non-zero term. We have $$\idim P(i)=\inf \{d \ | \ \top I^{d-1}=\top I^d \}=\inf \{d \ | \ \top I^{d-1}\le \top I^d \}.$$
If the resolution is finite, it either ends with
$$\ldots \rightarrow \textcolor{orange}{I(g^{j-1}(i-1))} \rightarrow \textcolor{purple}{I(g^{j}(a))} \rightarrow 0$$ or with 
$$\ldots \rightarrow \textcolor{purple}{I(g^j(a))} \rightarrow \textcolor{orange}{I(g^{j}(i-1))} \rightarrow 0.$$ Therefore,
\begin{equation}
\label{eq::Nakayama_idimP(i)}
\idim P(i)=
\inf (\{2j\ |\ g^{j}(i-1)\le g^{j+1}(a)\}\cup \{2j+1\ |\ g^{j+1}(a)\le g^{j+1}(i-1)\}).
\end{equation}

Analogously, using the map $f$ we can describe minimal projective resolutions of indecomposable injective modules. 

\subsection{2n-Gorenstein implies (2n+1)-Gorenstein}

Iwanaga and Fuller showed in \cite[Proposition 2.10]{FI93} that if a Nakayama algebra is $2$-Gorenstein it is necessarily also $3$-Gorenstein. In this section, we show that this statement holds in a greater generality.

\begin{theorem}
\label{2kImplies2k+1}
    Let $A$ be a Nakayama algebra. If $A$ is $2n$-Gorenstein for some $n\in \mathbb{N}^+$, then $A$ is necessarily $(2n+1)$-Gorenstein as well.
\end{theorem}

\begin{remark}
    Interestingly, parity plays an important role in the previous statement. For any $n\ge 1$ there exists a Nakayama algebra which is $(2n+1)$-Gorenstein but not $(2n+2)$-Gorenstein. An example would be the cyclic Nakayama algebra $A$ with $N=3(n+1)$ simple modules and Kupisch series $[2,3,3,\ldots,3]$. This generalises the example of Iwanaga and Fuller from \cite[Chapter 2]{FI93} of a $3$-Gorenstein Nakayama algebra which is not $4$-Gorenstein. 
    Indeed, in this algebra $A$ the indecomposable injective and projective modules are as follows:
    
    $$P(1)=\begin{smallmatrix}
  1\\
  2
\end{smallmatrix},\ I(3)= \begin{smallmatrix}
  2\\
  3
\end{smallmatrix} \text{ and } P(i)=\begin{smallmatrix}
  i\\
  i+1\\
  i+2
\end{smallmatrix}=I(i+2) \text{ for all }i\neq 1.$$

Calculation gives that $\idim P(1)=\infty=\pdim I(3).$ Moreover, $I(3)$ turns up in a minimal injective resolution of $P(1)$ in degree $2n+1$ for the first time. As a result, $A$ is indeed $(2n+1)$-Gorenstein but not $(2n+2)$-Gorenstein. 
    
\end{remark}

The following simple lemma plays a crucial role in the proof of \Cref{2kImplies2k+1}.

\begin{lemma}
    \label{lem::inequalitiesSocTop}
    Let $A=kQ/I$ be a Nakayama algebra, and use the notations as discussed in \Cref{subsec::notationNakayama}. Let $M$, $J$ and $P$ be indecomposable $A$-modules, such that $P$ is projective and $J$ is injective. Then the following hold:

    \begin{enumerate}[\rm (a)]
        \item If $\top (M) \le  \top(P)$ then $\soc (M)\le  \soc(P)$.
        \item If $\soc (P) <  \soc(M)$ then $\top (P) < \top(M)$.
        \item If $\soc (J) \le  \soc(M)$ then $\top (J) \le \soc(M)$. 
        \item If $\top (M) <  \top(J)$ then $\soc (M) < \soc(J)$.
        
    \end{enumerate}
    
\end{lemma}

\begin{proof}

    The proof relies on the fact that the indecomposable projective $P(x)$ corresponds to the unique right-maximal path starting at $x$, and the indecomposable injective $I(x)$ corresponds to the unique left-maximal path ending at $x$.
    
    Let $P=P(x)=[x,f(x)-1]$ and $M=[y,z].$ By definition of $f,$ $[x,f(x)]\in I$ is a relation, so $[y,z]$ cannot contain $[x,f(x)]$. Therefore, if $y\le x$ then $z\le f(x)-1.$ Moreover, if $f(x)-1<z$ then $x<y.$ These are exactly statements (a) and (b).
    
    Dually, let $J=I(x)=[g(x)+1,x]$ and $M=[y,z].$ By definition of $g,$ $[g(x),x]\in I$, so $[y,z]$ cannot contain $[g(x),x]$. Thus, if $x\le z$ then $g(x)+1\le y.$ Moreover, if $y<g(x)+1$ then $z<x.$ These are exactly statements (c) and (d).
    
    The following picture illustrates the intervals appearing in the proof. We need to use the interval $M_i$ for part (i).
    
    
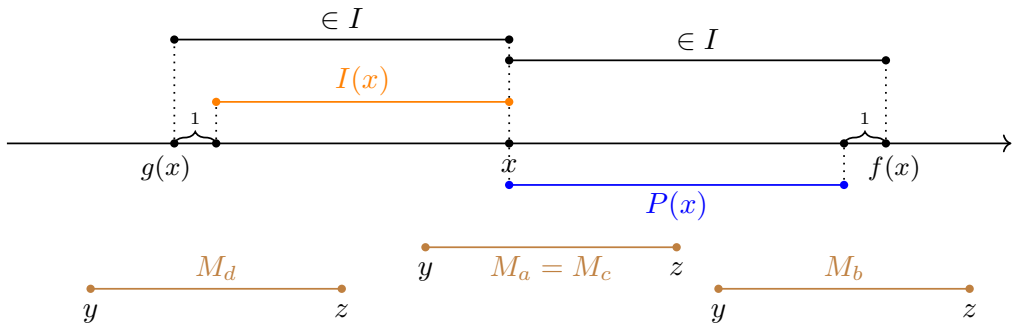
{


\centering
\resizebox{0.9\linewidth}{!}{
    \input{pictures/intervals_picture}}
    

}
\end{proof}

\begin{corollary}
\label{cor::Nakayama_f_g_formulas}
Let $A=kQ/I$ be a Nakayama algebra, and use the notations as discussed in Section \ref{subsec::notationNakayama}. Let $j,\ell\in \mathbb{Z}.$ Then the following hold:
\begin{enumerate}[\rm (a)]
    \item If $g(j)+1\le \ell$ then $j+1\le f(\ell).$
    \item If $f(\ell)-1<j$ then $\ell-1 < g(j).$ 
    \item If $j\le f(\ell)-1$ then $g(j) \le \ell-1.$
    \item If $\ell<g(j)+1$ then $f(\ell)<j+1.$
\end{enumerate}

\end{corollary}
\begin{proof}
    Apply \Cref{lem::inequalitiesSocTop} to the projective $P(\ell)=[\ell,f(\ell)-1]$ and the injective $I(j)=[g(j)+1,j].$
    
\end{proof}

\begin{proof}[Proof of \Cref{2kImplies2k+1}]  Assume that $A=kQ/I$ is $2n$-Gorenstein for some $n\ge 1.$ Let $i\in Q_0$ and consider a  minimal injective resolution of $P(i):$ 
\begin{equation}
    0\rightarrow {P(i)} \rightarrow I^0 \rightarrow I^1 \rightarrow I^2 \rightarrow \ldots \rightarrow I^{2n} \rightarrow \ldots
 \end{equation}
We need to show that $\pdim I^{2n}\le 2n.$
We will use the notation from \Cref{subsec::notationNakayama} for the non-zero terms in the above resolution: $$P(i)=[i,a]$$ $$\textcolor{purple}{I^{2j}}=\textcolor{black}{I(g^j(a))}=\textcolor{black}{[g^{j+1}(a)+1, g^j(a)]},$$ and $$\textcolor{orange}{I^{2j+1}}=\textcolor{black}{I(g^j(i-1))}= \textcolor{black}{[g^{j+1}(i-1)+1, g^j(i-1)]}$$ for $j\ge 0.$ 

We may assume that $\idim P(i)\ge 2n,$ otherwise, there is nothing to show. If $\idim P(i)=2n,$ we immediately obtain $\pdim I^{2n}=2n$ from \Cref{thm::preliminaries::ARbij} and the assumption that $A$ is $2n$-Gorenstein. So it remains to consider the case when $\idim P(i)>2n.$

As mentioned in \Cref{subsec::notationNakayama}, a minimal projective resolution of $I^{2n}$ has the following form:

\begin{equation}
    \ldots\rightarrow {P(f^2(c))} \rightarrow P(f(d)) \rightarrow {P(f(c))} \rightarrow {P(d)} \rightarrow {P(c)} \rightarrow  {I^{2n}} \rightarrow 0.
\end{equation}
Here, $c,d\in \mathbb{Z}$ such that $c=\top I^{2n}$ and $d=\soc I^{2n}+1$. We also use the notation $P_t$ for the term in degree $t$. By the dual of \Cref{eq::Nakayama_idimP(i)}, in order to show $\pdim I^{2n}\le 2n,$ it is enough to prove 
$$f^{n+1}(c)\le f^{n}(d).$$ The proof is divided into two steps: First, we prove $$f^{n+1}(c)\le a+1$$ by comparing $c,f(c),\ldots,f^{n+1}(c)$ to the tops of \textit{odd} terms in the minimal injective resolution of $P(i).$ In the second step, we prove $$ a+1\le f^{n}(d)$$ by comparing $d,f(d),\ldots,f^{n}(d)$ to the tops of \textit{even} terms in the minimal injective resolution of $P(i).$ We can do these as the values $c$ and $d$ arise from the resolution of $P(i).$ 

\textbf{Step 1.} Recall that $\top P_0= c=\top I^{2n}$ and we assumed $\idim P(i)> 2n.$ Then $$\top P_0= c < \top \textcolor{orange}{I^{2n-1}}=g^{n}(i-1)+1.$$ We can apply Corollary \ref{cor::Nakayama_f_g_formulas} (d) $n$ times to this inequality to obtain  $f^{n}(c) < i-1+1=i.$ Finally, since $f$ is monotone increasing, we have $f^{n+1}(c)\le f(i)=a+1,$ where the last equality follows from $P(i)=[i,a].$

Note that the first $n$ steps correspond to applying Lemma \ref{lem::inequalitiesSocTop} (d) to the pairs $(P_0,I^{2n-1}),$ $(P_2,I^{2n-3}), \ldots, (P_{2n-2},I^{1}).$

\textbf{Step 2.} We have $\top P_1=d=\top \textcolor{purple}{I^{2n-2}}=g^n(a)+1.$ We apply Corollary \ref{cor::Nakayama_f_g_formulas} (a) $n$ times to this inequality to obtain  $f^{n}(d)\ge a+1.$

Just as in Step 1, every step corresponds to applying Lemma \ref{lem::inequalitiesSocTop} (a) to the pairs $(I^{2n-2}, P_1), (I^{2n-4},P_3), \ldots,(I^0,P_{2n-1}) $. 

\end{proof}

\subsection{The Auslander-Reiten bijection}
As a main result of this section, we prove that \hyperlink{conjectureM}{Marczinzik's conjecture} holds for every Nakayama algebra. The proof relies on the results of the previous section, more precisely, on \Cref{2kImplies2k+1}.

\begin{theorem}
\label{thm::AGiffARbijwelldef_Nak}
    Let $A=kQ/I$ be a Nakayama algebra. Then $A$ is Auslander-Gorenstein if and only if $\psi_A$ is well-defined and bijective.
\end{theorem}

\textbf{Proof strategy:} We will use the following strategy to prove this statement: We assume that $A$ is not Auslander-Gorenstein.
Then either $\pdim D({}_AA)=\infty$ or there exists a $n\ge 1$ such that $A$ is $(2n-1)$-Gorenstein, but not $2n$-Gorenstein.
This follows from \Cref{thm::preliminaries:left-right-symmetry} and Theorem \ref{2kImplies2k+1}.
Note that every Nakayama algebra is $1$-Gorenstein. If $\pdim D({}_AA)=\infty$, then $\psi_A$ is not well-defined for an indecomposable injective.
So, it remains to consider the second case.
Then exists a vertex $j$ such that in the minimal projective resolution of $I(j)$ 

\begin{equation}
\label{eq::projResOfI(j)}
    \ldots \longrightarrow \textcolor{blue}{P_{2n-1}} 
    \longrightarrow P_{2n-2} \longrightarrow \ldots \longrightarrow P_1 \longrightarrow P_0 \longrightarrow I(j) \longrightarrow 0
\end{equation}
the $(2n-1)$-st term, $P_{2n-1}$ has injective dimension greater than $2n-1.$ Dual to  description in \Cref{subsec::notationNakayama}, the non-zero terms in this resolution have the following form:
$$I(j)=[g(j)+1,j]=[c,j]$$ $$\textcolor{black}{P_{2t}}=\textcolor{black}{P(f^t(c))}=\textcolor{black}{[f^t(c), f^{t+1}(c)-1]},$$ and $$\textcolor{black}{P_{2t+1}}=\textcolor{black}{P(f^t(j+1))}= \textcolor{black}{[f^t(j+1), f^{t+1}(j+1)-1]}$$ for $t\ge 0.$

The idea is then to show that this module $P_{2n-1}$ cannot be in the image of $\psi_A,$ implying that this is indeed not a bijection between the sets in question. For the sake of contradiction, we assume that $$\psi_A(J)\cong P_{2n-1}$$ for an indecomposable injective module $J.$ In other words, a minimal projective resolution of $J$ ends with $P_{2n-1}:$

\begin{equation}
\label{eq::projResolOfJ}
    0\longrightarrow \textcolor{blue}{P_{2n-1}}\cong R_d \longrightarrow R_{d-1} \longrightarrow R_{d-2} \longrightarrow \ldots \longrightarrow R_{d-2n} \longrightarrow \ldots \longrightarrow R_0\longrightarrow J \longrightarrow 0 .
\end{equation}

Our task is to show that such a resolution cannot exist. We do this by bounding the tops of the projective terms in this resolution from above using the resolution (\ref{eq::projResOfI(j)}), and bounding their socles from below using a minimal injective resolution of $P_{2n-3}.$ This way we find that the path corresponding to the indecomposable module $R_{d-2n}$ must contain a path which belongs to the ideal $I$. This, however, cannot happen if $R_{d-2n}$ is a non-zero module.

Let us start with getting a better understanding of the terms in (\ref{eq::projResolOfJ}) and introducing some notation. The following observations all follow from \Cref{thm::preliminaries::ARbij} and the assumptions that $A$ is $(2n-1)$-Gorenstein and $\idim P_{2n-1}>2n-1$:

\begin{enumerate}[\rm (A)]
    \item \label{a} $\pdim J=d\ge 2n$. 
    \item \label{b} $\pdim I(j)\ge 2n.$
    \item \label{c} $I(j)$ cannot be the last non-zero term of the minimal injective resolution of an indecomposable projective $P$ with $\idim P<2n.$ 
\end{enumerate}

Using \ref{a}, we obtain that there is a non-zero, projective term $R_{d-2n}$ in the minimal projective resolution of $J$, let us denote it as $$R_{d-2n}=[r,f(r)-1]$$ for some vertex $r.$ Then  $$R_d=[f^{n}(r),f^{n+1}(r)-1].$$ We would like $R_d$ to correspond to the exact same interval in $\mathbb{Z}$ as 
$P_{2n-1},$ not only a shift of it by some multiple of $N.$ So choose $r\in \mathbb{Z}$ such that $[f^{n-1}(j+1),f^{n}(j+1)-1]=[f^{n}(r),f^{n+1}(r)-1]$.

We divide the rest of the proof into three lemmas. The first tells us that $g(j)$ is an upper bound of $\top R_{d-2n}$.

\begin{lemma}
\label{lem::g(j)inR}
We have 
$g(j)\in [r,f(r)-1]=R_{d-2n}.$
\end{lemma}

\begin{proof}
Since $\pdim I(j)\ge 2n$ by our observation \ref{b}, $f^n(j+1)>f^n(c)>f^{n-1}(j+1).$ Here we used the formula (\ref{eq::Nakayama_idimP(i)}) for the projective dimension of an injective module. These vertices correspond to the tops of $P_{2n+1}, P_{2n}$ and $P_{2n-1},$ respectively. We chose $r$ so that  $f^{n-1}(j+1)=f^{n}(r)$ and $f^{n}(j+1)=f^{n+1}(r).$ Hence, 
$$f^{n+1}(r)>f^n(c)>f^{n}(r).$$ Since $f$ is monotone increasing, $f(r)>c>r.$ As $r$ and $c$ are both integers, $c-1\in [r,f(r)-1].$ By definition of $c,$ $c-1=g(j),$ so in conclusion, $g(j)\in [r,f(r)-1].$

Note that this way we actually show that $\top P_{2t}-1 =\soc P_{2t-2}\in R_{d-2n+2t}$ for $t=0,1,\ldots, n$, where $t=0$ is the case we are interested in.
\end{proof}

The next lemma will be used to describe the  socle of certain terms in an injective resolution of $P_{2n-3}$. 

\begin{lemma}
\label{lem::technical}
    For $0\le s\le n-1$ we have $g^{s}(f^s(j+1)-1)=j$.
\end{lemma}
\begin{proof}
    We will proceed by induction on $s.$ The base case $s=0$ holds trivially. For the induction step, assume that the formula holds for some $0\le s<n-1,$ and show that this implies that it also holds for $s+1.$
    
    Observe that $a=f^{s+1}(j+1)-1$ corresponds to the socle of $P_{2s+1},$ and $a,g(a),g^2(a),\ldots$ correspond to the socles of the even terms in the minimal injective resolution of $P_{2s+1}$ which is given as in (\ref{eq::injResolUsingg}): 
    
    $$0 \longrightarrow P_{2s+1} \longrightarrow I(a) \longrightarrow I(b)\longrightarrow I(g(a)) \longrightarrow I(g(b)) \longrightarrow I(g^2(a)) \longrightarrow \ldots $$ 
    The socles of the odd terms are $b=f^s(j+1)-1, g(b), g^2(b), \ldots$. We also use the notation $I^t$ to denote the term in degree $t$ in this resolution. Note that even if $\idim P_{2s+1}<t,$ $I^t$ is set to be the non-zero module $[g(a)^{ \frac{t}{2}+1}+1,g^{\frac{t}{2}}(a)]$ if $t$ is even and  $I^t=[g(b)^{\lfloor \frac{t}{2}\rfloor+1}+1,g^{\lfloor \frac{t}{2}\rfloor}(b)]$ if $t$ is odd. 
    
    By our induction assumption, $I^{2s+1}$ has socle $g^{s}(b)=j.$ So in the end, we would like to show that $\soc I^{2s+1}=\soc I^{2s+2}=g^{s+1}(a).$

    Since $s<n-1$ we have $2s+1<2n-1.$ As our algebra was assumed to be $(2n-1)$-Gorenstein, $\idim P_{2s+1}\le 2s+1.$ 
    Thus, if $\idim P_{2s+1}$ is even, $g^{s}(b)=g^{s+1}(a)$ which is the same as $\soc I^{2s+1}=\soc I^{2s+2}$. So, in this case we are done by our induction hypothesis.
    
    So assume that $\idim P_{2s+1}$ is odd. Let us first prove that $\idim P_{2s+1}<2s+1$ is not possible. Observe that $\soc P_{2s-1}<\soc P_{2s+1}=\soc I^0$. In symbols, $f^{s}(j+1)-1 < a.$ Note that this also makes sense for $s=0,$ as in that case $j=\soc I(j)<\soc P_{1}=a$. Applying Corollary \ref{cor::Nakayama_f_g_formulas} (b)  $s$ times we obtain $j<g^{s}(a).$ In other words, $\soc I^{2s+1}=g^s(b)=j<\soc I^{2s}=g^s(a)$. If $\idim P_{2s+1}=2d+1$ then $g^{t}(a)=g^{t}(b)$ for every $t\ge d+1$. Thus, $\idim P_{2s+1}\ge 2s+1$ indeed. 
    
    So the only possible scenario is $\idim P_{2s+1}=2s+1.$ This meant that $I^{2s+1}=I(g^s(b))=I(j)$ is the last non-zero term in a minimal injective resolution of $P_{2s+1}.$ Here we used the induction assumption for $g^s(b)=j.$ This, however, is impossible as observed in \ref{c}. In conclusion, $\idim P_{2s+1}$ needs to be even, and the statement holds for $s+1.$

\end{proof}

Using the previous lemma, we finally obtain that $j$ is a lower bound of $\soc R_{d-2n}.$

\begin{lemma}
    \label{lemma::jinR}
    We have $j\in [r,f(r)-1]=R_{d-2n}.$
\end{lemma}
\begin{proof}

First, observe that $f^{n}(r)-1=\soc R_{d-2}=\top P_{2n-1}-1=\soc P_{2n-3}=f^{n-1}(j+1)-1.$
Note that this also makes sense for $n=1,$ because in that case $\top P_{2n-1}-1= \soc I(j)=j.$
Applying \Cref{cor::Nakayama_f_g_formulas} (c) $n-1$ times yields $\soc R_{d-2n}=f(r)-1\ge g^{n-1}(f^{n-1}(j+1)-1)=j$. The last equality follows from \Cref{lem::technical}. 

Note that these $n-1$ steps correspond to applying \Cref{lem::inequalitiesSocTop} (c) to the pairs $(I^0,R_{d-2}),$ $(I^2, R_{d-4}),\ldots (I^{2n-4},R_{d-2n+2})$ where $I^0,I^1,\ldots $ denote the terms of the injective resolution of $P_{2n-3}$ given as in (\ref{eq::injResolUsingg}).

Second, $g(j)<j$ by definition of $g$, and $g(j)\ge r$ by \Cref{lem::g(j)inR}. Combining these, we obtain $j\ge r$. In conclusion, $j\in [r,f(r)-1].$
\end{proof}

\begin{proof}[Proof of \Cref{thm::AGiffARbijwelldef_Monomial}]
    Lemma \ref{lem::g(j)inR} and \ref{lemma::jinR} tell us that $g(j), j\in [r,f(r)-1]=R_{d-2n}$. Thus, 
    the interval $[g(j),j]$ is contained in the interval $[r,f(r)-1]$. However, by definition of $g,$ $[g(j),j]\in I$, so $[r,f(r)-1]\in I$ as well. This, however, contradicts the definition of $f,$ and the fact that $R_{d-2n}$ is a non-zero module. 
    
    Hence, under our assumptions, no $J$ can exist for which $\psi_A(J)=P_{2n-1}$. In conclusion, if $A$ is not Auslander-Gorenstein, $\psi_A$ is not a well-defined bijection.
    
\end{proof}

\newpage
\section{Main results for monomial algebras}
\label{chap::MainResultsForMonomial}

In this final chapter, we demonstrate the significance of \Cref{thm::cutting}: Using the results on Nakayama algebras established in Chapter \ref{chap::Nakayama}, we can derive statements about the $n$-Gorenstein properties of monomial algebras. Most notably, we prove the main result of this article: Marczinzik's conjecture holds for the entire class of monomial algebras (\Cref{thm::AGiffARbijwelldef_Monomial}). This leads to a new homological characterisation of the Auslander-Gorenstein property within this class and underlines the central role of the Auslander-Reiten map. In addition to this, we also obtain a proof for a stronger version of the Auslander-Reiten Conjecture in the setting of monomial algebras (see \Cref{cor::strongerVersionofARConj}).

\begin{corollary}
\label{cor::2kimplies2k+1_monomial}
    If a monomial algebra is $2n$-Gorenstein then it is also $(2n+1)$-Gorenstein. This holds for every $n\ge 1.$
\end{corollary}
\begin{proof}
    This property gets inherited from Nakayama algebras to monomial algebras. As explained in Chapter \ref{chap::reduction}, every 2-Gorenstein monomial algebra can be reduced to a Nakayama algebra via the cutting procedure introduced in that chapter. With the help of \Cref{thm::cutting} (1) and \Cref{2kImplies2k+1}, we obtain the desired statement.
    
\end{proof}

\begin{corollary}
\label{cor::Q2for2Gor_monomial}
    For a 2-Gorenstein monomial algebra $A$ the following are equivalent:
    \begin{enumerate}[\rm (i)]
        \item $A$ is Auslander-Gorenstein.
        \item $\psi_A$ is a well-defined bijection for $A.$ 
    \end{enumerate}
\end{corollary}
\begin{proof}
As the main result of Chapter \ref{chap::reduction}, we described how to obtain from every 2-Gorenstein monomial algebra a Nakayama algebra, which shares several important homological properties with the original algebra. This corollary is then a consequence of \Cref{thm::cutting} (3) and (4), as well as \Cref{thm::AGiffARbijwelldef_Nak}. 

\end{proof}

\subsection{Preparatory lemmas}

We will show that we can omit the 2-Gorenstein assumption from Corollary \ref{cor::Q2for2Gor_monomial}. For this, we need the following technical lemmas.

\begin{lemma}\label{lem::ARwelldefbij_Implies_1Gor}
    Let $A=kQ/I$ be a finite-dimensional quiver algebra for which $\psi_A$ is a well-defined bijection. Then $A$ is 1-Gorenstein. 
\end{lemma}
\begin{proof}
If $A$ is not $1$-Gorenstein, then neither is $A^{op}$. So in this case there exists a vertex $x$ such that the projective cover of $I(x)\cong D(Ae_x)$ is not injective. More precisely, there exists a left-maximal path $p$ in $A$ ending at $x$ such that $P(s(p))$ is not projective-injective. We will show that $P(s(p))$ cannot lie in the image of $\psi_A.$ Assume for the sake of contradiction that there is an indecomposable injective module $J$ such that 
    $$0\rightarrow P(s(p))=P_d \xrightarrow{p_{d}} P_{d-1} \rightarrow \ldots \rightarrow P_0 \rightarrow J \rightarrow 0$$ is a minimal projective resolution. As $P(s(p))$ is not injective, $d\ge 1.$ Recall that indecomposable projective $A$-modules are all isomorphic to a module $e_yA$ for some vertex $y$.
     We have that $e_{s(p)},p\in e_{s(p)}A = P(s(p))$ are non-zero elements of $P_d$ and $e_{s(p)}\cdot p=p.$ Since $p_d$ is a monomorphism, $0 \neq p_d(p)\in P_{d-1}$ and $p_d(e_{s(p)})\cdot p = p_d(p).$ In a minimal projective resolution $p_d(P_d)\subseteq \rad(P_{d-1})=P_{d-1}\cdot \rad(A),$ where $\rad(A)$ is the Jacobson radical of $A.$ Since $\rad(A)$ is spanned by the arrows in $Q,$ there exist non-zero elements $h_1,h_2,\ldots,h_m\in P_{d-1}$ and arrows $a_1,a_2,\ldots,a_m\in Q_1$ such that $$\Sigma_{i=1}^mh_i\cdot a_i=p_d(e_{s(p)}).$$
     As $p_d(e_{s(p)})\cdot p=p_d(p)\neq 0,$ there exists at least one $i\in\{1,2,\ldots, m\}$ such that $(h_i\cdot a_i)\cdot p=h_i\cdot a_ip\neq 0.$ Hence $a_ip\neq 0$. This, however, contradicts the left-maximality of the path $p.$ In conclusion, there is no such $J$, and thus, $\psi_A$ cannot be surjective if $A$ is not $1$-Gorenstein.

\end{proof}

\begin{remark}
Note that in the previous lemma we did not assume that $A$ is monomial. The arguments in the proof can be adjusted to also work for non-basic finite-dimensional $k$-algebras. 
\end{remark}

\begin{lemma}\label{lem:dimtopI<=2}
    Let $A=kQ/I$ be a monomial algebra and assume that $\psi_A$ is a well-defined bijection for $A$. Then $\dim (\top J)\le 2$ for every indecomposable injective $A$-module $J$, and $\dim (\soc P)\le 2$ for every indecomposable projective $A$-module $P.$
\end{lemma}
\begin{proof}
    For the sake of contradiction, assume this is not true and there is a vertex $x$ which is the starting point of three different right-maximal paths, $q_1,q_2,$ and $q_3.$ We may always choose these paths (and the corresponding vertex $x$) such that there is no arrow $a$ with $aq_i\notin I$ for all three $i\in\{1,2,3\}.$ By \Cref{lem::ARwelldefbij_Implies_1Gor}, A is 1-Gorenstein, therefore, we can deduce from the dual of \Cref{pdim1Lemma} (i) that $P(x)$ has injective dimension at least $2.$ As we assumed $\psi_A$ to be bijective, there exists an indecomposable injective $A$ module $J$ such that $P(x)$ is the last term in a minimal projective resolution of $J.$ $A$ is 1-Gorenstein, so \Cref{thm::preliminaries::ARbij} implies that $\idim (P(x))\le 1$ if and only if $\pdim J\le 1.$ Therefore, we obtain $d=\pdim J\ge 2,$ and $P(x)\cong \Omega^d J.$ By Lemma \ref{lem::syzygiesInMonomialAlg}, every such $P(x)$ is of the form $P(x)\cong pA$ for some path $p$ of non-zero length. In particular, $P(x)$ is a proper submodule of the indecomposable projective $e_{s(p)}A$ which has a simple top. This, however, cannot happen as we chose the $q_i$ such that there is no arrow $a$ for which $aq_i\notin I$  for all three $i\in\{1,2,3\}.$ This proves $\dim(\soc P(x))\le 2$ for every vertex $x.$ As in a 1-Gorenstein monomial algebra $\dim(\soc P(x))=\dim (\top I(x))$ for every vertex $x$ by Lemma \ref{lem::1GorMonomial: dimsocP(x)=dimtopI(x)}, the dual statement for the injective modules follows.
    
\end{proof}

The next lemma follows immediately from the proof of Lemma \ref{lem:dimtopI<=2}. 

\begin{lemma}\label{lem::technical_pdim1}
    Let $A=kQ/I$ be a monomial algebra such that $\psi_A$ is a well-defined bijection.
    Moreover, let $x$ be a vertex such that two different right-maximal paths start at $x,$ $p_1$ and $p_2$, and that there is no arrow $\alpha$ such that $ \alpha p_1 \notin I$ and $\alpha p_2\notin I$ hold at the same time. Then $\idim P(x)=1.$

\end{lemma}

\begin{lemma}\label{lem::technical_SecondSyzNotIndec}
Let $A=kQ/I$ be a monomial algebra such that $\psi_A$ is a well-defined bijection.
    Let $I(x)$ be an indecomposable injective $A$-module with $\dim(\top I(x))= 2$ and so that $\dim(\top \Omega(I(x)))\ge 2$. Then the top of $\Omega^{2}(I(x))$ is not simple.  
\end{lemma}
\begin{proof}
    We explicitly calculate a minimal projective presentation of $I(x)$ and reach the desired conclusion. We will use the notation introduced in \Cref{subsec::notationMonomial}. Let $p_1$ and $p_2$ be the two left-maximal paths ending at $x$. Note that if there was no arrow $a$ such that $ap_1\notin I$ and  $ap_2\notin I$ then $\top \Omega(I(x))$ would be 1-dimensional. Here we used that $A$ is 1-Gorenstein thanks to \Cref{lem::ARwelldefbij_Implies_1Gor}. Thus, there exists an $a$ with this property. 
    Then we are exactly in the case discussed in the proof of Lemma \ref{pdim1Lemma} (ii). 
    We already described $d_0$ and $\ker(d_0)$ there, so keeping that notation (see Figure \ref{fig:d_0}) we can describe the projective presentation explicitly:
    $$P(s(p))\oplus P(t(a))\xrightarrow{d_1}P(s_1)\oplus P(s_2)\xrightarrow{d_0}I(x)\rightarrow 0$$
    with $$d_0: (q_1,q_2)\mapsto p_1'\cdot q_1+p_2'\cdot q_2\text{ and } d_1: (q_1,q_2)\mapsto (r_{p_1,p_2}q_1-p_1a q_2,-r_{p_2,p_1}q_1).$$

    It remains to describe $\ker(d_1).$ For this, we also illustrate $d_1$ in \Cref{fig:projPresOfI(x)}. We have that $pz_1$ and $pz_2$ (resp. $z_1^-$ and $z_2^-$) are two different paths starting at $s(p)$ (resp. at $t(a)$). By \Cref{lem:dimtopI<=2}, 
    $\dim (\soc P(s(p)))=\dim (\soc P(t(a)))=2.$ Let $z$ be the path such that $z_i=zb_i\tilde{z_i}$ for $i\in\{1,2\}$ where $b_1\neq b_2$ are arrows and $\tilde{z_i}$ some paths. In particular, $p_1zb_2,p_2zb_1\in I.$ Then 
    $$\ker(d_1) = \langle (pzb_1, -z^-b_1),(0,z^-b_2) \rangle.$$
    The above elements are linearly independent and calculation yields that they are indeed in the kernel: $$d_1(pzb_1, -z^-b_1)=(r_{p_1,p_2}pzb_1-p_1a z^-b_1,-r_{p_2,p_1}pzb_1)=(0,0)$$ and $$d_1(0,z^-b_2)=(-r_{p_1,p_2}a z^-b_2,0)=(-p_1zb_2,0)=(0,0).$$ 
    
    Moreover, if $d_1(q_1,q_2)=0$ for some $(q_1,q_2)$ then by the explicit description of $d_1$, $r_{p_2,p_1}q_1=0$, which means that the path $q_1$ must contain  $pzb_1$ as a left subpath or $q_1=0$. This is a consequence of $\dim (\soc P(s(p)))=2$ and
    $pz_1$ and $pz_2$ being two different paths starting at $s(p).$ If $q_1=0,$ then $p_1aq_2=0$ and by similar reasons, $q$ must contain $z^-b_2$ as a left-subpath. With this, we can conclude that these two elements indeed generate the kernel. Hence, $S_{t(b_1)}\oplus S_{t(b_2)}\cong \top \ker(d_1),$ which is as a consequence, not simple. 
    
    \vspace{10 pt}
    {\centering
\begin{figure}[h]
    \centering

    \resizebox{0.9\linewidth}{!}{
    \input{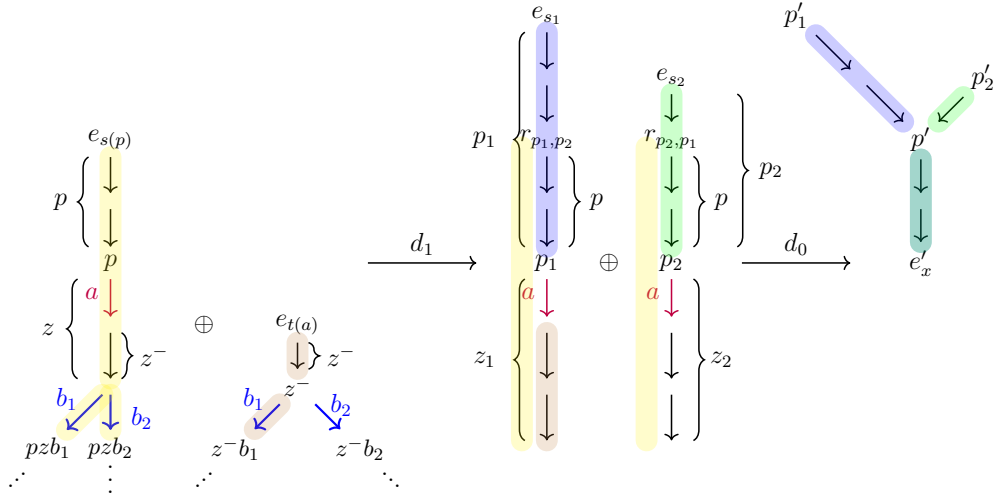}
    }
    
    \caption{Minimal projective presentation of $I(x)$}
    \label{fig:projPresOfI(x)}
\end{figure}

}

\end{proof}

\begin{lemma}
\label{lem::cokerOfMonoInProjRes}
    Let $A=kQ/I$ be a 1-Gorenstein monomial algebra, $x$ and $y$ two vertices of $Q$ and 
    $$\iota: P(x) \hookrightarrow P(y)$$ a monomorphism such that $\iota(P(x))\subseteq \rad P(y)$. Then there is a unique arrow $\alpha$ such that $\alpha A\cong P(x).$ Moreover, if there is no other arrow in $Q$ starting at $s(\alpha),$ then $S_{s(\alpha)}$ is a direct summand of $\soc \coker(\iota)$.
\end{lemma}

\begin{proof}
As $\iota: e_xA\rightarrow e_yA$ is a monomorphism, $0\neq \iota(e_xz)=\iota(e_x)\cdot z\in e_yA $ for every non-zero path $z$ starting at $x.$ Thus, $\iota(e_x)=\Sigma_p c_p p$ where the sum runs over all paths $p$ with $s(p)=y, t(p)=x$ and $pz\notin I$ for every $z\in C_{\max},$ where $C_{\max}$ denotes the set of right-maximal paths starting at $x.$
The coefficients $c_p$ are in $k.$ Since $\iota(P(x))\subseteq \rad P(y)$, every such $p$ has a positive length. As $A$ is 1-Gorenstein, $I(t(z))$ is projective-injective for every $z\in C_{\max}$. This implies that for every such $z$ there is at most one arrow $\alpha$ with $\alpha z\notin I.$ As $\iota\neq 0,$ there exists such an $\alpha$ and is the same arrow for all the different right-maximal paths $z$. This means $\alpha A\cong P(x)$. Thus, we have $p=p^+\alpha$ for every $p$ appearing in the sum above with a non-zero coefficient. 

So $\iota$ corresponds to left-multiplication by $\iota(e_x)=\Sigma_p c_p p,$ and $\Sigma_p c_p p^+$ is a non-zero element in $\coker(\iota).$ Since $(\Sigma_p c_p p^+)\cdot \alpha=\iota(e_x) \in \im(\iota)$ and $(\Sigma_p c_p p^+)\cdot \beta =0$ for every arrow $\beta\neq \alpha$ which follows from the assumption that $\alpha$ is the only arrow starting at $s(\alpha)=t(p^+)$, $\Sigma_p c_p p^+\in \soc \coker(\iota)$. Thus, $S_{s(\alpha)}$ is indeed a direct summand of $\soc \coker(\iota).$

\end{proof}

\subsection{The main result}

\begin{theorem}\label{thm::AGiffARbijwelldef_Monomial}
    For a monomial algebra $A$ the following statements are equivalent:
    \begin{enumerate}[\rm (i)]
        \item $A$ is Auslander-Gorenstein.
        \item $\psi_A$ is a well-defined bijection for $A.$ 
    \end{enumerate}
\end{theorem}

\begin{proof}
By Corollary \ref{cor::Q2for2Gor_monomial}, it only remains to show that if $\psi_A$ is a well-defined bijection for $A$ then $A$ is $2$-Gorenstein. Equivalently, we prove that $A^{op}$ is 2-Gorenstein. So let us assume that (ii) holds. By \Cref{lem::ARwelldefbij_Implies_1Gor}, $A$, and thus $A^{op},$ are 1-Gorenstein. So we need to show that if
$$P_1\rightarrow P_0 \rightarrow I(x) \rightarrow 0$$ is a minimal projective presentation of an indecomposable injective module $I(x)$, then $\idim P_1\le 1.$ 
\Cref{lem:dimtopI<=2} tells us that one of the following two cases must occur. 

\textbf{Case 1:} $I(x)$ is uniserial. 

Let $i_x$ denote the unique left-maximal path ending at $x,$ and set $s:=s(i_x).$ As $A^{op}$ is 1-Gorenstein, $P_0\cong P(s) \cong M(w)$ where $w$ is the unique right-maximal path starting at $s$. The path $w$ is also left-maximal. If $w=i_x$ then $P_1\cong 0$ and hence, $\idim (P_1)\le 1$ holds. So assume that $w=i_x\alpha z$ for some arrow $\alpha$ and a path $z.$ Denote $t_{\alpha}=t(\alpha)$. Then 
\begin{equation}\label{eq::P1Case1}
    P_1\cong P(t_{\alpha}).
\end{equation}

 \vspace{5 pt}
    {\centering

    \resizebox{0.4\linewidth}{!}{
    \input{pictures/proofOfMainThmCase1}
}

}

\textbf{Case 2:} there are exactly two different left-maximal paths ending at $x$, $p_1$ and $p_2.$

We can write $p_j=r_jp$ for some paths $p, r_1,$ and $r_2$, such that $r_1$ and $r_2$ have non-zero length and end with different arrows. As $A^{op}$ is 1-Gorenstein, $P_0\cong P(s_1)\oplus P(s_2)$ where $s_j=s(p_j)$ and $P(s_j)\cong M(w_j)$  for the unique  right-( and left-)maximal path $w_j=p_jz_j$ starting at $s_j.$ If $z_1$ and $z_2$ start with different arrows, then $$P_1\cong P(s(p)).$$ Otherwise, there exists an arrow $\alpha$, such that both $z_1$ and $z_2$ start with $\alpha.$ Let $t_{\alpha}=t(\alpha)$. Then 
\begin{equation}\label{eq::P1Case2}
    P_1\cong P(s(p))\oplus P(t_{\alpha}).
\end{equation}
This is exactly the situation illustrated in \Cref{fig:projPresOfI(x)}. Note that in Case 2 $s(p)$ satisfies the assumptions from \Cref{lem::technical_pdim1}, thus, $\idim P(s(p))= 1.$ So we should focus on  $P(t_{\alpha})$ from \Cref{eq::P1Case1} and \Cref{eq::P1Case2}.
Let $$0\rightarrow R_{d} \xrightarrow{p_d} R_{d-1} \rightarrow \ldots \rightarrow R_0 \rightarrow I(v) \rightarrow 0$$ be a minimal projective resolution of an indecomposable injective module $I(v)$.
\\

\textbf{Claim:} if $d\ge 2$ and every direct summand of $\Omega^{d-1}(I(v))$ has a simple top, then  $P(t_{\alpha})\ncong R_d.$
\\
\\
\textit{Proof of the Claim:} 
 By assumption (ii), $R_d$ is indecomposable, so $\Omega^{d-1}(I(v))$ has a unique non-projective direct summand, $C.$ Let every direct summand of $\Omega^{d-1}(I(v))$ have a simple top. So in particular, there exists a vertex $c$ such that $S_c\cong \top C.$ So we have the following short exact sequence:

$$0\longrightarrow R_d \xrightarrow{p_d=(\iota,0)}P(c)\oplus P'\cong R_{d-1}\xrightarrow{(\pi,id)} C\oplus P' \cong \Omega^{d-1}(I(v)) \longrightarrow 0,$$ where $P'$ is projective and $\pi: P(c)\to C$ is the projective cover of $C.$  For the sake of contradiction, assume that $P(t(\alpha))\cong R_d.$ In both Case 1 and Case 2, this leads to the following: 
\begin{equation}
\label{eq::visdirectsummandofsocle}
    S_x \text{ is a direct summand of }\soc C.
\end{equation}

To see this, we would like to use \Cref{lem::cokerOfMonoInProjRes}. Indeed, as we are talking about a minimal projective resolution, $\iota(P(t_{\alpha}))\subseteq \rad P(c)$  holds and the lemma applies. Hence, there is a unique arrow $a$ such that $P(t_{\alpha})\cong aA$, and if there is no other arrow starting at $s(a)$, $S_{s(a)}$ is a direct summand of $\soc C.$ 
In Case 1, $0\neq z\in P(t_{\alpha})$ and $w=i_x\alpha z$ is the unique left-maximal path ending at $t(z).$ In Case 2, $0\neq z_1^-\in P(t_{\alpha})$, and $w_1=p_1z_1=p_1\alpha z_1^-$ is the unique left-maximal path ending at $t(z_1^-).$ So in both cases, $a=\alpha$ and $s(a)=x.$ Since $I(x)$ was uniserial in Case 1, the same is true for $P(x)$ according to \Cref{lem::1GorMonomial: dimsocP(x)=dimtopI(x)}. Thus, $\alpha$ is the only arrow starting at $x.$ For Case 2 recall that $\dim (\soc P(t_{\alpha}))\le 2$ by \Cref{lem:dimtopI<=2}. So as $z_1$ and $z_2$ are already two different right-maximal paths starting at $s(\alpha)$, $\alpha$ is the only arrow starting at $x.$ Hence, (\ref{eq::visdirectsummandofsocle}) holds in both cases. 

This leads to our contradiction. Since $C$ is a submodule of the projective $R_{d-2}$, every direct summand of $\soc C$ is a direct summand of the socle of an indecomposable projective module as well. So by (\ref{eq::visdirectsummandofsocle}), this is true in particular for $S_x$. By the 1-Gorenstein property of $A$, this implies that $I(x)$ is projective-injective. This, however, contradicts our assumptions both in Case 1 and in Case 2. In conclusion, $P(t_{\alpha})\ncong R_d.$

We finish our proof by investigating which vertices $v$ satisfy the assumption of the Claim. An immediate observation is that this is true if $d\ge 3$, by \Cref{lem::syzygiesInMonomialAlg}. Moreover, this is also true if $d=2$ and $I(v)\cong M(i_v)$ is uniserial. Indeed, we get from our previous discussion in Case 1 that $\Omega(I(v))\cong M(z),$ where $i_v\beta z$ is the unique right-maximal path starting at $s(i_v)$ and $\beta$ is an arrow. In particular, $\Omega(I(v))$ has a simple top. Finally, Lemma \ref{lem::technical_SecondSyzNotIndec} implies that if $I(v)$ is not uniserial and the top of $\Omega(I(v))$ is not simple, then $\top \Omega^2(I(v))$ is not simple either, which means that $R_2$ is not indecomposable. As we assumed (ii), $\psi_A$ is well-defined and $d>2$ in this case. So the assumptions of the Claim also hold in the case $d=2.$ 

As a consequence, $\psi_A(I(v))\neq P_{t_{\alpha}}$ for any $v$ with $\pdim I(v)\ge 2.$ Due to assumption (ii), the only possibility is that there exists a vertex $v$ with $\pdim I(v)=1$ and $\Omega(I(v))\cong P(t_{\alpha}),$ or a vertex $v$ with $I(v)\cong P(t_{\alpha}).$ Since $A$ is 1-Gorenstein, \Cref{thm::preliminaries::ARbij} yields that $\idim P(t_{\alpha}) = \idim \psi_A(I(v)) = \pdim I(v)\le 1.$

In conclusion, $\pdim (I^1)\le 1$ and $A$ is 2-Gorenstein.

\end{proof}

\subsection{A stronger version of the Auslander-Reiten Conjecture}
\label{subsec::strongAuslanderReitenConj}

The Auslander-Reiten Conjecture states that if an Artin algebra $A$ is $n$-Gorenstein for all $n\ge 1$, then $A$ is necessarily Iwanaga-Gorenstein. As a corollary of \Cref{thm::cutting}, we obtain that a strengthening of this statement holds for every monomial algebra. 

\begin{corollary}
\label{cor::strongerVersionofARConj}
    Let $A=kQ/I$ be a monomial algebra with $n$ simple modules, where $n\ge 1$. Assume that $A$ is $(4n-2)$-Gorenstein. Then $A$ is $(4n-2)$-Iwanaga-Gorenstein.  
\end{corollary}

\begin{proof} Let $A=kQ/I$ be a $(4n-2)$-Gorenstein monomial algebra. Since $n\ge1$, we obtain from \Cref{cor::2kimplies2k+1_monomial} that $A$ is also $(4n-1)$-Gorenstein. Then we can apply \Cref{thm::cutting} to $A$ repeatedly to obtain a Nakayama algebra $B$ which is, first of all, $(4n-1)$-Gorenstein, and second, $B$ is $(4n-2)$-Iwanaga-Gorenstein if and only if $A$ is. Note that we can easily find an upper bound for the number of simple modules of $B.$ Since $Q$ has $n$ vertices, we applied the cutting procedure to $A$ at most $n$ times, once for every degree 4 vertex. Every cutting operation creates one extra vertex, so the resulting Nakayama algebra $B$ has at most $2n$ vertices.

In the following, we assume for the sake of contradiction that $B$ is not $(4n-2)$-Iwanaga-Gorenstein, i.e. that $\idim {}_BB=\pdim D({}_BB)\ge 4n-1$ or $\idim B_B\ge 4n-1.$ Let us start with the case $\pdim D({}_BB)\ge 4n-1.$ This assumption implies that there exists a simple $B$-module $S$ whose injective envelope $I(S)$ has projective dimension at least $4n-1.$ As $B$ is $(4n-1)$-Gorenstein, $I(S)$ cannot appear as a direct summand of any of the first $4n-2$ terms $I^i$ in a minimal injective resolution of $B_B$:
$$0\longrightarrow B_B \longrightarrow I^0 \longrightarrow I^1 \longrightarrow \ldots .$$

This means that $\Ext_B^i(S,B)=0$ for $i=0,1,\ldots,4n-2$. 
Therefore, if we apply the functor $(-)^*=Hom_B(-,B)$ to a minimal projective resolution of $S$
$$\ldots \longrightarrow P_1 \longrightarrow P_0 \longrightarrow S \longrightarrow 0,$$
we obtain the following exact sequence of projective $B^{op}$-modules: $$0\longrightarrow P_0^* \longrightarrow P_1^* \longrightarrow \ldots \longrightarrow P_{4n-2}^* \xrightarrow{\ g \ } P_{4n-1}^*.$$

Then the module $M=\coker(g)$ has projective dimension at least $4n-1$ over $B^{op}$. 
However, this is impossible, since $B^{op}$ is a Nakayama algebra and the finitistic dimension of a Nakayama algebra with $N$ simple modules is at most $2N-2$, see for example \cite[Corollary 3.3]{ChY14}. Since $B^{op}$ has at most $2n$ simple modules, $\pdim M\le 4n-2$ should hold, which is a contradiction. In conclusion, $\idim {}_BB\ge 4n-1$ is not possible. As being $(4n-1)$-Gorenstein is left-right symmetric, we can use an analogous argument for the opposite algebra of $B$ to show that $\idim B_B< 4n-1$ must hold as well. Thus, $B$ is $(4n-2)$-Iwanaga-Gorenstein, and so is $A.$

\end{proof}

\section*{Acknowledgement}

The author is very grateful to Ren\'e Marczinzik for the insightful discussions and the numerous interesting questions suggested by him about Auslander-Gorenstein algebras. 
The author also thanks Jacob Fjeld Grevstad for the helpful discussions regarding \Cref{cor::2-Gor-class-Nak}.

The author was supported by the Deutsche Forschungsgemeinschaft (DFG, German Research Foundation) under Germany's Excellence Strategy grant EXC-2047/1-390685813.

\bibliographystyle{alpha}
\bibliography{refs}

\end{document}